\documentclass[preprint]{imsart}
\usepackage{amsmath}
\usepackage{amsthm}
\usepackage{amssymb}
\usepackage{footmisc}
\usepackage{natbib}
\usepackage[colorlinks=true,linkcolor=blue,citecolor=blue]{hyperref}
\usepackage{verbatim}
\usepackage{textcomp}
\usepackage{enumerate}
\usepackage{graphicx}
\usepackage{array}
\usepackage{mathtools}
\usepackage{bbm}
\newcolumntype{H}{>{\setbox0=\hbox\bgroup}c<{\egroup}@{}}
\RequirePackage[OT1]{fontenc}
\usepackage{multirow}
\usepackage{fancyhdr}

\usepackage{xr}
\newcommand{\R}{{\mathbb R}}
\newcommand{\E}{{\mathbb E}}
\newcommand{\V}{{\mathbb V}}
\newcommand{\VC}{{\mathbb{V}}}
\renewcommand{\P}{{\mathbb P}}
\newcommand{\N}{{\mathbb N}}

\newcommand{\eps}{\varepsilon}

\newcommand{\M}{\mathbb M}

\newcommand{\lmax}{\lambda_{\max}}
\newcommand{\lmin}{\lambda_{\min}}

\DeclareMathOperator{\trace}{trace}

\DeclareMathOperator{\diag}{diag}
\DeclareMathOperator{\rank}{rank}

\DeclareMathOperator{\s}{span}
\DeclareMathOperator{\corr}{corr}

\newtheorem{theorem}{Theorem}[section]
\newtheorem{proposition}[theorem]{Proposition}
\newtheorem{lemma}[theorem]{Lemma}

\newtheorem{corollary}[theorem]{Corollary}
\newtheorem{remark}[theorem]{Remark}
\newtheorem*{remark*}{Remark}

\newtheorem{condition}{Condition}
\newtheorem{cond}{Condition}
\newtheorem*{condition*}{Condition}
\newtheorem*{definition*}{Definition}

\numberwithin{equation}{section}

\newcounter{rcnt}[section]
\renewcommand{\thercnt}{(\roman{rcnt})}
\newcommand{\rem}[1]{}

\newcounter{desccount}

\newcommand{\descref}[1]{\hyperref[#1]{#1}}

\arxiv{arXiv:1611.01043}

\begin{document}

\begin{frontmatter}

\title{Uniformly valid confidence intervals post-model-selection}

\runauthor{Bachoc, Preinerstorfer, Steinberger}

\runtitle{}

\begin{aug}
\author{\fnms{Fran\c{c}ois} \snm{Bachoc}\ead[label=a1]{francois.bachoc@math.univ-toulouse.fr}},
\author{\fnms{David} \snm{Preinerstorfer}\thanksref{t2}\ead[label=a2]{david.preinerstorfer@ulb.ac.be}}
\\\and
\author{\fnms{Lukas} \snm{Steinberger}\thanksref{t3}\ead[label=a3]{lukas.steinberger@stochastik.uni-freiburg.de}}

\thankstext{t2}{David Preinerstorfer was supported by the Austrian Science Fund (FWF): P27398 and by the Danish National Research Foundation (grant DNRF 78, CREATES).}
\thankstext{t3}{Lukas Steinberger was supported by the Austrian Science Fund (FWF): P28233 and by the German Research Foundation (DFG): RO 3766/401.}

\affiliation{Universit\'e Paul Sabatier, Universit\'e libre de Bruxelles \\and University of Freiburg}

\address{
	Fran\c{c}ois Bachoc\\
	Institut de Math\'ematiques de Toulouse\\
	Universit\'e Paul Sabatier\\
	118 route de Narbonne\\
	31062 Toulouse Cedex 9, France\\
	\printead{a1}
	}
\address{	
	David Preinerstorfer\\
	ECARES \\
	Universit\'e libre de Bruxelles \\
	50 Ave. F.D. Roosevelt \\
	CP 114/04\\
	B1050 - Bruxelles\\
	\printead{a2}
	}
\address{	
	Lukas Steinberger\\
	Department of Mathematical Stochastics \\
University of Freiburg \\
Eckerstra{\ss}e 1 \\
79104 Freiburg im Breisgau, Germany \\
	\printead{a3}
	}
	
\end{aug}

\bigskip

\begin{abstract}
We suggest general methods to construct asymptotically uniformly valid confidence intervals post-model-selection. The constructions are based on principles recently proposed by \cite{Berk13}. In particular the candidate models used can be misspecified, the target of inference is model-specific, and coverage is guaranteed for \textit{any} data-driven model selection procedure. After developing a general theory we apply our methods to practically important situations where the candidate set of models, from which a working model is selected, consists of fixed design homoskedastic or heteroskedastic linear models, or of binary regression models with general link functions.
In an extensive simulation study, we find that the proposed confidence intervals perform remarkably well, even when compared to existing methods that are tailored only for specific model selection procedures.
\end{abstract}



\end{frontmatter}

\section{Introduction}

Fitting a statistical model to data is often preceded by a model selection step, and practically always has to face the possibility that the candidate set of models from which a model is selected does not contain the true distribution. The construction of valid statistical procedures in such situations is quite challenging, even if the candidate set of models does contain the true distribution (cf. \cite{leeb05model, leeb2006, leeb08model}, \cite{kabaila06large} and \cite{Poe09a}, and the references given in that literature), and has recently attained a considerable amount of attention. In a Gaussian homoskedastic location model and fitting possibly misspecified linear candidate models to data, \cite{Berk13} have shown how one can obtain valid confidence intervals post-model-selection for (non-standard) model-dependent targets of inference in finite samples (cf. also the discussion in \cite{leeb13various}, and related results obtained for prediction post-model-selection in \cite{bachoc14valid}). In this setup, their approach leads to valid confidence intervals post-model-selection \emph{regardless} of the specific model selection procedure applied. This aspect is of fundamental importance, because many model selection procedures used in practice are almost impossible to formalize: researchers typically use combinations of visual inspection and numerical algorithms, and sometimes they simply select models that let them reject many hypotheses, i.e., they are hunting for significance. 
These often unreported and informal practices of model selection prior to conducting the actual analysis may also play a key role in the current crisis of reproducibility. Thus, to establish and popularize statistical methods that are in some sense robust to `bad practice' is highly desirable.

The methods discussed in \cite{Berk13} are based on the assumption that the true distribution is Gaussian and homoskedastic, and the authors consider only situations where linear models are fit to data. It is of substantial interest to generalize this approach, and to obtain generic methods for constructing confidence intervals post-model-selection that are widely applicable beyond the Gaussian homoskedastic model considered in \cite{Berk13}. We develop a general asymptotic theory for the construction of uniformly valid confidence sets post-model-selection. These results are applicable whenever the estimation error can be expanded as the sum of independent centered random vectors and a remainder term that is negligible relative to the variance of the leading term. Such a representation typically follows from standard first order linearization arguments, and can therefore be obtained in many situations. 

Our confidence intervals can be based on either consistent estimators of the variance of the previously mentioned sum, or, more importantly, if such estimators are not available (which is usually the case when all working models are misspecified), can be based on variance estimators that consistently \emph{overestimate} their targets. We also present results that allow one to obtain such estimators in general and demonstrate their construction in specific applications, where they often coincide with well known sandwich-type estimators. This overcomes another limitation present in \cite{Berk13}, namely the assumption that there exists an unbiased (and chi-square distributed) or uniformly consistent estimator of the variance of the observations (cf. the discussion in Remark~2.1 of \cite{leeb13various} and in Appendix~A of \cite{bachoc14valid}). The usage of variance estimators that overestimate their targets, while leading to more conservative inference, renders the approach applicable to the fully misspecified setting. Moreover, the suggested conservative estimators usually have the property that their bias vanishes if the selected model is correct (cf. Remark~\ref{rem:adaptive} and Subsection~\ref{sec:homlinmods:varest}).

Another important aspect of the results obtained is that they are valid uniformly over wide classes of potential underlying distributions, which is particularly important as this guarantees that the results provide a better description of finite sample properties than `pointwise' asymptotic results (cf. \cite{leeb03finite}, \cite{leeb05model} and \cite{Tib15a} for a discussion of related issues in a model selection context).

Moreover, we apply our general theory to three important modeling situations: First, we consider the case where linear homoskedastic models are fitted to non-Gaussian homoskedastic data. This provides an extension of the results of \cite{Berk13} to the non-Gaussian case, without requiring a consistent variance estimator. Next, we study the problem of fitting heteroskedastic linear models to non-Gaussian heteroskedastic data. This scenario necessitates a more careful choice of variance estimators and leads to an extension of the influential results of \cite{eicker67} to the misspecified post-model-selection context. Our third application then considers the problem of fitting binary regression models to binary data. In this case, also the link function may be chosen in a data driven way. On a technical level, the third example is quite different from the previous ones, because here non-trivial existence and uniqueness questions concerning the targets of inference and the (quasi-)maximum likelihood estimators have to be addressed. 

Our confidence intervals obtained in these specific situations are particularly convenient for practitioners, because they are structurally very similar to the confidence sets one would use in practice following the naive (and invalid \citep[see, e.g.,][]{leeb13various, bachoc14valid}) approach that ignores that the model has been selected using the same data set. The main difference of our construction to the naive (and invalid) approach is the choice of a critical value: Quantiles from a standard normal or $t$-distribution are replaced by so-called POSI-constants (cf. \cite{Berk13} and Section~2.5 below). Thus, the procedures are conceptually simple and easy to implement. Moreover, we provide mild and easily verifiable regularity conditions on observable quantities (e.g., the design or the link functions) under minimal restrictions on the unknown data generating process.

Finally, in a series of numerical examples, we illustrate that the proposed confidence intervals are valid also in small samples while their lengths appear to be practically reasonable when compared to naive (and invalid) procedures. Furthermore, we compare our methods to those of \cite{Tib15a} and \cite{taylor17post}, and find that our intervals are often shorter than their competitors, even when we study the exact same scenarios for which those competing methods were tailored for and even though our confidence intervals offer much stronger theoretical guarantees.

The structure of the present article is as follows: We first develop a general asymptotic theory for the construction of uniformly valid confidence sets post-model-selection in Section \ref{sec:genposi}. In Section \ref{sec:applications}, we apply our theoretical results to the three previously mentioned modeling scenarios. Of course, the selection of examples in Section~\ref{sec:applications} is by no means exhaustive. But besides covering three very important modeling frameworks, Section~\ref{sec:applications} serves as an illustration of how the general theory developed in Section~\ref{sec:genposi} can be applied. An outline of the numerical results is presented in Section~\ref{sec:sim}. In Section~\ref{sec:concl} we conclude and discuss possible extensions of the results obtained in this paper that are currently under investigation. Details of the simulations as well as all the proofs are collected in Sections~\ref{app:sim}, \ref{app:aux}, \ref{app:posi} and \ref{app:appl} of the appendix.

\subsection{Related work}

The present article is devised in the spirit of \cite{Berk13}, in the sense that we aim at inference post-model-selection that is valid irrespective of the employed model selection procedure. Very recently, \cite{Rinaldo16} have investigated a classical sample spitting procedure that is also independent of the underlying selection method. However, they consider only the i.i.d. case, thereby excluding, for instance, fixed design regression. Several other authors have proposed inference procedures post-model-selection that are tailored towards specific selection methods and for specific modeling situations. In the context of fitting linear regression models to Gaussian data, methods that provide valid confidence sets post-model-selection, and that are constructed for specific model selection procedures (e.g., forward stepwise, least-angle-regression or the lasso) and for targets of inference similar to those considered in the present article, have been recently obtained by \cite{tibshirani2014exact}, \cite{LeeTay14}, \cite{Fit15a} and \cite{lee15exact}. \cite{Tib15a} extended the approach of \cite{tibshirani2014exact} to non-Gaussian data by obtaining uniform asymptotic results. Furthermore, valid inference post-model-selection on conventional regression parameters under sparsity conditions was considered, among others, by \citet{Bel11a, Bel14a, van14a} and \citet{Zha14a}.

\newpage

\section{Inference post-model-selection: A general asymptotic theory}
\label{sec:genposi}

\subsection{Framework, problem description, and approach}\label{sec:fram}

Consider a situation where we observe a data set $y\in\R^{n\times\ell}$ that is a realization of an unknown probability distribution $\P_n$ on the Borel sets of the sample space $\R^{n \times \ell}$. We denote the $i$-th row of the data vector (matrix) $y$ by $y_i\in\R^{1\times \ell}$, so that $y=(y_1',\dots, y_n')'$, and write $\P_{i,n}$ for the marginal distribution corresponding to that row. Throughout, we assume that the data generating distribution is of product form, that is $\P_n = \bigotimes_{i=1}^n \P_{i,n}$. Suppose further that one wants to conduct inference on $\mathbb{P}_n$, and intends to use as a working model an element of $\mathsf{M}_n$, a set consisting of $d$ nonempty sets of distributions $\mathbb{M}_{1,n}, \hdots,\mathbb{M}_{d,n}$ on the Borel sets of $\R^{n \times \ell}$. Throughout $d$ is fixed, i.e., does not depend on $n$. We emphasize that it is \textit{not} assumed that $\mathbb{P}_n$ is contained in one of the sets $\mathbb{M}_{j,n}$ for $j = 1, \hdots, d$. That is, the candidate set $\mathsf{M}_n$ might be \textit{misspecified}. 

For each model $\mathbb{M} \in \mathsf{M}_n$ one has to define a corresponding target of inference $\theta^*_{\mathbb{M},n} = \theta^*_{\mathbb{M},n}(\mathbb{P}_n)$, say, which we take as given throughout the present section. Furthermore we assume that for every $\mathbb{M}_{j,n} \in \mathsf{M}_n$ the target is an element of a Euclidean space of finite dimension $m(\mathbb{M}_{j,n})$ which does not depend on $n$. As an example in the case $\ell=1$, consider the situation where $\P_n$ has mean vector $\mu_n\in\R^n$ and $\mathbb{M} \in \mathsf{M}_n$ is given by the collection of all $n$-dimensional normal distributions with covariance matrix proportional to identity and mean $X_\M\beta$, for different values of $\beta\in\R^{m(\M)}$, and where $X_\M$ is an $n\times m(\M)$ matrix obtained by selecting certain columns from a given fixed design matrix $X\in\R^{n\times p}$. In this setting, \cite{Berk13} consider the target $\theta_{\M,n}^*(\P_n) = (X_\M'X_\M)^{-1}X_\M'\mu_n$ (cf. also Section \ref{sec:applications} for more on this and further examples). In general, $\theta^*_{\mathbb{M},n}$ will typically be the value of the parameter that corresponds to the projection of $\mathbb{P}_n$ onto $\mathbb{M}$ w.r.t. some measure of closeness, e.g., the Kullback-Leibler divergence, or the Hellinger-distance. Note that in general such a projection might not uniquely exist, or might not exist at all, and that in each application additional conditions -- on $\mathbb{P}_n$ and/or the candidate set $\mathsf{M}_n$ of models -- need to be imposed to obtain well defined targets. Note also that the target is model-specific, i.e., it depends on $\mathbb{M}$. Lastly we emphasize that defining and working with (pseudo) targets of inference in potentially misspecified models has a long-standing tradition in statistics, dating back at least to \cite{Huber67}, and we confer the reader to this strand of literature for further discussion.

Given data $y$ the statistician now has two problems to solve: (i) model selection, i.e., the statistician needs to choose an ``appropriate'' working model from the candidate set $\mathsf{M}_n$; and (ii) statistical inference post-model-selection, i.e., given the selected model, the statistician typically wants to conduct inference on the targets in this model. Note that such targets are  random, as they depend on the data via the model selection procedure used. We do not contribute anything new to how models can be selected from data. We take a model selection procedure as given, and denote the model selection procedure used by $\hat{\mathbb{M}}_n: \R^{n \times \ell} \to \mathsf{M}_n$ (measurable). That is, the quantity $\hat{\mathbb{M}}_n(y)$ denotes the selected model upon observing $y$. We also assume that for every model $\mathbb{M} \in \mathsf{M}_n$ an estimator $\hat{\theta}_{\mathbb{M},n}: \R^{n \times \ell} \to \R^{m(\mathbb{M})}$ (measurable) of the corresponding target $\theta^*_{\mathbb{M},n}$ is available. Summarizing, the statistician selects the model using $\hat{\mathbb{M}}_n$, and estimates $\theta^*_{\hat{\mathbb{M}}_n,n}$ using $\hat{\theta}_{\hat{\mathbb{M}}_n, n}$. In this article we address the question how valid confidence intervals can be constructed for the coordinates of the target $\theta^*_{\hat{\mathbb{M}}_n,n}$. Our approach is as follows:
\begin{enumerate}
\item Given $\alpha \in (0, 1)$, we construct confidence intervals $\mathrm{\mathrm{CI}}^{(j)}_{1-\alpha, \mathbb{M}}$ for the $j$-th component $\theta_{\mathbb{M}, n}^{*(j)}$ of $\theta_{\mathbb{M}, n}^{*}$, for every $j = 1, \hdots, m(\mathbb{M})$ and every $\mathbb{M} \in \mathsf{M}_n$ so that
\begin{equation*}
\liminf_{n \to \infty} \mathbb{P}_n\left( \theta^{*(j)}_{\mathbb{M}, n} \in \mathrm{CI}^{(j)}_{1-\alpha, \mathbb{M}} \text{ for all } j = 1, \hdots, m(\mathbb{M}) \text{ and all } \mathbb{M} \in \mathsf{M}_n \right)
\end{equation*}
is not smaller than $1-\alpha$.
\item For a model selection procedure $\hat{\M}_n$, our suggested confidence intervals are then obtained via 
\begin{equation*}
\mathrm{\mathrm{CI}}^{(j)}_{1-\alpha, \hat{\mathbb{M}}_n} \text{ for } j = 1, \hdots, m(\hat{\mathbb{M}}_n).
\end{equation*}
From the coverage property in Part~1 we obtain
\begin{equation*}
\liminf_{n \to \infty} \mathbb{P}_n\left( \theta^{*(j)}_{\hat{\mathbb{M}}_n, n} \in \mathrm{CI}^{(j)}_{1-\alpha, \hat{\mathbb{M}}_n} \text{ for all } j = 1, \hdots, m(\hat{\mathbb{M}}_n) \right) \geq 1-\alpha.
\end{equation*}
\end{enumerate}

As already discussed in the introduction, the fact that our approach does not restrict the model selection procedure used is important. It is precisely this aspect that allows practitioners to obtain valid confidence intervals post-model-selection in situations where a wide variety of (formal or informal) mechanisms have been incorporated to select the model. 


\subsection{Discussion} \label{ssc:discfram}
The above framework is certainly somewhat abstract, but its generality is necessary to achieve the scope of the present paper, which is the development of results for the construction of confidence intervals post-model-selection that are widely applicable. In particular, apart from allowing for a misspecified candidate set of models, the framework allows the marginals $\P_{i,n}$ for $i = 1, \hdots, n$ to be non-identical. This property is not just a mere technical aspect, but is necessary if one wants to cover situations such as fixed-design regression models. 

Most importantly, we work with a sequence $\P_n$ of data generating mechanisms. Again, this is not a technical nuisance. Rather, this aspect ensures that the results obtained can be used to construct uniformly valid confidence intervals post-model-selection. For specific applications we refer to Section~\ref{sec:applications}, but the approach is conceptually simple, generally applicable and extends substantially beyond our examples. We shall give a brief outline of the underlying idea subsequently, also to convince the reader that working with sequences of data generating mechanisms is worth the effort. Suppose $\mathbb{P}_n$, the distribution that generated the data $y$, is known to be an element of a set $\mathbf{P}_n$. The set $\mathbf{P}_n$ describes the assumptions one is willing to impose on the unknown distribution in a particular modeling scenario, and will typically be large and potentially nonparametric. Suppose further that one wants to work with a candidate set of models $\mathsf{M}_n$ (possibly misspecified, i.e., $\mathbf{P}_n \not \subseteq \bigcup_{\M\in\mathsf{M}_n}\M$) and corresponding model specific targets $\theta_{\mathbb{M}, n}^*$ as above, and that the goal is to construct confidence sets post-model-selection. Under weak assumptions on $\mathbf{P}_n$, the general results developed in this paper allow one to construct confidence intervals so that 
\begin{equation*}
\liminf_{n \to \infty} \mathbb{P}_n\left( \theta^{*(j)}_{\hat{\mathbb{M}}_n, n} \in \mathrm{CI}^{(j)}_{1-\alpha, \hat{\mathbb{M}}_n} \text{ for all } j = 1, \hdots, m(\hat{\mathbb{M}}_n) \right) \geq 1-\alpha
\end{equation*}
holds for any (measurable) model selection procedure $\hat{\mathbb{M}}_n$, and for every sequence of distributions $\mathbb{P}_n$ that satisfies $\mathbb{P}_n \in \mathbf{P}_n$ for every $n \in \N$. Certainly, this then implies
\begin{equation*}
\liminf_{n \to \infty} \inf_{\mathbb{P}_n \in \mathbf{P}_n}  \mathbb{P}_n\left( \theta^{*(j)}_{\hat{\mathbb{M}}_n, n} \in \mathrm{CI}^{(j)}_{1-\alpha, \hat{\mathbb{M}}_n} \text{ for all } j = 1, \hdots, m(\hat{\mathbb{M}}_n) \right) \geq 1-\alpha,
\end{equation*}
i.e., asymptotic validity of the constructed confidence sets \textit{uniformly} over $\mathbf{P}_n$. That the development of results that hold uniformly over large classes of distributions is important, in particular so in the context of inference post-model-selection, is well understood (see \cite{leeb03finite} and \cite{leeb05model}). One recent article that studies uniform coverage properties post-model-selection is \cite{Tib15a}. Merits of uniform results in contrast to pointwise asymptotic results are discussed in their Section 1.1. \cite{Tib15a} consider a setup similar to the example we consider in Section~\ref{sec:homlinmods} and for specific model selectors, but compared to our results uniform validity is established only over substantially smaller sets of distributions, and they need to impose stronger conditions on the design matrices, which rule out some important cases our results allow for, e.g., polynomial trends. See also Section~\ref{sec:sim:Tib} for numerical results and comparisons. 
%

\subsection{Notation}

Before we proceed to our general theory and the corresponding basic assumption, we introduce some notation that is used throughout this article: A normal distribution with mean $\mu$ and (possibly singular) covariance matrix $\Sigma$ is denoted by $N(\mu, \Sigma)$. For $\alpha \in (0, 1)$ and a covariance matrix $\Gamma$ we denote by $K_{1-\alpha}(\Gamma)$ the $1-\alpha$-quantile of the distribution of the supremum-norm $\|Z\|_{\infty}$ of $Z \sim N(0, \Gamma)$. The correlation matrix corresponding to a covariance matrix $\Sigma$ is denoted by $\corr(\Sigma) = \diag(\Sigma)^{\dagger/2} \Sigma \diag(\Sigma)^{\dagger/2}$, where $\diag(\Sigma)$ denotes the diagonal matrix obtained from $\Sigma$ by setting all off-diagonal elements equal to $0$, $A^{\dagger}$ denotes the Moore-Penrose inverse of the quadratic matrix $A$, $A^{1/2}$ denotes the symmetric non-negative definite square root of the non-negative definite matrix $A$, and where we abbreviate $[A^{\dagger}]^{1/2}$ by $A^{\dagger/2}$. The smallest and largest eigenvalue of a real symmetric matrix $A$ is denoted by $\lmin(A)$ and $\lmax(A)$, respectively. For a vector $v$ with coordinates $v^{(1)}, \hdots, v^{(l)}$ we also use the symbol $\diag(v)$ to denote the diagonal matrix with first diagonal entry $v^{(1)}$, second $v^{(2)}$, and so on. The operator norm of a matrix $A$ (w.r.t. the Euclidean norm) is denoted by $\|A\|$, and the Euclidean norm of a vector $v$ is denoted by $\|v\|$. Furthermore, $A_{ii}$, the $i$-th diagonal element of a quadratic matrix $A$, is occasionally abbreviated as $A_i$. We also identify the indicator function $\mathbbm{1}_B$ of a set $B$ with the set $B$ itself, whenever there is no risk of confusion. Weak convergence of a sequence of probability measures $\mathbb{Q}_n$ to $\mathbb{Q}$ is denoted by $\mathbb{Q}_n \Rightarrow \mathbb{Q}$. The image measure induced by a random variable (or vector) $x$ defined on a probability space $(F, \mathcal{F}, \mathbb{Q})$ is denoted by $\mathbb{Q} \circ x$. If not stated otherwise, limits are taken as $n \to \infty$. For a sequence $(a_n)_{n\in\N}$, we say that a property \emph{holds eventually} if there exists a positive integer $n_0$ such that the property holds for every $a_n$ with $n\ge n_0$. The expectation operator and the variance-covariance operator w.r.t. $\mathbb{P}_n$ is denoted by $\mathbb{E}_n$ and $\mathbb{V}_n$, respectively; and the expectation operator and the variance-covariance operator w.r.t. $\mathbb{P}_{i,n}$ is denoted by $\mathbb{E}_{i,n}$ and $\mathbb{V}_{i,n}$, respectively.

\subsection{Main assumption}\label{sec:ass}

Our methods for constructing uniformly valid confidence intervals post-model-selection are developed under a high-level condition imposed on the stacked vector of estimators $\hat{\theta}_n = (\hat{\theta}_{\mathbb{M}_1, n}', \hdots, \hat{\theta}_{\mathbb{M}_d, n}')'$ centered at the corresponding stacked vector of targets $\theta_n^* = (\theta^{*'}_{\mathbb{M}_1, n}, \hdots, \theta^{*'}_{\mathbb{M}_d, n})'$. In this section we denote the dimension of $\hat{\theta}_n$ by 
$$k \quad := \quad \sum_{j = 1}^d m(\mathbb{M}_{j,n}),$$ 
which does not depend on $n$. The condition is as follows:
\begin{condition}\label{cond:sum}
There exist Borel measurable functions $g_{i,n}: \R^{1 \times \ell}\to \R^k$ for $i = 1, \hdots, n$, and $\Delta_n:\R^{n\times \ell}\to\R^k$, possibly depending on $\theta^*_n$, so that for $y\in\R^{n\times \ell}$
\begin{equation}\label{eq:repres}
\hat{\theta}_n(y) - \theta_n^*   \quad = \quad  \sum_{i = 1}^n g_{i,n}(y_i) + \Delta_n(y), 
\end{equation}
where, writing $r_n(y) := \sum_{i = 1}^n g_{i,n}(y_i)$, it holds for every $i \in \{1, \hdots, n\}$ and every $j \in \{1, \hdots, k\}$ that
\begin{equation}\label{eq:mom}
\mathbb{E}_{i,n}\left(g^{(j)}_{i,n}\right) = 0 \quad \text{ and } \quad 0 < \mathbb{V}_{n}\left(r_n^{(j)} \right) < \infty.
\end{equation}
Furthermore, for every coordinate $j \in \{1, \hdots, k\}$ we have
\begin{equation}\label{eq:linde}
\begin{aligned}
&\mathbb{V}_n^{-1}\left(r^{(j)}_n\right)\sum_{i = 1}^n \int_{\R^{1 \times \ell}} \left[g^{(j)}_{i, n}\right]^2 \left\{|g^{(j)}_{i, n}| \geq \varepsilon \mathbb{V}^{\frac{1}{2}}_n(r^{(j)}_n) \right\}  d\mathbb{P}_{i,n} \to 0 \\ &\text{for every } \eps > 0,
\end{aligned}
\end{equation}
and
\begin{equation*}
\mathbb{P}_n \left( \big|\mathbb{V}^{-1/2}_{n}\left(r_n^{(j)} \right) \Delta_n^{(j)} \big| \geq \eps  \right) \to 0 \text{ for every } \eps > 0.
\end{equation*}
\end{condition}

Clearly, an expansion as in Equation \eqref{eq:repres} of Condition~\ref{cond:sum} is satisfied in many applications, and can typically be obtained by a standard linearization argument (see Subsection~\ref{sec:binary} for an example and further discussion). We emphasize that the two last assumptions in Condition \ref{cond:sum} are formulated in terms of rescaled summands, which, in applications, can be exploited to circumvent restrictive compactness assumptions on moments of the distribution generating the data or the design (e.g., in Subsections~\ref{sec:homlinmods} and \ref{sec:hetlinmods} we do not need to restrict variance parameters to a compact set - as opposed to the conditions used by, e.g., \cite{eicker67} or \cite{Tib15a}; and in Subsection~\ref{sec:binary}, we do not require the smallest singular value of the design matrix to diverge to infinity - as opposed to, e.g., \cite{lv14model}).

\begin{remark}\label{rem:gin}
The careful reader will have noticed, that the functions $g_{i,n}:\R^{1\times\ell}\to\R^k$ in Condition~\ref{cond:sum} do not depend on all of the observation matrix $y\in\R^{n\times\ell}$, but only on its $i$-th row $y_i\in\R^{1\times\ell}$. This is crucial. In the sequel, however, it will be convenient to also consider $g_{i,n}$ as a function on the full sample space $\R^{n\times \ell}$. Thus, we sometimes identify $g_{i,n}$ with the composition $g_{i,n}\circ\pi_{i,n} : \R^{n\times \ell}\to \R^k$, where $\pi_{i,n}:\R^{n\times\ell}\to\R^{1\times\ell}$ is the coordinate projection $\pi_{i,n}(y) = y_i$. 
\end{remark}

Before proceeding to the main results, we briefly highlight the most important consequence of Condition~\ref{cond:sum} for our method of constructing confidence sets post-model-selection. The first step of our approach outlined in Subsection \ref{sec:fram} required the construction of confidence intervals for \textit{each} coordinate of the stacked vector of targets $\theta^*_n$. Naturally, such confidence intervals will be centered at the respective coordinates of $\hat{\theta}_n$. Now, as a first step towards the construction of such confidence intervals, Condition \ref{cond:sum} can be used to provide a useful asymptotic approximation to $\hat{\theta}_n - \theta^*_n$. More specifically, the first part of the subsequent Lemma~\ref{lemma:equivSn} provides an asymptotic approximation to the distribution  
\begin{equation}\label{eqn:distresc}
\mathbb{P}_n \circ  \left[\diag(\mathbb{V}_n(r_n))^{\dagger/2} \left(\hat{\theta}_n - \theta_n^* \right)\right].
\end{equation}
One can not expect, in general, that the distribution in the previous display converges weakly to a limiting distribution as $n \to \infty$, simply because the correlations may not stabilize. However, under Condition~\ref{cond:sum} we can show that the distributions are ``well approximated'' by the sequence of Gaussian distributions $N(0, \corr(\mathbb{V}_n(r_n)))$. Being ``well approximated'' is understood in the sense that
\begin{equation*}
d_w\left(\mathbb{P}_n \circ  \left[\diag(\mathbb{V}_n(r_n))^{\dagger/2} \left(\hat{\theta}_n - \theta_n^* \right)\right], N(0, \corr(\mathbb{V}_n(r_n))) \right) \to 0 
\end{equation*}
holds as $n \to \infty$. Here $d_w$ denotes a distance metrizing weak convergence of probability measures on the Borel sets of the respective Euclidean space the dimension of which is not shown in the notation (cf. the discussion in \cite{dudleyreal} pp. 393 for specific examples). Note that in case $\mathrm{corr}(\mathbb{V}_n(r_n))$ is constant this reduces to weak convergence. Furthermore, in the second part of Lemma~\ref{lemma:equivSn}, defining under Condition~\ref{cond:sum} the matrix
\begin{equation}\label{eq:Sn}
S_n(y) := \sum_{i = 1}^n g_{i,n}(y_i) g'_{i,n}(y_i),
\end{equation}
we show that a suitable approximation statement continues to hold if $\mathbb{V}_n(r_n)$ is replaced by $S_n$: the $d_w$-distance between
\begin{equation}\label{eqn:distresc2}
\mathbb{P}_n \circ  \left[\diag(S_n)^{\dagger/2} \left(\hat{\theta}_n - \theta_n^* \right)\right],
\end{equation}
and the sequence of (random) Gaussian distributions $N\left(0, \corr(S_n)\right)$ converges to $0$ in $\P_n$-probability as $n \to \infty$. This latter property is instrumental for our approach to constructing covariance estimators, as will be explained after the lemma.
\begin{lemma}\label{lemma:equivSn}
Under Condition \ref{cond:sum}
\begin{equation*}
d_w\left(\mathbb{P}_n \circ  \left[\diag(\mathbb{V}_n(r_n))^{\dagger/2} \left(\hat{\theta}_n - \theta_n^* \right)\right] , N\left(0, \corr(\mathbb{V}_n(r_n))\right)\right) \to 0,
\end{equation*}
and, for every $\eps > 0$, it holds that
\begin{equation*}
\mathbb{P}_n \left(d_w\left(\mathbb{P}_n \circ  \left[\diag(S_n)^{\dagger/2} \left(\hat{\theta}_n - \theta_n^* \right)\right] , N(0, \corr(S_n))\right) \geq \eps \right) \to 0.
\end{equation*}
\end{lemma}
The result is proved in Section~\ref{sec:equivSn} of the appendix using tightness arguments, a result in \cite{pollak1972}, and Raikov's theorem (cf. the statement in \cite{gnedenko} on p. 143, originally published in \cite{raikov}). At first sight one might be tempted to think that one can now immediately use $S_n$ as a covariance estimator to construct confidence intervals as envisioned in Subsection~\ref{sec:fram}. However, we emphasize that $S_n$ is in general \textit{not} an estimator of $\mathbb{V}_n(r_n)$. Typically $g_{i,n}$ depends on $\theta_n^*$, which is unknown, and thus $S_n$ is \textit{infeasible}. Hence, while Lemma~\ref{lemma:equivSn} presents a first step towards the construction of confidence sets post-model-selection, the construction of suitable covariance estimators is another step that we need to address. We nevertheless note that although Lemma~\ref{lemma:equivSn} does not answer how such estimators can be obtained, it suggests that in applications one might use as an estimator for $\mathbb{V}_n(r_n)$ a ``suitable'' predictor for $S_n$, e.g., by using ``suitable'' predictors for the unobserved components $g_{i,n}$. 

Our setup allows for substantial misspecification of the candidate set of models. Importantly, the extent to which finding consistent estimators of $\V_n(r_n)$ is possible, depends crucially on the degree of misspecification of the candidate set of models. This aspect is discussed in detail in Section~\ref{sec:CIPOSI}. In particular, there we need to distinguish between the two cases where consistent estimators are available, and the practically more relevant case where estimators need to be used that, due to the presence of a non-negligible bias component, consistently \textit{overestimate} their targets. An important part of the theory in Section~\ref{sec:CIPOSI} is that we present general results showing \textit{how} such estimators can actually be constructed. 

\subsubsection{Checking Condition \ref{cond:sum}}\label{sec:checking}
In light of Lemma \ref{lemma:equivSn}, a remarkable aspect of Condition~\ref{cond:sum} is perhaps that we obtain a multivariate central limit theorem even though the condition \textit{does not require} a joint Lindeberg-type condition concerning the random vectors $g_{i,n}$. Instead, it requires $k$ separate Lindeberg conditions concerning the behavior of the marginals only. To verify that marginal Lindeberg conditions are sufficient for our theory to go through, we exploit a result due to \cite{pollak1972}, showing that an infinitely-divisible distribution is normal if and only if each of its marginals is normal. This aspect can be very convenient when applying our results developed below, since in particular applications results on $\hat{\theta}_{\M,n}$ as required in Condition \ref{cond:sum} are likely to be available in the literature concerning asymptotic properties of estimators in misspecified models \textit{without} a model selection procedure being applied before conducting inference. Note, however, that additional arguments might be needed to obtain asymptotic results that are uniform in the true distribution, which is one of our main objectives. We also emphasize the following alternative formulation of the Lindeberg condition appearing in Condition~\ref{cond:sum} above.
 
\begin{remark}\label{rem:lindeequiv}
Using, e.g., \cite{gnedenko} Theorem 3 in Paragraph 21, one obtains that Equation \eqref{eq:linde} in Condition \ref{cond:sum} can be \emph{equivalently} phrased as
\begin{align*}
\mathbb{P}_n\circ \left( \frac{r_n^{(j)}}{\mathbb{V}^{1/2}_n(r_n^{(j)})} \right) &\Rightarrow N(0, 1) \quad\text{and}\\[7pt]
\max_{i = 1, \hdots, n} \mathbb{P}_n \left(|g_{i,n}^{(j)} | \geq \eps\mathbb{V}^{1/2}_n(r_n^{(j)}) \right) &\to 0\; \text{ for every } \eps > 0.
\end{align*}
In some applications it might be easier to check these two conditions directly (for every $j$), in particular in case one can use existing results in the literature on misspecified models without model selection as indicated above.
\end{remark}

\subsection{Confidence intervals post-model-selection}
\label{sec:CIPOSI}

In this subsection we shall now present our general asymptotic results for the construction of valid confidence intervals post-model-selection under Condition \ref{cond:sum}. We consider two different situations: (i) a situation where a consistent estimator of $\mathbb{V}_n(r_n)$ is available; (ii) a situation where a consistent estimator of $\mathbb{V}_n(r_n)$ is \textit{not} available, but it is possible to construct estimators that ``consistently overestimate'' the diagonal entries of $\mathbb{V}_n(r_n)$. Concrete examples of such consistent or ``consistently overestimating'' 
estimators are also provided, based on approximating the summands $g_{i,n}$ appearing in Condition~\ref{cond:sum}. 

Given $\mathbb{M} = \mathbb{M}_{j,n} \in \mathsf{M}_n$ we abbreviate
\begin{equation*}
\rho(\mathbb{M}) \quad := \quad \sum_{l = 1}^{j-1} m(\mathbb{M}_{l,n}),
\end{equation*}
where sums over an empty index set are to be interpreted as $0$. 

\subsubsection{Confidence intervals based on consistent estimators of $\mathbb{V}_n(r_n)$}

Our first result considers the construction of confidence intervals post-model-selection under Condition \ref{cond:sum}, and under the additional assumption that it is possible to construct a consistent estimator $\hat{S}_n$ of $\mathbb{V}_n(r_n)$. The latter assumption is certainly very restrictive, due to possible misspecification of the model, and is relaxed substantially in the following subsection.

\begin{theorem}\label{thm:fbl}
Let $\alpha \in (0, 1)$, suppose Condition \ref{cond:sum} holds, and let $\hat{S}_n: \R^{n \times \ell} \to \R^{k \times k}$ be a sequence of Borel-measurable functions so that for every $\varepsilon > 0$ 
\begin{equation*} 
\mathbb{P}_n \left( \|\corr(\hat{S}_n) - \corr\left(\mathbb{V}_n(r_n)\right)\|  + \|
\diag(\mathbb{V}_n(r_n))^{-1} \diag(\hat{S}_n) - I_k \| \geq \eps \right)
\end{equation*}
converges to $0$, or equivalently, that for every $\varepsilon > 0$
\begin{equation} \label{eq:consest}
\mathbb{P}_n \left( \|\corr(\hat{S}_n) - \corr\left(S_n\right)\|  + \|
\diag(S_n)^{\dagger} \diag(\hat{S}_n) - I_k \| \geq \varepsilon \right) \to 0.
\end{equation}
Define for every $\mathbb{M} \in \mathsf{M}_n$ and every $j = 1, \hdots, m(\mathbb{M})$ the confidence interval
\begin{equation*}
\mathrm{CI}_{1-\alpha, \mathbb{M}}^{(j), \mathrm{est}} \quad = \quad \hat{\theta}^{(j)}_{\mathbb{M}, n} \pm \sqrt{[\hat{S}_n]_{\rho(\M) + j }} ~  K_{1-\alpha}\left(\corr(\hat{S}_n)\right).
\end{equation*}
Then, 
$
\mathbb{P}_n\left( \theta^{*(j)}_{\mathbb{M}, n} \in \mathrm{CI}_{1-\alpha, \mathbb{M}}^{(j), \mathrm{est}} \text{ for all } \M\in\mathsf{M}_n \text{ and all } j = 1, \hdots, m(\mathbb{M}) \right) 
$
converges to $1-\alpha$ as $n\to\infty$. In particular, for every (measurable) model selection procedure $\hat{\mathbb{M}}_n$, we have
\begin{equation}\label{eq:infcov}
\liminf_{n \to \infty} \mathbb{P}_n\left( \theta^{*(j)}_{\hat{\mathbb{M}}_n, n} \in \mathrm{CI}_{1-\alpha, \hat{\mathbb{M}}_n}^{(j), \mathrm{est}} \text{ for all } j = 1, \hdots, m(\hat{\mathbb{M}}_n) \right) \geq 1-\alpha.
\end{equation}
\end{theorem}

Theorem~\ref{thm:fbl} is based on the assumption that an estimator $\hat{S}_n$ is available that consistently estimates $\mathbb{V}_n(r_n)$. Coming back to the discussion at the end of Subsection~\ref{sec:ass}, the vectors $g_{i,n}(y_i)$ appearing in the definition of $S_n$ are typically \textit{not} observable, because they will depend on the unknown target $\theta^*_n$, i.e., they are, more explicitly, of the form $g_{i,n}(y_i, \theta^*_n)$. In such cases $S_n$ is not a feasible candidate for $\hat{S}_n$ in the previous theorem, and therefore one will, in most cases, naturally try to obtain predictors $\hat{g}_{i,n}(y)$ for $g_{i,n}(y_i)$ by replacing the unknown target by its estimator $\hat{\theta}_n$, i.e., by setting $\hat{g}_{i,n}(y) = g_{i,n}(y_i, \hat{\theta}_n(y))$. The subsequent proposition now provides conditions on predictors $\hat{g}_{i,n}(y)$, which, if satisfied, immediately allow the construction of a consistent estimator $\hat{S}_n$ of $\mathbb{V}_n(r_n)$ by replacing each $g_{i,n}(y_i)$ in Equation \eqref{eq:Sn} by its predictor  $\hat{g}_{i,n}(y)$. In the result the predictor $\hat{g}_{i,n}(y)$ may be of the form $g_{i,n}(y_i, \hat{\theta}_n(y))$ as discussed above, but the proposition is not restricted to that particular case. Again, the conditions are assumptions concerning the large sample behavior of the marginals only, which facilitates their verification in practice.
\begin{proposition}\label{prop:obtest}
Suppose Condition \ref{cond:sum} is satisfied, and let $\hat{g}_{i,n}: \R^{n \times \ell} \to \R^k$ be Borel measurable for $i = 1, \hdots, n$ and for every $n$. Suppose that for every $j = 1, \hdots, k$ and for every $\varepsilon > 0$ it holds that
\begin{equation}\label{eq:negldiff}
\mathbb{P}_n \left( \frac{\sum_{i = 1}^n \left(g_{i,n}^{(j)} -  \hat{g}_{i,n}^{(j)} \right)^2}{\sum_{i = 1}^n [g_{i,n}^{(j)}]^2} \geq \varepsilon \right) \to 0, 
\end{equation}
or equivalently that 
\begin{equation}\label{eq:negldiff2}
\mathbb{P}_n \left( \frac{\sum_{i = 1}^n \left(g_{i,n}^{(j)} -  \hat{g}_{i,n}^{(j)} \right)^2}{\sum_{i = 1}^n \mathbb{V}_n \left( g_{i,n}^{(j)}\right)} \geq \varepsilon \right) \to 0.
\end{equation}
Then the convergence in \eqref{eq:consest} is satisfied for
\begin{equation*}
\hat{S}_n = \sum_{i = 1}^n \hat{g}_{i,n}\hat{g}'_{i,n}.
\end{equation*}
\end{proposition}

\subsubsection{Confidence intervals based on estimators that consistently overestimate the diagonal entries of $\mathbb{V}_n(r_n)$}\label{sec:overest}

Due to an asymptotically non-negligible bias term arising from misspecification of the model, it is typically difficult to obtain an estimator $\hat{S}_n$ satisfying the condition in Theorem~\ref{thm:fbl} (see Remark~\ref{rem:adaptive} and Section~\ref{sec:homlinmods:varest} for details). Nevertheless, it is often still possible to construct estimators of the diagonal entries of the matrix $\mathbb{V}_n(r_n)$ that, while possibly inconsistent, asymptotically \textit{overestimate} their targets; for a corresponding constructive result see Proposition \ref{prop:overest} below. Similarly, it is in general not difficult to find an estimator of $K_{1-\alpha}(\corr(S_n))$ that consistently overestimates that quantity, see the discussion and the result following Proposition \ref{prop:overest} below concerning upper bounds on the function $K_{1-\alpha}(.)$ over the set of all correlation matrices (using this upper bound, although leading to wider confidence intervals, also leads to substantial computational advantages). Based on such estimators it is then possible to construct asymptotically valid confidence intervals post-model-selection, even though the candidate set of models might be (severely) misspecified. This is the content of the subsequent result, which, together with Proposition \ref{prop:overest} below, is the main theoretical result in this section.

\begin{theorem}\label{thm:cons}
Let $\alpha \in (0, 1)$, and suppose Condition \ref{cond:sum} is satisfied. For every $n$ and every $j = 1, \hdots, k$ let $\hat{\nu}^2_{j,n} \geq 0$ be an estimator of $\mathbb{V}_n(r_n^{(j)})$, and let $\hat{K}_n \geq 0$ be an estimator of $K_{1-\alpha}(\corr(\mathbb{V}_n(r_n)))$, so that the sequence
\begin{equation*}
\kappa_n = \frac{K_{1-\alpha}(\corr(\mathbb{V}_n(r_n)))}{\hat{K}_n} \max_{j = 1, \hdots, k} \sqrt{\frac{[\mathbb{V}_n(r_n)]_{j} }{  \hat{\nu}^2_{j,n} }},
\end{equation*}
satisfies
\begin{equation}\label{eqn:asptup}
\mathbb{P}_n \left( \kappa_n \geq 1+ \eps \right) \to 0 \text{ for every } \eps > 0,
\end{equation}
(implicitly including that $\mathbb{P}_n(\kappa_n \text{ is well defined} ) \to 1$) or, equivalently, that the condition in \eqref{eqn:asptup} holds with $\kappa_n$ replaced by
\begin{equation*}
\frac{K_{1-\alpha}(\corr(S_n))}{\hat{K}_n} \max_{j = 1, \hdots, k} \sqrt{\frac{[S_n]_{j} }{  \hat{\nu}^2_{j,n} }}.
\end{equation*}
For every $\mathbb{M} \in \mathsf{M}_n$ and every $j = 1, \hdots, m(\mathbb{M})$, define the confidence interval
\begin{equation*}
\mathrm{CI}_{1 - \alpha, \mathbb{M}}^{(j), \mathrm{oest}} \quad = \quad \hat{\theta}^{(j)}_{\mathbb{M}, n} \pm \sqrt{\hat{\nu}^2_{\rho(\M) + j, n}} ~  \hat{K}_n.
\end{equation*}
Then, for every (measurable) model selection procedure $\hat{\mathbb{M}}_n$, we have
\begin{equation*}
\liminf_{n \to \infty} \mathbb{P}_n \left( \theta_{\hat{\mathbb{M}}_n, n}^{*(j)} \in \mathrm{CI}_{1 - \alpha, \hat{\mathbb{M}}_n}^{(j), \mathrm{oest}} \text{ for all } j = 1, \hdots, m(\hat{\mathbb{M}}_n)  \right) \geq 1-\alpha.
\end{equation*}
In the important special case where $\hat{K}_n \geq K_{1-\alpha}(\corr(\mathbb{V}_n(r_n)))$ holds eventually,  the condition in Equation \eqref{eqn:asptup} is implied by the condition that for every $j = 1, \hdots, k$ it holds that
\begin{equation}\label{eqn:special1}
\mathbb{P}_n \left( \sqrt{\frac{[S_n]_{j} }{  \hat{\nu}^2_{j,n} }} \geq 1+ \eps \right) \to 0 \text{ for every } \eps > 0,
\end{equation}
or equivalently, that for every $j = 1, \hdots, k$ it holds that
\begin{equation}\label{eqn:special2}
\mathbb{P}_n \left( \sqrt{\frac{[\mathbb{V}_n(r_n)]_{j} }{  \hat{\nu}^2_{j,n} }} \geq 1+ \eps \right) \to 0 \text{ for every } \eps > 0.
\end{equation}
\end{theorem}

The preceding theorem operates under the assumption that estimators are available that consistently overestimate the diagonal entries of $\mathbb{V}_n(r_n)$ and ~$K_{1-\alpha}\linebreak[1](\corr(\mathbb{V}_n(r_n)))$. The following result now shows how such estimators for the diagonal entries of $\mathbb{V}_n(r_n)$ can be obtained. To construct an estimator $\hat{K}_n$ that eventually satisfies $\hat{K}_n \geq K_{1-\alpha}(\corr(\mathbb{V}_n(r_n)))$ (as required for the special case of Theorem~\ref{thm:cons}) one can numerically compute the upper bound in Lemma \ref{lem:upper} below. The subsequent result considers the case where the vectors $g_{i,n}$ from Condition~\ref{cond:sum} are well approximated in the sense of the condition appearing in Proposition \ref{prop:obtest}, but where the approximating quantities are now \textit{unobservable} due to non-stochastic additive error terms. These additive error terms typically are bias terms due to misspecification of the model. This is further discussed after the proposition.

\begin{proposition}\label{prop:overest}
Suppose Condition \ref{cond:sum} is satisfied, and let $\tilde{g}_{i,n}: \R^{n \times \ell} \to \R^k$ and $\hat{g}_{i,n}: \R^{n \times \ell} \to \R^k$ be Borel measurable for $i = 1, \hdots, n$ and for every $n$. Suppose that for every $j = 1, \hdots, k$ and for every $\varepsilon > 0$ the condition \eqref{eq:negldiff}, or equivalently \eqref{eq:negldiff2}, is satisfied. Suppose further that there exist real numbers $a_{i,n}^{(j)}$ so that for $y\in\R^{n\times\ell}$
\begin{equation*}
\tilde{g}^{(j)}_{i,n}(y) \quad = \quad \hat{g}^{(j)}_{i,n}(y) + a_{i,n}^{(j)}
\end{equation*}
holds for every $n \in \N$, $i \in \{1, \hdots, n\}$ and $j \in \{1, \hdots, k\}$. Then the statement in \eqref{eqn:special2} is satisfied for
\begin{equation*}
\hat{\nu}^2_{j,n} \quad = \quad \sum_{i = 1}^n \left[\tilde{g}^{(j)}_{i,n}\right]^2 \quad \text{ for } \quad j = 1, \hdots, k.
\end{equation*}
\end{proposition}
The proposition is developed for situations where random variables $\tilde{g}_{i,n}^{(j)}$ are observed, that can be decomposed as the sum of unobserved random variables $\hat{g}_{i,n}^{(j)}$, which satisfy \eqref{eq:negldiff}, and unobserved real numbers $a_{i,n}^{(j)}$. In contrast to the situation in Proposition~\ref{prop:obtest}, now the (unobservable) random variables $\hat{g}_{i,n}^{(j)}$ can not be used for the construction of estimators. Nevertheless, the proposition shows how suitable variance estimators can then still be constructed based on the observed quantities $\tilde{g}_{i,n}^{(j)}$. Confidence intervals post-model-selection can then be obtained via Theorem~\ref{thm:cons}. Besides being suitable for situations where random variables satisfying \eqref{eq:negldiff} are not observed (otherwise one could use Proposition \ref{prop:obtest} to obtain consistent estimators), Proposition~\ref{prop:overest} is particularly geared towards the case where the non-stochastic additive components $a_{i,n}^{(j)}$ are non-negligible in the sense that 
\begin{equation*}
\frac{\sum_{i = 1}^n [a_{i,n}^{(j)}]^2}{\mathbb{V}_n(r_n^{(j)})} \not \to 0 \text{ holds for some } j \in \{1, \hdots, k\}.
\end{equation*}
For if the non-stochastic additive components are negligible in this sense, a consistent estimator of $\mathbb{V}_n(r_n)$ in the sense of \eqref{eq:consest} can be constructed:

\begin{remark}\label{rem:adaptive}
Using the simple bound $(g_{i,n}^{(j)} - \tilde{g}_{i,n}^{(j)})^2 \le 2(g_{i,n}^{(j)} - \hat{g}_{i,n}^{(j)})^2 + 2 a_{i,n}^2$, it is easy to verify that if the non-stochastic additive components $a_{i,n}$ are negligible in the previously defined sense, then $\tilde{g}_{i,n}$ satisfies the assumptions of $\hat{g}_{i,n}$ appearing in Proposition \ref{prop:obtest}. As a consequence, the estimator 
\begin{equation*}
\tilde{S}_n = \sum_{i = 1}^n \tilde{g}_{i,n}\tilde{g}'_{i,n}
\end{equation*}
satisfies \eqref{eq:consest}, and one can construct confidence intervals based on this estimator as discussed in Theorem \ref{thm:fbl}. Note that $\hat{\nu}^2_{j,n} = [\tilde{S}_n]_{j}$.
\end{remark}

Let us finally consider an upper bound on $K_{1-\alpha}(\Gamma)$ as required in the special case of Theorem \ref{thm:cons} above. The bound we shall discuss is based on the quantity $B_{\alpha}(q,N)$, for $q,N \in \mathbb{N}$, defined as the smallest $t >0$ so that
\[
\mathbb{E}_G \left( \min\left(1, N \left[  1 - F_{Beta,1/2,(q-1)/2} (t^2/G^2) \right] \right) \right) \leq \alpha,
\]
where $F_{Beta,1/2,(q-1)/2}$ is the cumulative distribution function of the \linebreak Beta($1/2$,$(q-1)/2$) distribution, and where $G^2$ follows a chi-squared distribution with $q$ degrees of freedom.
The quantity $B_{\alpha}(q,N)$ corresponds to the quantity $K_4$ of \cite{bachoc14valid} in the known variance case (for a discussion of numerical algorithms for obtaining $B_{\alpha}(q,N)$ in practice we confer the reader to that reference). We have \citep{bachoc14valid,Berk13} that $B_{\alpha}(q,N)$ is larger than all the $1-\alpha$ quantiles of random variables of the form
$\max_{i=1,...,N} |v_i' \epsilon|$, where $v_1,...,v_N$ are column vectors of $\mathbb{R}^q$ with $||v_i|| \leq 1$ and where $\epsilon \sim N(0,I_q)$; furthermore, for fixed $\alpha$ and $N$ the function $q\mapsto B_{\alpha}(q, N)$ is monotonically increasing.

Asymptotic approximations of $B_{\alpha}(q,N)$ for large $q$ and $N$ are provided in \cite{bachoc14valid}, \cite{Berk13} and \cite{zhang15spherical}. In particular, as $q,N \to \infty$, 
\[
B_{\alpha}(q,N)/\sqrt{q \left( 1-N
^{-2/\left( q-1\right) }\right) }\rightarrow 1,
\]
from Proposition 2.10 in \cite{bachoc14valid}, itself building on results from \cite{Berk13} and \cite{zhang15spherical}. 

An often useful upper bound on $K_{1-\alpha}(\Gamma)$ with $\Gamma$ a $k \times k$-dimensional correlation matrix is provided in the following lemma:

\begin{lemma} \label{lem:upper}
For every $\alpha \in (0, 1)$ and a $k \times k$ correlation matrix $\Gamma$ we have
\[
K_{1-\alpha}(\Gamma)
\leq 
B_{\alpha}(\rank(\Gamma),k). 
\]
\end{lemma}
In a particular application it might of course be possible to obtain better upper bounds by exploiting structural properties of the specific correlation matrix $\Gamma$ at hand, cf. Subsection~\ref{sec:homlinmods}. Using the upper bound of Lemma~\ref{lem:upper} can also be very useful in situations where the computation of $K_{1-\alpha}(\Gamma)$ is infeasible.

\section{Applications}
\label{sec:applications}

In this section we now apply the general results obtained in Section \ref{sec:genposi} to some important special cases that are frequently encountered in practice. As already mentioned in Section \ref{ssc:discfram}, we now consider situations of the following type:
\begin{enumerate}
\item The underlying distribution $\mathbb{P}_n$ is assumed to be an element of a set of distributions $\mathbf{P}_n$.
\item A model $\hat{\mathbb{M}}_n$ is selected in a data-driven way from a candidate set $\mathsf{M}_n$, which is potentially misspecified, i.e., $\mathbf{P}_n \not \subseteq \bigcup_{\M\in\mathsf{M}_n}\M$.
\item One aims at constructing confidence intervals for all coordinates of the model-specific target parameter $\theta^*_{\hat{\mathbb{M}}_n, n}$.
\end{enumerate}
The scenarios we discuss in this section are all concerned with the case $\ell=1$ (the case $\ell>1$ is of interest, e.g., in a regression problem with random design where one observes a data matrix $(y_i, x_{i1}, \dots, x_{ip})_{i=1}^n$ which is a realization of a probability distribution $\P_n$ on the sample space $\R^{n\times (p+1)}$), that is, we observe realizations of a random $n$-vector $Y_n=(Y_{1,n}, \dots, Y_{n,n})'$ defined on some probability space $(\Omega, \mathcal A, \P)$, whose distribution under $\P$ coincides with $\P_n\in\mathbf{P}_n$ (we write $\E$ and $\V$ to denote the expectation and variance-covariance operator with respect to $\P$). In Subsection~\ref{sec:homlinmods}, we consider the case where the candidate set $\mathsf{M}_n$ consists of fixed design homoskedastic linear models. In this framework, the model selection problem is equivalent to a subset-selection problem of regressors. Here, the model-specific target we consider is the coefficient vector of the projection of the mean vector $\mu_n = \E(Y_n)\in\R^n$ onto the model-specific fixed regressor matrix. In such a setup, intervals post-model-selection have also been suggested in \cite{Tib15a}, but for specific model selection methods. Our approach can also be used to obtain confidence intervals in their setup, and requires less assumptions on the set of distributions over which uniformity is achieved and on the design matrices allowed. In Subsection~\ref{sec:hetlinmods} we then discuss the case where $\mathsf{M}_n$ consists of fixed design heteroskedastic linear models. While the model-specific target is the same as in the homoskedastic case, the construction of confidence sets is more complicated as the heteroskedasticity needs to be taken into account. The results of this section can be viewed as an extension of the influential results in \cite{eicker67} to the potentially misspecified, post-model-selection context. Comparable results do not exist to the best of our knowledge. Finally, in Subsection~\ref{sec:binary}, we consider the situation where $\mathsf{M}_n$ consists of binary regression models. We allow for situations where both the regressors and the link function is chosen in a data-driven way. In each candidate model the model-specific target vector is here obtained as a minimizer of the Kullback-Leibler divergence. For numerical results concerning the methods discussed in Sections~\ref{sec:homlinmods} and \ref{sec:binary} see Section~\ref{sec:sim} as well as Section~\ref{app:sim} of the appendix.

\subsection{Inference post-model-selection when fitting fixed design linear models to homoskedastic data}\label{sec:homlinmods}

One important application of our general theory is the case where homoskedastic linear regression models are fit to data. The feasible sets for the true underlying distribution $\mathbb{P}_n$ we can allow for in this setup is denoted as $\mathbf{P}_n^{(\mathrm{lm})}\left(\delta, \tau \right)$, where $\delta > 0$ and $\tau \ge 1$, and is defined as follows: the distribution $\P_n$ of the random $n$-vector $Y_n = (Y_{1,n}, \hdots, Y_{n,n})'$ is an element of $\mathbf{P}_n^{(\mathrm{lm})}\left(\delta, \tau \right)$ if and only if the $n$ coordinates of $Y_n$ are independent, homoskedastic (i.e., the variances of the coordinates are equal to some $\V(Y_{i,n}) = \sigma_n^2 \in (0, \infty)$, for all $i=1,\dots, n$), and 
\begin{equation*}
\max_{i = 1, \hdots, n} \mathbb{E}\left(|Y_{i,n} - \mathbb{E}(Y_{i,n})|^{2+\delta}\right)^{\frac{2}{2+\delta}} \leq \tau \sigma_n^2.
\end{equation*}
%
%
%
Note that $\mathbf{P}_n^{(\mathrm{lm})} \left(\delta, \tau \right)$ is empty for $\delta>0$ and $\tau< 1$, because then the inequality in the previous display can never be satisfied. Furthermore, observe that $\mathbf{P}_n^{(\mathrm{lm})} \left(\delta, \tau \right)$ contains the set of $n$-variate spherical normal distributions with unrestricted mean vector if 
$$
\mathbf{\Gamma}\left( \frac{3+\delta}{2}\right) \le \left( \tau/2 \right)^{1+\delta/2}\sqrt{\pi},
$$
where $\mathbf{\Gamma}(.)$ denotes the Gamma-function. For such a pair $(\delta,\tau)$ the set $\mathbf{P}_n^{(\mathrm{lm})} \left(\delta, \tau \right)$ thus contains the Gaussian model considered in \citet{Berk13}. Finally note that there is no restriction on the mean vector $\mu_n = \E(Y_n)\in\R^n$ of elements of $\mathbf{P}_n^{(\mathrm{lm})} \left(\delta, \tau \right)$.

We are interested in a situation where one works with candidate sets consisting of homoskedastic linear models. That is, a situation where one wants to conduct  inference on the mean vector $\mu_n$ of the underlying distribution $\mathbb{P}_n$, and it is \emph{assumed} by the practitioner that $\mu_n$ is an element of  $\s(X_n)$, the column span of a design matrix $X_n \in \R^{n \times p}$, with $p$ not depending on $n$, or that $\mu_n$ is at least ``well-approximated'' by an element of that linear space; and that the practitioner \textit{knows} (and takes into account in the construction of the confidence sets) that the observations have identical variances (for a situation where the observations are heteroskedastic see Subsection~\ref{sec:hetlinmods}). In such a situation one then often tries to decide in a data-driven way \textit{which} regressors to use, i.e., one needs to solve a subset-selection problem. We assume that we are given a nonempty set $\mathcal{I} = \{M_1, \hdots, M_d\}$ of nonempty subsets of $\{1, 2, \hdots, p\}$, that does not depend on $n$. Given $M \in \mathcal{I}$ we shall denote by $X_n[M]$ the matrix obtained from $X_n$ by striking all columns whose index is not an element of $M$. We then consider for each $j \in \{1, \hdots, d\}$ a linear, homoskedastic candidate model $\mathbb{M}_{j, n}$ with fixed design $X_n[M_j]$, i.e., the distribution of a random vector $z = (z_1, \hdots, z_n)'$ is an element of $\mathbb{M}_{j, n}$ if and only if there exists a $\beta \in \R^{|M_j|}$ so that the random (residual) vector $z - X_n[M_j]\beta$ has independent, homoskedastic coordinates with mean zero.
%
Our candidate set of models is then given by
\begin{equation*}
\mathsf{M}_n = \left\{\mathbb{M}_{j,n}: j = 1, \hdots, d \right\}.
\end{equation*}
We assume that $X_n$ satisfies the following condition, where we denote the $i$-th row of $X_n$ by $X_{i,n}$:
\renewcommand{\thecond}{X1}
\begin{cond}\label{cond:X} 
Eventually $\rank(X_n) = p$, and for every $M\in\mathcal I$,
\begin{equation*}
\max_{i = 1, \hdots, n} X_{i,n}[M] \left( X_{n}[M]'X_{n}[M]  \right)^{-1} X_{i,n}[M]' \to 0.
\end{equation*}
\end{cond}

\begin{remark}
Condition~\ref{cond:X} particularly holds if $\rank(X_n) = p$, eventually, and $\max_{i=1,\dots, n} X_{i,n}(X_n'X_n)^{-1}X_{i,n}' \to 0$. Moreover, it also holds in case $\|X_{i,n}\|$ is bounded and $\lambda_{\min}(\frac{1}{n} X_n'X_n)$ is bounded away from $0$, which is typically the case in sufficiently balanced factorial designs, but Condition~\ref{cond:X} is obviously much more general. For example, it also covers the important cases of polynomial regressors, trigonometric regressors, or mixed polynomial and trigonometric regressors (cf. the discussion in \cite{eicker67}, pp. 64).
Finally, we point out that the condition
\begin{equation*}
\max_{i = 1, \hdots, n} X_{i,n}[M] \left( X_{n}[M]'X_{n}[M]  \right)^{-1} X_{i,n}[M]' \to 0
\end{equation*}
is classical, and is necessary for asymptotic normality of the ordinary-least-squares estimator in the fixed model $\M$ \citep[see][]{Huber73, Arnold80}.

\end{remark}

The model-specific target of inference is then (eventually) defined as follows: Given $\mathbb{M}\in \mathsf{M}_n$ with a corresponding index set $M$, we let
\begin{equation}\label{eqn:targethom}
\beta_{\mathbb{M}, n}^* = \beta_{\mathbb{M}, n}^*(\mathbb{P}_n) = \left(X_n[M]'X_n[M]\right)^{-1} X_n[M]'  \mu_n,
\end{equation}
i.e., $\beta_{\M,n}^*$ is the coefficient vector corresponding to the orthogonal projection of $\mu_n$ onto $\s(X_n[M])$. 

We shall now describe how asymptotically uniformly valid confidence sets can be constructed post-model-selection for the target defined in Equation \eqref{eqn:targethom} above: Given $\mathbb{M} \in \mathsf{M}_n$ with index set $M$, we estimate the corresponding target by the model-specific ordinary-least-squares estimator, i.e., by
\begin{equation}\label{eq:OLS}
\hat{\beta}_{\mathbb{M}, n}(y) = \left(X_n[M]'X_n[M]\right)^{-1} X_n[M]'y;
\end{equation}
let
\begin{equation*}
\hat{\sigma}^2_{\mathbb{M}, n}(y) = \frac{1}{n- m(\mathbb{M})} \sum_{i = 1}^n (y_i - X_{i,n}[M] \hat{\beta}_{\mathbb{M}, n}(y) )^2,
\end{equation*}
where $m(\mathbb{M})$ here coincides with $|M|$, the cardinality of $M$, and define for $\alpha \in (0, 1)$ and $j = 1, \hdots, m(\mathbb{M})$
\begin{equation}\label{eq:POSICIlm}
\mathrm{CI}_{1-\alpha, \mathbb{M}}^{(j), \mathrm{lm}} = \hat{\beta}^{(j)}_{\mathbb{M}, n} \pm  \sqrt{\hat{\sigma}^2_{\mathbb{M}, n} \left[\left( X_n[M]'X_n[M]  \right)^{-1}\right]_j} K_{1-\alpha}(\corr(\Gamma_n)), 
\end{equation}
where the block-matrix $\Gamma_n$ is defined via its $s,t$-th block of dimension $|M_s|\times|M_t|$ given by
\begin{align*}
&\E_n\left[ \left(\hat{\beta}_{\M_s,n} - \beta^*_{\M_s,n}\right) \left(\hat{\beta}_{\M_t,n} - \beta^*_{\M_t,n}\right)'  \right] \\
&\quad= \sigma_n^2\left( X_n[M_s]' X_n[M_s] \right)^{-1} X_n[M_s]' X_n[M_t] \left( X_n[M_t]' X_n[M_t] \right)^{-1}, 
\end{align*}
for $s, t \in \{1, \hdots, d\}$. Note that while $\Gamma_n$ depends on $\sigma_n^2$, $\corr(\Gamma_n)$ is observed. Essentially, the construction in \eqref{eq:POSICIlm} coincides with the confidence intervals of \citet{Berk13}. However, there are two major differences. First of all, we here do not assume that the data are Gaussian, which is why we resort to asymptotic results. This is also the reason why our constant $K_{1-\alpha}$, the so called POSI constant, is the quantile of a maximum of Gaussian rather than t-distributed random variables, as is the case in \citet{Berk13}. Furthermore, we simply use the usual variance estimator $\hat{\sigma}^2_{\mathbb{M}, n}$ which, in general, is not unbiased or uniformly consistent (due to potential misspecification) as required in \citet{Berk13}, but we still obtain uniformly valid inference asymptotically. This shows that the restrictive assumption of \citet{Berk13}, that there exists an unbiased or a uniformly consistent estimator for $\sigma_n^2$ (cf. Proposition~\ref{prop:no:consistent:sigma} below, as well as the discussion in Remark~2.1 of \cite{leeb13various} and in Appendix~A of \cite{bachoc14valid}), is not needed for uniform asymptotic validity. If the estimator $\hat{\sigma}^2_{\mathbb{M}, n}$ is used in the construction of \citet{Berk13}, then their confidence intervals asymptotically coincide with our procedure.
We also point out that the classical variance estimator used here adapts to misspecification in the sense that it is consistent for $\sigma_n^2$ if a first order correct model is selected and it otherwise overestimates the target in the sense of Section~\ref{sec:overest} (cf. Remark~\ref{rem:adaptive} and Subsection~\ref{sec:homlinmods:varest} for details). 
 
It is also worth noting that up to the choice of the last multiplicative factor $K_{1-\alpha}(\corr(\Gamma_n))$ in the definition of the confidence intervals above, i.e., the POSI constant, this is just the usual confidence interval for the $j$-th coordinate of the coefficient vector one would typically use in practice working with homoskedastic linear models, and by following the naive way of ignoring the data-driven model selection step. The crucial difference, however, is that the naive approach is invalid \citep[see, e.g.,][]{leeb13various, bachoc14valid}.

We now present the main result of this subsection, where we emphasize once more that the (measurable) model selection procedure $\hat{\M}_n$ is data-driven and unrestricted, and that some, or all of the candidate models in $\mathsf{M}_n$ may be misspecified, i.e., $\mathbf{P}_n^{(\mathrm{lm})} \left(\delta, \tau \right) \not \subseteq \bigcup_{\M\in\mathsf{M}_n}\M$. Nevertheless it is possible to construct an asymptotically uniformly valid confidence set for the model-specific target vector $\beta_{\hat{\M}_n,n}^*$. 
\begin{theorem}\label{thm:homlinmods}
Let $\alpha \in (0, 1)$, $\delta > 0$ and $\tau \ge 1$, suppose Condition~\ref{cond:X} holds, and let $\hat{\mathbb{M}}_n$ be a (measurable) model selection procedure, i.e., a measurable map from the sample space $\R^n$ to $\mathsf{M}_n$. Then
\begin{equation*}
\liminf_{n \to \infty} \inf_{\mathbb{P}_n \in \mathbf{P}_n^{(\mathrm{lm})} \left(\delta, \tau \right)} \mathbb{P}_n \left(\beta_{\hat{\mathbb{M}}_n, n}^{* (j)} \in \mathrm{CI}_{1-\alpha, \hat{\mathbb{M}}_n}^{(j), \mathrm{lm}} \text{ for all } j = 1, \hdots, m(\hat{\mathbb{M}}_n)  \right) \geq 1-\alpha.
\end{equation*}
\end{theorem}

\subsubsection{Coverage of individual parameters}\label{sec:individual}

The statement in Theorem \ref{thm:homlinmods} concerns simultaneous coverage of all coefficients of the model-dependent target parameter. In some applications, it may be of interest to construct confidence intervals only for single coefficients, i.e., coefficients corresponding to a certain regressor. Of course, as simultaneous coverage implies individual coverage, the confidence intervals in the previous section achieve this goal a fortiori. However, shorter confidence intervals can be constructed if one only wants to achieve individual coverage. This is discussed subsequently. Consider the case where one wants to conduct inference on the first column vector of $X_n$ (otherwise just re-arrange the columns of $X_n$). Then, the candidate sets $M_1, \hdots, M_d$ one works with, will all necessarily include the first column of $X_n$, i.e., $1\in M_j$ for all $j=1,\dots, d$. Whereas the candidate models in $\mathsf{M}_n$ are still defined as above, the model-specific target vector changes: Now, for every $\mathbb{M} \in \mathsf{M}_n$ the target is the first coordinate of the vector $\beta^*_{\mathbb{M}, n}$ defined in \eqref{eqn:targethom}, which we denote as $\eta^*_{\mathbb{M}, n}$. Furthermore, in each model $\mathbb{M} \in \mathsf{M}_n$ we estimate this target by $\hat{\eta}_{\mathbb{M}, n}$, the first coordinate of the model-specific OLS estimator defined in \eqref{eq:OLS}. Finally, to define our confidence intervals in this case, let $\Xi_n$ denote the $d \times d$-dimensional matrix with $s,t$-th entry given by
\begin{equation*}
\sigma_n^2\left[ \left( X_n[M_s]' X_n[M_s] \right)^{-1} X_n[M_s]' X_n[M_t] \left( X_n[M_t]' X_n[M_t] \right)^{-1} \right]_{1}.
\end{equation*}
Given $\alpha \in (0, 1)$, define for every $\mathbb{M} \in \mathsf{M}_n$ with corresponding index set $M$ the confidence interval
\begin{equation*}
\mathrm{CI}_{1-\alpha, \mathbb{M}} = \hat{\eta}_{\mathbb{M}, n} \pm \sqrt{\hat{\sigma}^2_{\mathbb{M}, n} \left[ (X_n[M]'X_n[M])^{-1} \right]_1} K_{1-\alpha}(\mathrm{corr}(\Xi_n)).
\end{equation*}
Note that, apart from the choice of the POSI constant, the confidence interval $\mathrm{CI}_{1-\alpha, \mathbb{M}}$ is identical to $\mathrm{CI}^{(1), \mathrm{lm}}_{1-\alpha, \mathbb{M}}$. While the latter interval is based on $K_{1-\alpha}(\mathrm{corr}(\Gamma_n))$, the POSI constant used here is $K_{1-\alpha}(\mathrm{corr}(\Xi_n))$. From the definition of $K_{1-\alpha}$, together with the relationship of $\Xi_n$ and $\Gamma_n$, it follows immediately that 
$$K_{1-\alpha}(\mathrm{corr}(\Xi_n)) \leq K_{1-\alpha}(\mathrm{corr}(\Gamma_n)),$$  
with equality holding only if $\mathcal{I} = \{\{1\}\}$, i.e., no model selection. That is, the confidence intervals for individual parameters constructed here are smaller than the ones guaranteeing simultaneous coverage. The following can now be said about their asymptotic coverage properties.
\begin{theorem}\label{thm:homlinmodsindi}
Let $\alpha \in (0, 1)$, $\delta > 0$ and $\tau \ge 1$, suppose Condition~\ref{cond:X} holds, and let $\hat{\mathbb{M}}_n$ be a (measurable) model selection procedure. Suppose every element of $\mathcal{I}$ contains $1$. Then
\begin{equation*}
\liminf_{n \to \infty} 
\inf_{\mathbb{P}_n \in \mathbf{P}_n^{(\mathrm{lm})} \left(\delta, \tau \right)} \mathbb{P}_n \left(\eta^*_{\hat{\mathbb{M}}_n, n} \in  \mathrm{CI}_{1-\alpha, \hat{\mathbb{M}}_n} \right) \geq 1-\alpha.
\end{equation*}
\end{theorem}

We finally mention that similar arguments can be used to construct confidence intervals for model-dependent linear combinations of regression coefficients (i.e., contrasts). Furthermore, analogous constructions can be used to obtain confidence intervals for individual coefficients (or, more generally, model-dependent contrasts) in the examples discussed in Sections~\ref{sec:hetlinmods} and \ref{sec:binary} below. Due to space constraints we do not provide details.

\subsubsection{The POSI-intervals automatically adapt to misspecification}
\label{sec:homlinmods:varest}

Let us for a moment forget about the model selection step and consider the classical construction of confidence intervals for $\beta_{\M,n}^*$ based on the asymptotic normality of $\hat{\beta}_{\M,n}$, for $\M\in\mathsf{M}_n$ fixed. 
It is well known that the usual variance estimator $\hat{\sigma}_{\M,n}^2$ in model $\M\in\mathsf{M}_n$ with index set $M\in\mathcal I$ is consistent for the true error variance $\sigma_n^2$, if and only if, the model $\M$ is asymptotically first order correct. More specifically, the estimator $\hat{\sigma}_{\M,n}^2$ is upward biased when $\M$ is mean-misspecified and the bias is given by $\|(I_n - P_{X_n[M]})\mu_n/\sigma_n\|^2/(n-|M|)$, where $P_{\dots}$ denotes the projection matrix corresponding to the column span of the matrix indicated in the subscript and where $\mu_n = \E(Y_n)$. Consequently, using $\hat{\sigma}_{\M,n}^2$ in the construction of confidence intervals for $\beta_{\M,n}^*$ leads to `conservative' inference in case $\M$ is misspecified, but the resulting inference is still valid. Moreover, if $\M$ is correct then the resulting inference on $\beta_{\M,n}^*$ is also asymptotically efficient in the sense that the obtained intervals have minimal asymptotic length (they coincide with the infeasible intervals calculated with knowledge of $\sigma_n^2$). 

If the working model is now selected by a data dependent selection procedure $\hat{\M}_n : \R^n \to \mathsf{M}_n$, the POSI-intervals suggested in \eqref{eq:POSICIlm} involve the post-model-selection estimator $\hat{\sigma}_{\hat{\M}_n,n}^2$ and we can again ask the question about its consistency properties. Clearly, consistency now depends on the existence of a correct model in the class $\mathsf{M}_n$ of candidate models as well as on the ability of $\hat{\M}_n$ to identify one such correct model. The following proposition provides a precise quantitative formulation of this claim (see Section~\ref{sec:lmvarest} in the appendix for the proof). A similar statement is also discussed in \citet[][Theorem~3.6 and Lemma~C.2]{bachoc14valid}.

\begin{proposition}\label{prop:lmvarest}
Fix $\delta>0$ and $\tau\ge 1$ and suppose that eventually $\rank(X_n)=p$. For any sequence $\P_n \in\mathbf{P}_n^{(\mathrm{lm})}(\delta,\tau)$, the variance estimator $\hat{\sigma}_{\M,n}^2(y) = (n-|M|)^{-1}y'(I_n-P_{X_n[M]})y$ satisfies
\begin{align*}
\P_n\left( \left|\frac{\hat{\sigma}_{\M,n}^2}{\sigma_n^2} \left(1 + \frac{\|(I_n-P_{X_n[M]})\mu_n/\sigma_n\|^2}{n-|M|}\right)^{-1} - 1\right| > \eps \right)
\;\to\;0,
\end{align*}
for every $\eps>0$ and for every $M\in\mathcal I$.
\end{proposition}

In particular, since the index set $\mathcal I$ is finite, not depending on $n$, Proposition~\ref{prop:lmvarest} shows that the post-model-selection estimator $\hat{\sigma}_{\hat{\M}_n,n}^2$ is consistent for $\sigma_n^2$ (on a relative scale) along sequences $\P_n\in\mathbf{P}_n^{(\mathrm{lm})}(\delta,\tau)$ with the property that 
\begin{equation}\label{eq:CondCons}
\frac{\|(I_n-P_{X_n[\hat{M}_n]})\mu_n/\sigma_n\|^2}{n-|\hat{M}_n|} \;\xrightarrow[n\to\infty]{}\; 0
\end{equation}
in $\P_n$-probability, where $\hat{M}_n$ denotes the index set corresponding to $\hat{\mathbb{M}}_n$. Of course, in general, this condition can not be verified in practice. But it still tells us that if the model selection procedure $\hat{\M}_n$ finds an approximately first order correct model in the sense of \eqref{eq:CondCons}, then the POSI-intervals in \eqref{eq:POSICIlm} have the same asymptotic length as the infeasible intervals that use knowledge of $\sigma_n^2$. 

One can now raise the question whether it is at all possible to construct a uniformly (over $\mathbf{P}_n^{(\mathrm{lm})}(\delta,\tau)$) consistent estimator for $\sigma_n^2$. However, in the framework we consider, and in order to obtain valid confidence intervals, it is necessary to use an estimator that, for certain sequences $\P_n\in \mathbf{P}_n^{(\mathrm{lm})}(\delta,\tau)$, consistently overestimates the variance $\sigma_n^2$. More precisely, in Proposition~\ref{prop:no:consistent:sigma} below (proved in Section~\ref{subsection:proof:no:consistent:sigma} of the supplement), we show that there does not exist a uniformly consistent estimator of the variance $\sigma_n^2$ if $\tau>1$.

\begin{proposition} \label{prop:no:consistent:sigma}
Let $\delta > 0$ and $\tau > 1$. There does \emph{not} exist a sequence of measurable functions $(\hat{\sigma}^2_{n})_{n \in \mathbb{N}}$ with $\hat{\sigma}^2_{n}: \mathbb{R}^n \to [0, \infty)$, so that for every $\varepsilon > 0$
\begin{equation}\label{eqn:conscounter}
\sup_{\mathbb{P}_n \in \mathbf{P}_n^{(\mathrm{lm})}(\delta, \tau)} \mathbb{P}_n \left( \left| \frac{\hat{\sigma}^2_n}{\sigma^2_n} - 1 \right| > \varepsilon \right) \to 0.
\end{equation}
\end{proposition}


\subsection{Inference post-model-selection when fitting fixed design linear models to heteroskedastic data}
\label{sec:hetlinmods}

The feasible sets for $\mathbb{P}_n$ we consider here again depend on two parameters $\delta > 0$ and $\tau \ge 1$ but, compared to the set $\mathbf{P}_n^{(\mathrm{lm})} \left(\delta, \tau \right)$ defined above, we now drop the requirement of homoskedasticity: the distribution of a random $n$-vector $Y_n = (Y_{1,n}, \hdots, Y_{n,n})'$ is an element of $\mathbf{P}_n^{(\mathrm{het})} \left(\delta, \tau \right)$ if and only if the $n$ coordinates of $Y_n$ are independent, the variance $\sigma_{i,n}^2 = \V(Y_{i,n}) \in(0,\infty)$ exists for every $i = 1, \hdots, n$, and 
\begin{equation*}
\max_{i = 1, \hdots, n} \left[\mathbb{E}\left(|Y_{i,n} - \mathbb{E}(Y_{i,n})|\right)^{2+\delta}\right]^{\frac{2}{2+\delta}} \leq \tau \min_{i = 1, \hdots, n} \sigma_{i,n}^2.
\end{equation*}
Here, we consider a situation where one works with candidate sets consisting of heteroskedastic linear models, i.e., where similar as in Section~\ref{sec:homlinmods} one is interested in conducting inference on $\mu_n = \E(Y_n)\in\R^n$, and it is \textit{assumed} that $\mu_n$ is an element of  $\s(X_n)$, the column span of a design matrix $X_n \in \R^{n \times p}$ with $p$ fixed; but where it is now taken into account that the observations may have different variances. We start with a set $\mathcal{I} = \{M_1, \hdots, M_d \}$ as in Subsection~\ref{sec:homlinmods}, and we then define for each $j \in \{1, \hdots, d\}$ the linear, heteroskedastic model $\mathbb{M}_{j, n}$ as follows: the distribution of a random vector $z = (z_1, \hdots, z_n)'$ is an element of $\mathbb{M}_{j, n}$ if and only if there exists a $\beta \in \R^{|M_j|}$ so that the random (residual) vector $z - X[M_j]\beta$ has independent coordinates with positive finite variances and mean zero. The corresponding candidate set of models is then given by
\begin{equation*}
\mathsf{M}_n = \left\{\mathbb{M}_{j,n}: j = 1, \hdots, d \right\}.
\end{equation*}
As in Section~\ref{sec:homlinmods} we assume that $X_n$ satisfies Condition~\ref{cond:X}, and define our model-specific target of inference as in Equation \eqref{eqn:targethom}. Again, we estimate the corresponding target by the model-specific ordinary-least-squares estimator in \eqref{eq:OLS}.
For variance estimation we do no longer use the estimator as defined in Subsection~\ref{sec:homlinmods}, but now take into consideration, that the observations may be heteroskedastic. Therefore, we consider an approach based on estimators suggested by \cite{eicker67}. As in Subsection~\ref{sec:homlinmods}, the variance estimators used here are not uniformly consistent due to potential model misspecification, but overestimate their targets in the sense of Subsection~\ref{sec:overest}. Furthermore, in contrast to the construction of Subsection~\ref{sec:homlinmods}, the construction of the confidence sets now needs to incorporate an upper bound for the POSI constant $K_{1-\alpha}(\corr(\Gamma_n))$, because here $\Gamma_n = \V_n[(\hat{\beta}_{\M_1,n}',\dots, \hat{\beta}_{\M_d,n}')']$, and also $\corr(\Gamma_n)$, is unobserved and can not be estimated consistently due to potential misspecification. Define for every $\mathbb{M} \in \mathsf{M}_n$ with corresponding index set $M$ the Eicker-estimator $\tilde{S}_{\M,n}$ as
\begin{equation*}
\left( X_n[M]'X_n[M] \right)^{-1} X_n[M]' \diag\left(\hat{u}^2_{1,\mathbb{M}}, \hdots, \hat{u}^2_{n,\mathbb{M}} \right) X_n[M] \left( X_n[M]'X_n[M] \right)^{-1},
\end{equation*}
where, for $y\in\R^n$, we let $\hat{u}_{\mathbb{M}}(y) = (\hat{u}_{1, \mathbb{M}}(y), \hdots, \hat{u}_{n, \mathbb{M}}(y))' = y - X_n[M] \hat{\beta}_{\mathbb{M}, n}(y)$, and denote the $j-$th diagonal entry ($j = 1, \hdots, m(\mathbb{M})$) of $\tilde{S}_{\M,n}$ by
$\hat{\sigma}^2_{j,\mathbb{M}, n}$.
Finally, given $\alpha \in (0, 1)$, we define for each $\mathbb{M} \in \mathsf{M}_n$ with corresponding index set $M$ and for every $j = 1, \hdots, m(\mathbb{M})$ the confidence sets
\begin{equation*}
\mathrm{CI}_{1-\alpha, \mathbb{M}}^{(j), \mathrm{hlm}} = \hat{\beta}^{(j)}_{\mathbb{M}, n} \pm  \sqrt{\hat{\sigma}^2_{j,\mathbb{M}, n}} B_{\alpha}(\min(k, p),k),
\end{equation*}
with $k = \sum_{\M\in\mathsf{M}_n} m(\mathbb{M})$, and where $B_\alpha$ is defined at the end of Section~\ref{sec:overest}.

Note, similarly as in Subsection~\ref{sec:homlinmods} above, that up to the choice of the last multiplicative factor $B_{\alpha}(\min(k, p),k)$, an upper bound for the corresponding POSI-constant, this is just the usual confidence interval for the $j$-th coordinate of the coefficient vector one would typically use in practice working with heteroskedastic linear models by following the naive way of ignoring the data-driven model selection step. Our construction delivers an adjustment to that approach, which turns it, regardless of the (measurable) model selection procedure applied, into an asymptotically valid statistical procedure. The main result of this subsection is as follows:
\begin{theorem}\label{thm:hetlinmods}
Let $\alpha \in (0, 1)$, $\delta > 0$ and $\tau \ge 1$, suppose Condition~\ref{cond:X} holds, and let $\hat{\mathbb{M}}_n$ be a (measurable) model selection procedure, i.e., a measurable map from the sample space $\R^n$ to $\mathsf{M}_n$. Then
\begin{equation*}
\liminf_{n \to \infty} \inf_{\mathbb{P}_n \in \mathbf{P}_n^{(\mathrm{het})} 
\left(\delta, \tau \right) } \mathbb{P}_n \left(\beta_{\hat{\mathbb{M}}_n, n}^{* (j)} \in \mathrm{CI}_{1-\alpha, \hat{\mathbb{M}}_n}^{(j), \mathrm{hlm}} \text{ for all } j = 1, \hdots, m(\hat{\mathbb{M}}_n)  \right) \geq 1-\alpha.
\end{equation*}
\end{theorem}


\subsection{Inference post-model-selection when fitting binary regression models to binary data}
\label{sec:binary}

The feasible sets $\mathbf{P}_n^{(\mathrm{bin})}(\tau)$ for $\P_n$ we consider here depend on a parameter $\tau \in (0,1/4)$ and are defined as follows: the distribution of a random vector $Y_n = (Y_{1,n}, \hdots, Y_{n,n})'$ is an element of $\mathbf{P}_n^{(\mathrm{bin})}(\tau)$ if and only if the $n$ coordinates of $Y_n$ are independent, each coordinate $Y_{i,n}$ takes on either $0$ or $1$, and $\V(Y_{i,n})\ge \tau$.
%
%
%
We consider a situation where binary regression models are fit to binary data generated under one of the elements $\P_n\in \mathbf{P}_n^{(\mathrm{bin})}(\tau)$. It is important to point out, however, that unlike other work on misspecified binary regression (e.g., \cite{Ruud83, Kub17}), we here do not assume that the true data generating process $\P_n$ is itself a binary regression model, but we consider the non-parametric case where every observation $Y_{i,n}$ may have its own success rate $p_{i,n} = \P(Y_{i,n}=1)$, with the only restriction that $\V(Y_{i,n}) = p_{i,n}(1-p_{i,n})\ge \tau$. In binary regression the maintained \emph{modeling assumption} is that the probability of a success on the $i$-th observation ($Y_{i,n}=1$), or equivalently its expectation, is given by $h(X_{i,n}\beta)$, for some $\beta\in\R^p$, some response function $h:\R\to (0,1)$ and where $X_{i,n}$ is the $i$-th row of a design matrix $X_n\in\R^{n\times p}$. Usually, when $h$ is invertible, $h^{-1}$ is called the link function. Thus, unlike the previous two examples, here we also have to make a choice for the response function $h$, in addition to selecting variables from $X_n$. Classical choices are the logit and the probit functions, but we allow also for other choices of response functions $h$, as long as they belong to a finite set $\mathcal H = \{h_1,\dots, h_{d_1}\}$ of potential candidates, that does not depend on $n$. Together with the collection $\mathcal I = \{M_1,\dots, M_{d_2}\}\subseteq 2^{\{1,\dots, p\}}\setminus\varnothing$ of candidate regressor subsets, we can define for every $j_1\in\{1,\dots, d_1\}$ and $j_2\in\{1,\dots, d_2\}$ a candidate binary regression model $\mathbb M_{(j_1,j_2),n}$ as follows: the distribution of a random vector $z = (z_1, \hdots, z_n)'$ is an element of $\mathbb M_{(j_1,j_2),n}$ if and only if the $n$ coordinates of $z$ are independent, each coordinate $z_i$ takes on either $0$ or $1$, and there exists a $\beta \in \R^{|M_{j_2}|}$ so that the mean of $z_i$ equals $h_{j_1}(X_{i,n}[M_{j_2}]\beta)$ for $i = 1, \hdots, n$.
%
Thus, our candidate set of size $d = d_1 \cdot d_2$ is given by
\begin{equation*}
\mathsf{M}_n = \left\{\mathbb{M}_{(j_1,j_2),n}: j_1 \in\{1, \hdots, d_1\}, j_2 \in \{1,\dots, d_2\} \right\}.
\end{equation*}

We need to impose some regularity conditions on the possible response functions $h\in\mathcal H$ and the design $X_n$.

\renewcommand{\thecond}{X2}
\begin{cond}\label{cond:X2} 
Let $C>0$ be fixed. Eventually, we have
\begin{enumerate}[(i)]
\item\label{cX2:rank} $\rank(X_n) = p$;
\item\label{cX2:Huber} $\max_{i = 1, \hdots, n} X_{i,n}[M] \left( X_{n}[M]'X_{n}[M]  \right)^{-1} X_{i,n}[M]' \le C/n$, for every $M \in \mathcal{I}$;
\item\label{cX2:eigen} $\lmax(X_n'X_n)/\lmin(X_n'X_n) \le C$;
\end{enumerate}
\end{cond}

\renewcommand{\thecond}{H}
\begin{cond}\label{cond:H} 
The elements $h\in\mathcal H$ have the following properties:
\begin{enumerate}[(i)]
\item\label{cH:cdf} $h:\R\to(0,1)$ is a continuous cumulative distribution function;
\item\label{cH:concave} The functions $\phi_1(\gamma) := \log(h(\gamma))$ and $\phi_2(\gamma) := \log(1-h(\gamma))$ are strictly concave on $\R$;
\item\label{cH:C2} $h$ is twice continuously differentiable and $\phi_1$ and $\phi_2$ have strictly negative second derivatives on $\R$;
\item\label{cH:Hdot} The derivative $\dot{h}$ of $h$ is strictly positive on $\R$;
\end{enumerate}
\end{cond}

\begin{remark}
Condition~\ref{cond:X2} is a strengthened version of Condition~\ref{cond:X}. It is still satisfied if $\|X_{i,n}\|$ is bounded and $\lmin(\frac{1}{n}X_n'X_n)$ is bounded away from $0$, as is typically the case in factorial designs. Note, however, that Condition~\ref{cond:X2} is invariant under scaling of $X_n$, so that, in particular, it does not require that $\lmin(X_n'X_n)\to\infty$, a condition commonly used to prove consistency of the MLE. Furthermore, Condition~\ref{cond:X2}\eqref{cX2:Huber} is implied by 
$$
X2(ii')\quad\max_{i=1,\dots, n} X_{i,n}(X_n'X_n)^{-1}X_{i,n}' \le C/n.
$$
\end{remark}

\begin{remark}
The Conditions~\ref{cond:H}\eqref{cH:cdf} and \ref{cond:H}\eqref{cH:Hdot} are rather natural and essential for parameter identification. Condition~\ref{cond:H}\eqref{cH:C2} is also classical and used to ensure continuity of the Hessian of the log-likelihood \citep[cf.][]{Fahrmeir85, Fahrmeir90}. Finally, Condition~\ref{cond:H}\eqref{cH:concave}, which is implied by Condition~\ref{cond:H}\eqref{cH:C2}, ensures strict concavity of the log-likelihood, which, in turn, guarantees uniqueness of pseudo parameters and the MLE (see Lemma~\ref{lemma:pseudo} and Lemma~\ref{lemma:unifCons} below). 
It is easy to see that Condition~\ref{cond:H} is satisfied, e.g., for response functions corresponding to the classical logit, probit, log-log and complementary log-log link functions discussed in \citet[][p.108]{McCullagh89}.
\end{remark}

Note that since the design matrix $X_n\in\R^{n\times p}$ is fixed, a candidate model $\mathbb M \in \mathsf{M}_n$ can be identified with a pair $\mathbb M \triangleq (h, M) \in \mathcal H\times \mathcal I$. Estimating the parameter $\beta\in\R^{|M|}$ of a candidate model $\M\in\mathsf{M}_n$ is usually done by numerically maximizing the likelihood. The (quasi-)log-likelihood function for model $\M \triangleq (h,M)$ can be expressed as
\begin{equation*}
\ell_{\M,n}(y,\beta) \;=\; \sum_{i=1}^n \left[ y_i \phi_1(X_{i,n}[M]\beta) + (1-y_i)\phi_2(X_{i,n}[M]\beta)\right],
\end{equation*}
where $\phi_1(\gamma) = \log h(\gamma)$ and $\phi_2(\gamma) = \log (1-h(\gamma))$, and $y = (y_1,\dots, y_n)'\in\{0,1\}^n$, $\beta\in\R^{|M|}$. Whenever Condition~\ref{cond:H}\eqref{cH:C2} holds, we denote the matrix of negative second derivatives of $\ell_{\M,n}$ by
$$
H_{\M,n}(y,\beta) \;=\; -\frac{\partial^2\ell_{\M,n}(y,\beta)}{\partial \beta \partial \beta'}
\;=\; X_n[M]'D_{\M,n}(y,\beta)X_n[M],
$$
where $D_{\M,n}(y,\beta)$ is a diagonal matrix with $i$-th diagonal entry equal to 
$$-y_i \ddot{\phi}_1(X_{i,n}[M]\beta) - (1-y_i)\ddot{\phi}_2(X_{i,n}[M]\beta).$$
Note that under Conditions~\ref{cond:X2}\eqref{cX2:rank} and \ref{cond:H}\eqref{cH:C2}, $H_{\M,n}(y,\beta)$ is positive definite.

As our target of inference we take the model dependent vector $\beta_{\M,n}^* \in\R^{|M|}$ that maximizes the expected log-likelihood $\beta\mapsto \E_n[\ell_{\M,n}(\cdot,\beta)]$ under the true data generating distribution $\P_n\in\mathbf{P}_n^{(\mathrm{bin})}(\tau)$. If $\beta_{\M,n}^*$ exists, then it is easy to see that it also minimizes the Kullback-Leibler divergence between the true data generating distribution $\P_n$ and the class of distributions specified by the working model $\M\in\mathsf{M}_n$. Focusing on the Kullback-Leibler minimizer has a longstanding tradition in the misspecification literature dating back at least to \citet{Huber67} (see also \citet{White82} and the references given therein). For references more specific to generalized linear models see \citet{Fahrmeir90} and \citet{lv14model}. 
That this target uniquely exists in the present context of binary regression is the subject of the following lemma.\footnote{A similar claim is made in Theorem~5 of \cite{lv14model} and its proof is deferred to Version~1 of the arXiv preprint \cite{lv10v1}, where it appears to be the case that the \emph{existence} issue has been ignored. For a complete proof of our Lemma~\ref{lemma:pseudo} see Section~\ref{sec:pseudo:proof} of the supplement.}

\begin{lemma} \label{lemma:pseudo}
Suppose that $\rank(X_n)=p$ and \ref{cond:H}(\ref{cH:cdf},\ref{cH:concave}) hold. Then, for every $\M\in\mathsf{M}_n$ and for every $\P_n\in\bigcup_{\delta>0}\mathbf{P}_n^{(\mathrm{bin})}(\delta)$, there exists a unique vector $\beta_{\M,n}^* = \beta_{\M,n}^*(\P_n)\in\R^{m(\M)}$, such that 
\begin{equation*}
\int_{\R^n} \ell_{\M,n}(y,\beta_{\M,n}^*(\P_n)) \,d\P_n(y) \;=\; \sup_{\beta\in\R^{m(\M)}} \int_{\R^n} \ell_{\M,n}(y,\beta) \,d\P_n(y).
\end{equation*}
\end{lemma}

Furthermore, it is well known that for some points in the sample space $\{0,1\}^n$ the MLE in the binary regression model does not exist \citep[see, e.g.,][]{Wedderburn76}. But those samples have vanishing asymptotic probability. The following lemma establishes this asymptotic existence of the (quasi-) MLE $\hat{\beta}_{\M,n}$ in the present setting, along with uniform consistency. Its proof is deferred to Section~\ref{sec:unifCons} of the appendix.

\begin{lemma}\label{lemma:unifCons}
Suppose that Conditions~\ref{cond:X2}(\ref{cX2:rank},\ref{cX2:Huber}) and \ref{cond:H}(\ref{cH:cdf},\ref{cH:concave},\ref{cH:C2}) hold and fix $\tau\in(0,1/4)$. Then, for every $n\in\N$, every $\M\in\mathsf{M}_n$ and every $\P_n\in\mathbf{P}_n^{(\mathrm{bin})}(\tau)$, there exists a function $\hat{\beta}_{\M,n}:\{0,1\}^n\to\R^{m(\M)}$ (depending only on $n$ and $\M$) and a set $E_{\M,\P_n,n}\subseteq \{0,1\}^n$, such that
\begin{align*}
\ell_{\M,n}\left(y, \hat{\beta}_{\M,n}(y) + \beta\right) \;<\; \ell_{\M,n}\left(y, \hat{\beta}_{\M,n}(y)\right) \quad \forall y\in\ E_{\M,\P_n,n},\;\forall \beta\neq 0
\end{align*}
and $$\inf_{\M\in\mathsf{M}_n}\inf_{\P_n\in\mathbf{P}_n^{(\mathrm{bin})}(\tau)} \P_n(E_{\M,\P_n,n}) \;\xrightarrow[n\to \infty]{} \; 1.$$
Moreover, for the pseudo parameter $\beta_{\M,n}^*\in\R^{m(\M)}$ of Lemma~\ref{lemma:pseudo}, we have
$$
\limsup_{n\to\infty}\sup_{\substack{\M\in\mathsf{M}_n\\\P_n\in\mathbf{P}_n^{(\mathrm{bin})}(\tau)}}
\P_n\left( 
	\left\| (X_n[M]'X_n[M])^{1/2}(\hat{\beta}_{\M,n} - \beta^*_{\M,n}(\P_n))\right\| > \delta\right)\; \to\; 0,
$$
as $\delta\to\infty$.
\end{lemma}

To construct asymptotically valid confidence intervals for the components of $\beta_{\M,n}^*$, we need an estimate of the asymptotic covariance matrix of $\hat{\beta}_{\M,n}$. In the misspecified setting it is usually not possible to obtain a consistent estimator. We here follow the suggestion of \citet[][p. 491]{Fahrmeir90} who proposed a sandwich-type estimator for misspecified generalized linear models. This estimator fits with the general idea of Section~\ref{sec:overest}. For $\M\in\mathsf{M}_n$, $\M \triangleq (h,M)$, define \noeqref{eq:tildeSlogistic}
\begin{align}\label{eq:tildeSlogistic}
\tilde{S}_{\M,n} = \hat{H}_{\M,n}^{-1}X_n[M]'
	\diag\left(\hat{u}^2_{1,\mathbb{M}}, \hdots, \hat{u}^2_{n,\mathbb{M}} \right) 
	X_n[M] \hat{H}_{\M,n}^{-1},
\end{align} 
where $\hat{H}_{\M,n}(y) = H_{\M,n}(y,\hat{\beta}_{\M,n}(y))$,  
$$
\hat{u}_{i,\mathbb{M}}(y) = \frac{\dot{h}(\hat{\gamma}_{i,n,M}(y))}{h(\hat{\gamma}_{i,n,M}(y))(1-h(\hat{\gamma}_{i,n,M}(y)))}\left(y_i - h(\hat{\gamma}_{i,n,M}(y)) \right)
$$
and $\hat{\gamma}_{i,n,M}(y) = X_{i,n}[M]\hat{\beta}_{\M,n}(y)$, and denote the $j-$th diagonal entry ($j = 1, \hdots, m(\mathbb{M})$) of $\tilde{S}_{\M,n}$ by
\begin{equation}\label{eq:Sbinary}
\hat{\sigma}^2_{j,\mathbb{M}, n}.
\end{equation}
Finally, given $\alpha \in (0, 1)$, we define for each $\mathbb{M} \in \mathsf{M}_n$ and for every $j = 1, \hdots, m(\mathbb{M})$ the confidence sets
\begin{equation*}
\mathrm{CI}_{1-\alpha, \mathbb{M}}^{(j), \mathrm{bin}} = \hat{\beta}^{(j)}_{\mathbb{M}, n} \pm  \sqrt{\hat{\sigma}^2_{j,\mathbb{M}, n}} B_{\alpha}(\min(k, n),k),
\end{equation*}
with $k = \sum_{\M\in\mathsf{M}_n} m(\mathbb{M})$, and where $B_\alpha$ is as defined at the end of Section~\ref{sec:overest}.

These confidence intervals have the same basic structure as in Section~\ref{sec:hetlinmods}, in the sense that they use estimators $\hat{\sigma}_{j,\M,n}^2$ for the asymptotic variances that consistently overestimate their respective target quantities and replace the usual Gaussian quantile by the correction constant $B_{\alpha}(\min(k, n),k)$ that adjusts for the effect of model selection. This leads to asymptotically valid inference post-model-selection, as stated in the following theorem. 

\begin{theorem}\label{thm:binary}
Let $\alpha \in (0, 1)$ and $\tau \in(0,1/4)$, suppose Conditions~\ref{cond:X2} 
and \ref{cond:H} 
hold, and let $\hat{\mathbb{M}}_n$ be a model selection procedure, i.e., a map from the sample space $\{0,1\}^n$ to $\mathsf{M}_n$. Then
\begin{equation*}
\liminf_{n \to \infty} \inf_{\mathbb{P}_n \in \mathbf{P}_n^{(\mathrm{bin})} 
\left(\tau \right) } \mathbb{P}_n \left(\beta_{\hat{\mathbb{M}}_n, n}^{* (j)} \in \mathrm{CI}_{1-\alpha, \hat{\mathbb{M}}_n}^{(j), \mathrm{bin}} \;\forall j = 1, \hdots, m(\hat{\mathbb{M}}_n)  \right) \geq 1-\alpha.
\end{equation*}
\end{theorem}

\begin{remark}\label{rem:canonical:response:function}
It is important to note that if one decides a priori to use only the canonical link function, which, in the present case of binary regression, corresponds to the logistic response function $h^{(c)}(\gamma) := e^\gamma/(1+e^\gamma)$, then Theorem~\ref{thm:binary} holds with the POSI-constant $B_\alpha(\min(k,n),k)$ decreased to $B_\alpha(\min(k,p),k)$. See Corollary~\ref{corr:binaryCanoncial} in Section~\ref{sec:binaryCanonical} of the supplement.
\end{remark}

\begin{remark}
We point out that similar principles used to derive Theorem~\ref{thm:binary} can also be employed to treat other quasi-maximum likelihood or general  M-, and Z-estimation problems \citep[see, e.g.,][for a more general treatment of generalized linear models]{Fahrmeir90}. 
The general theory of M- and Z-estimation as presented, e.g., in \citet[][Sections~3.2 and 3.3]{Vaart96}, usually also leads to expansions of the form required in Conditon~\ref{cond:sum} \citep[cf.][Theorem~3.2.16 and Theorem~3.3.1]{Vaart96}. These results are stated in a pointwise fashion but can be made uniform over large classes of data generating processes by using ideas from Section~2.8 of the same reference.
However, in more specific examples, such as the present binary regression setting, conditions can be directly imposed on the design and the link functions and can be optimized for this setup.
\end{remark}

The main technical difference compared to the previous two examples is a non-trivial existence and uniqueness issue of, both, the target parameters as well as the estimators. The main conceptual difference is that here a data driven model selection procedure $\hat{\M}_n$ may not only select variables among the $p$ candidate regressors in $X_n$, but may also result in a choice of a response function $h$ from some pre-specified class $\mathcal H$. In practice, $\mathcal H$ often contains certain classical candidates such as, e.g., the response functions corresponding to the logit, probit or complementary log-log link function. A working model could then be selected, for instance, by minimizing some penalized (quasi-)likelihood criterion over all possible choices of $(h,M)\in\mathcal H\times \mathcal I$. However, we emphasize once more, that the specifics of the possibly data driven model selection procedure $\hat{\M}_n$ are completely inconsequential for the validity of our proposed confidence intervals and could also involve visual inspection of the data and subjective preferences.

\section{Simulation study}\label{sec:sim}

In this section, we present the main findings of an extensive simulation study, the details of which can be found in Section~\ref{app:sim} of the appendix. 

\subsection{Comparison with \cite{Tib15a}}
\label{sec:sim:Tib}

For linear homoskedastic models, we first address the least angle regression (LAR) model selector \cite{efron04least} and compare the ``POSI'' confidence intervals of Theorem~\ref{thm:homlinmods} with the ``TG'' (truncated Gaussian) intervals developed in \cite{Tib15a} (with the plug-in approach for $\sigma^2$). The latter intervals are specifically tailored for the LAR model selector. We consider $n_{step} = 3$ model selectors $\hat{\mathbb{M}}_n^{(1)},\hat{\mathbb{M}}_n^{(2)},\hat{\mathbb{M}}_n^{(3)}$, that are obtained from the LAR algorithm. To compute $\hat{\M}_n^{(k)}(y)$, for $k=1,2,3$, we run $k$ steps of the LAR algorithm, i.e., $\hat{\M}_n^{(k)}$ always selects exactly $k$ variables. As in \cite{Tib15a}, we seek inference for the variable that is selected in the final ($k$-th) step of the LAR algorithm. 
We set $n=50$, $p=10$ and repeat $N = 500$ independent repetitions of data generations, model selections and confidence interval computations. The setup is the same as in \cite{Tib15a} (see Section~\ref{app:sim} in the supplement for details). In Table~\ref{table:homolm}, we report the coverage proportions, the median lengths and the $90 \%$ quantiles of the lengths for each of the six procedures (``POSI'' and ``TG'' for $k=1,2,3$), in different settings. We also report the proportions of times where the three targets corresponding to the regressors selected after step~3 of the LAR algorithm are simultaneously contained by the three respective confidence intervals.

The ``POSI'' confidence intervals always have target-specific and simultaneous coverage above the nominal level. The coverage proportions are large, which is so because these confidence intervals offer strong guarantees: they are valid for any model selection procedure, and simultaneously over all the variables in the selected model. 
Turning to the ``TG'' confidence intervals, we observe that these intervals have coverage probabilities approximately equal to the nominal level when the three targets are considered separately but their median lengths are often larger and never much smaller than the lengths of the ``POSI'' intervals. 
Finally, the $90 \%$ quantiles are always larger for the ``TG'' intervals, for which they can be very large. In Table~\ref{table:homolm}, we sometimes report infinite $90 \%$ quantiles for the ``TG'' intervals. This is because, although the confidence intervals in \cite{Tib15a} always have finite length in theory, the numerical implementation in the \verb@R@ package \verb@selectiveInference@ can return lower or upper bounds equal to $\pm \infty$. In contrast, the confidence intervals suggested in this paper are more robust, in the sense that their $90 \%$ quantile lengths are always less than twice as large as their median lengths.

We believe that the numerical results of Table~\ref{table:homolm} favor the ``POSI'' confidence intervals suggested in this paper over the ``TG'' procedure. Indeed, we have seen that, even though the LAR model selector is used, the ``POSI'' confidence intervals have larger coverage proportions, remain valid when considered simultaneously, generally have smaller median lengths, and never exhibit very large quantile lengths. On top of this, the ``POSI'' confidence intervals are much more broadly applicable, as they have theoretical guarantees for any model selection procedure.

One needs to mention here that \cite{Tib15a} also discuss a bootstrap version of their ``TG'' intervals. These bootstrap confidence intervals have similar coverage properties as the ``TG'' intervals, but much smaller median width. Their width seems to be comparable to the width of our ``POSI'' confidence intervals (cf. Tables~1 and 2 in \cite{Tib15a}). All the advantages of the ``POSI'' method discussed in the preceding paragraph (besides the comments concerning their smaller width) also apply to the bootstrapped ``TG'' intervals. Furthermore, this suggests that our ``POSI'' intervals could potentially be improved by using suitable bootstrap methods as well. However, answering this question goes beyond the scope of the present article.

\begin{table}[h]
\begin{center}
\begin{tabular}{c|ccc|ccc|ccc|c}
$u$ & \multicolumn{3}{c}{Step 1} \vline & \multicolumn{3}{c}{Step 2} \vline & \multicolumn{3}{c}{Step 3}  & Simult. \\ 
& cov. & med. &  qua. & cov. & med. &  qua. & cov. & med. &  qua.  & cov. \\  \hline\hline
\multirow{2}{*}{N} &  1.00 & 6.58 & 7.44 & 0.99 & 6.07 & 6.98 & 0.98 & 6.05 & 7.20 & 0.95 \\ 
& 0.87 & 5.09 & 21.32 & 0.89 & 11.05 & 62.11 & 0.89 & 24.51 & Inf & 0.80 \\ \cline{2-11}
\multirow{2}{*}{L} & 1.00 & 6.52 & 7.79 & 0.99 & 5.99 & 7.39 & 0.98 & 5.99 & 7.59 & 0.95 \\
& 0.91 & 5.15 & 16.37 & 0.91 & 10.25 & 58.06 & 0.87 & 25.09 & Inf & 0.75 \\  \cline{2-11}
\multirow{2}{*}{U}  & 1.00 & 6.58 & 7.26 & 1.00 & 6.08 & 6.72 & 0.99 & 6.09 & 6.99 & 0.95 \\
& 0.90 & 5.08 & 20.46 & 0.90 & 11.93 & 56.45 & 0.90 & 25.02 & Inf & 0.77 \\  \cline{2-11}
\multirow{2}{*}{SN}  &  1.00& 6.56&  7.59 &0.99&  6.05 & 7.16 &0.97  & 6.07&  7.34 & 0.95\\
& 0.90 & 5.08 &16.33& 0.88 &11.55& 59.57& 0.89  &26.79&   Inf & 0.76 \\ \hline\hline
\multirow{2}{*}{N}  & 0.99 & 6.45 & 7.36 &1.00 & 8.43& 13.02 &1.00 & 11.27 &16.63 & 0.97 \\
& 0.91 & 7.71 & 36.29 & 0.91 & 56.00 & Inf & 0.92 & 97.82 & Inf & 0.78  \\ \cline{2-11}
\multirow{2}{*}{L}  & 1.00 & 6.26 & 7.65 & 0.99 & 8.45 & 13.24 & 1.00 & 10.65 & 16.62 & 0.97 \\
& 0.92 & 8.03 & 43.09 & 0.89 & 65.66 & Inf & 0.89 & 97.82 & Inf & 0.81 \\ \cline{2-11}
\multirow{2}{*}{U}  & 0.99 & 6.47 & 7.05 & 0.99 & 8.36 & 13.17 & 1.00 & 11.44 & 16.88 & 0.98\\
& 0.89 & 8.23 & 39.62 & 0.89 & 61.46 & Inf & 0.88 & 117.81 & Inf & 0.84\\ \cline{2-11}
\multirow{2}{*}{SN}  & 0.99 & 6.48 & 7.45 & 1.00 & 8.29 & 13.00 & 1.00 & 10.96 & 16.97 & 0.97 \\
& 0.91 & 8.36 & 33.18 & 0.92 & 55.24 & Inf & 0.90 & 97.82 & Inf & 0.81 \\ 
\end{tabular}%
\end{center}
\caption{Coverage proportion (cov.), median length (med.) and $90 \%$ quantile length (qua.) for the ``POSI'' and ``TG'' confidence intervals at nominal level $1-\alpha=0.9$. The design matrix is generated with independent (upper half of the table) or correlated (lower half) columns. The errors $u$ have normal (N), Laplace (L), uniform (U) or skewed normal (SN) distributions. For each setting, the ``POSI'' intervals correspond to the first row and the ``TG'' intervals correspond to the second row. The coverage proportions, median and quantile lengths are given for each of the three targets for the three first steps of the LAR algorithm. The last column provides the simultaneous coverage proportion of the three targets after step~3.}
\label{table:homolm}
\end{table}

\subsection{The case of `significance hunting'}

Furthermore, we investigate a model selection procedure which we call ``significance hunting'' and which is closely related to the SPAR procedure in \cite{Berk13}. We first sort all the possible candidate models $\mathbb{M}\in\mathsf M_n$ according to their penalized log-likelihood and then select the model $\M$ and index $j$ that maximize the test statistics
\[
\frac{
\left| \hat{\beta}^{(j)}_{\mathbb{M},n} \right|
}{
  \sqrt{\hat{\sigma}^2_{\mathbb{M},n} \left[\left( X_n[M]'X_n[M]  \right)^{-1}\right]_j}
},
\]
among the $n_{best}$ models with largest penalized log-likelihood.
We set $n=100$, $p=5$, $1-\alpha=0.9$ and consider two settings for $\beta$. In the ``zero'' setting, we set $\beta = (0,...,0)'$. In the ``non-zero'' setting, we set $\beta = (2,-1,0,0,1)'$. We consider the values $n_{best}=5$ and $n_{best} = 20$.
The errors are normally generated. In Table~\ref{table:homolm:significance:hunting}, the coverage proportions are significantly lower than in Table~\ref{table:homolm}, and closer to the nominal level. Hence, the confidence intervals suggested in this paper may have conservative coverage proportions for some model selection procedures (such as LAR) but this is somehow necessary, since there exist other model selection procedures (such as ``significance hunting'') for which the coverage proportions are close to the nominal level.

\begin{table}[tbp]
\begin{center}
\begin{tabular}{c|c|ccc}
$n_{best}$ & $\beta$ & cov. & med. & qua. \\ \hline
\multirow{2}{*}{20} & zero & 0.88 & 4.99 & 5.83  \\ \cline{2-5}
 & non-zero & 0.93 & 5.00 & 5.73  \\ \hline
\multirow{2}{*}{5} & zero &  0.92 & 4.87 & 5.35  \\ \cline{2-5}
 & non-zero & 0.94 & 4.94 & 5.41
\end{tabular}%
\end{center}
\caption{Coverage proportion (cov.), median lengths (med.) and $90 \%$ quantile lengths (qua.) of the ``POSI'' confidence intervals at level $1-\alpha=0.9$ for the ``significance hunting'' model selection procedure.}
\label{table:homolm:significance:hunting}
\end{table}

\subsection{Further results}

In Section~\ref{app:sim} of the appendix we provide all the details of the previous simulations as well as further discussions of the results. We also present simulations for the binary regression problem of Section~\ref{sec:binary}, comparing our methods to a procedure suggested by \cite{taylor17post} and to naive intervals that ignore the data driven model selection step. Furthermore, we investigate the effect of misspecification and we also consider the ``significance hunting'' procedure in the binary regression case. The overall picture is similar to the results for the linear model, with the additional aspect that the ``POSI'' intervals remain valid also under misspecification, whereas the coverage probabilities of the methods of, e.g., \cite{taylor17post} can be substantially below the nominal level in that case.

\section{Conclusion}\label{sec:concl}

We have presented a general theory for the construction of asymptotically valid confidence sets post-model-selection. Our methods can be used in a wide number of situations, because they are only based on a standard representation that can often be obtained by simple linearization arguments. We have also applied our theory to construct valid confidence sets after selecting and fitting fixed design linear models to (possibly non-Gaussian) homoskedastic or heteroskedastic data. Moreover, we have investigated the practically very important case when binary regression models are fit to binary data. In this case, in addition to selecting variables from a given design matrix, also the choice of an appropriate link function can be made in a data driven way. The general theory and the proposed methods are applicable irrespective of whether any of the candidate models under consideration is correctly specified, leading to more or less conservative inference depending on the severity of misspecification (see Remark~\ref{rem:adaptive}). This feature is also present in the applications of Section~\ref{sec:homlinmods} (see Subsection~\ref{sec:homlinmods:varest}), \ref{sec:hetlinmods} and \ref{sec:binary}. In simulation experiments we have illustrated that the confidence intervals constructed in the examples compare favorably to existing procedures (typically offering higher coverage with the confidence intervals having comparable or much smaller length), even though they are not tailored towards specific model selection procedures.

Open questions that go beyond the scope of this article, but are currently under investigation, include the extension of the approach discussed here to dependent data; the applicability and performance of bootstrap procedures; and the development of procedures in the spirit of \cite{Berk13} in the challenging situation when the number of models fitted can grow with sample size. In ongoing work we apply our methods to real data and investigate if they can prevent spurious findings while detecting true reproducible effects.

\section*{Acknowledgements}

Results related to the present article were presented in the Statistics and Econometrics Research Seminar at the Department of Statistics and Operations Research at the University of Vienna, and we would like to thank the participants, in particular Hannes Leeb, Benedikt M. P\"otscher and Ulrike Schneider, for helpful comments and suggestions. We are also grateful for the comments and suggestions of two anonymous referees who helped to produce a considerably improved version of the paper.

\appendix

\section{Simulation study}\label{app:sim}

In this section, we investigate the confidence intervals suggested in this paper in a numerical study. It is an extended and more detailed version of Section~\ref{sec:sim} in the main article. We consider linear models and binary regression. When studying linear models, we first address the least angle regression (LAR) model selector \cite{efron04least} and compare the confidence intervals of Theorem~\ref{thm:homlinmods} with those developed in \cite{Tib15a}. The latter intervals are specifically tailored for the LAR model selector.
Then, we investigate a model selection procedure which we call ``significance hunting'' and which is arguably representative of a certain practice of data mining.

When studying binary regression, we consider the lasso model selector, with a fixed regularization parameter $\lambda$. We compare the confidence intervals of Theorem \ref{thm:binary} with those suggested by \cite{taylor17post} and with ``naive'' confidence intervals.  
The confidence intervals of \cite{taylor17post} are specific to the lasso model selector. The ``naive'' confidence intervals ignore the model selection step. We also investigate the significance hunting procedure in this case.

\subsection{Linear models}

We study the setting of Section \ref{sec:homlinmods}, where linear models are fit to homoskedastic data. Furthermore, we address the well-specified case, where the true data generating process corresponds to one of the candidate models.
We consider observations of $Y_n = X_n \beta + \sigma u $, where $X_n$ is an $n \times p$ matrix (which will be randomly generated in the simulations), $\beta$ is a $p \times 1$ vector, $\sigma$ is positive and $u$ is an $n \times 1$ vector with independent and identically distributed components which is also independent of $X_n$. 
For each model $\mathbb{M}$ with index set $M \subset \{1,...,p\}$, the $|M| \times 1$ target of inference $\beta_{\mathbb{M},n}^*$ is given by \eqref{eqn:targethom} with $ \mu_n$ replaced by $X_n \beta$.

\subsubsection{Comparison with the confidence intervals of \cite{Tib15a}}

We consider $n_{step} = 3$ model selectors $\hat{\mathbb{M}}_n^{(1)},\hat{\mathbb{M}}_n^{(2)},\hat{\mathbb{M}}_n^{(3)}$, that are obtained from the LAR algorithm (\cite{efron04least} with the function \verb@lar@ of the \verb@R@ package \verb@lars@). To compute $\hat{\M}_n^{(k)}(y)$, for $k=1,2,3$, we run $k$ steps of the LAR algorithm, i.e., $\hat{\M}_n^{(k)}$ always selects exactly $k$ variables. As in \cite{Tib15a}, we seek inference for the variable that is selected in the final ($k$-th) step of the LAR algorithm. That is, when using $\hat{\M}_n^{(k)}$, the target of inference considered by \cite{Tib15a} is $\beta_{\hat{\mathbb{M}}_n^{(k)},n}^{* (\hat{j}_k)}$, where $1\le\hat{j}_k\le k$ is a data dependent index. 
We compare the following two confidence intervals. First, we consider  the interval $\mathrm{CI}_{1-\alpha, \hat{\mathbb{M}}_n^{(k)}}^{(\hat{j}_k), \mathrm{lm}}$ suggested in Theorem~\ref{thm:homlinmods}, which we call ``POSI''. Note that by Theorem~\ref{thm:homlinmods} the intervals $\mathrm{CI}_{1-\alpha, \hat{\mathbb{M}}_n^{(k)}}^{(j), \mathrm{lm}}$, for $j=1,\dots, k$, simultaneously cover the respective coordinates $\beta_{\hat{\mathbb{M}}_n^{(k)},n}^{* (j)}$, for $j=1,\dots, k$, with (asymptotic) probability not smaller than $1-\alpha$. A fortiori the ``POSI'' interval  $\mathrm{CI}_{1-\alpha, \hat{\mathbb{M}}_n^{(k)}}^{(\hat{j}_k), \mathrm{lm}}$ covers $\beta_{\hat{\mathbb{M}}_n^{(k)},n}^{* (\hat{j}_k)}$ with (asymptotic) probability not smaller than $1-\alpha$.
Secondly, we consider the interval $\mathrm{CI}_{1-\alpha, \hat{\mathbb{M}}_n^{(k)}}^{(\hat{j}_k), \mathrm{TG}}$, which we call ``TG'' (truncated Gaussian), and which is suggested in \cite{Tib15a} (Section 2.4 with the plug-in approach for $\sigma^2$). This confidence interval also has asymptotic validity properties, similar to Theorem~\ref{thm:homlinmods}, but under the important restriction that $\hat{\mathbb{M}}_n$ has to be obtained from the LAR procedure. [In \cite{Tib15a}, similar intervals are developed for the lasso or forward-stepwise procedures.] To compute the ``TG'' intervals, we have used the function \verb@larInf@ of the \verb@R@ package \verb@selectiveInference@.

We compare the two confidence intervals above in a simulation study conducted as follows. We set $n=50$, $p=10$, $\sigma=1$, $1-\alpha=0.9$ and $\beta_0 = (-4,4,0,...,0)'\in\R^p$. We repeat $N = 500$ independent repetitions of data generations, model selections and confidence interval computations. To generate a matrix $X_n$ we consider two different cases. In the ``independent'' case, we sample each column of $X_n$ independently. With probability $1/3$, each column is filled with independent entries from either a normal $(0,1)$, a Bernoulli $(1/2)$, or a skewed normal $(0,1,5)$ distribution. Then, each column is normalized to have unit Euclidean norm. In the  ``correlated'' case, we first generate each row of $X_n$ independently from a Gaussian distribution with mean vector $0$ and covariance matrix $( e^{-0.1 |i-j|} )_{1 \leq i,j \leq p}$. Then, each column is normalized to have unit Euclidean norm.
Once $X_n$ is sampled, $Y_n$ is generated by independently sampling the components of $u$ from a normal, Laplace, uniform or skewed normal (with shape parameter $5$) distribution. In each case, the error distribution has mean $0$ and variance $1$. We remark that this data generation setting, in the ``independent'' case for $X_n$, is the same as the one considered in Table~1 of \cite{Tib15a}, where the ``TG'' confidence intervals are numerically investigated. The only difference is that we resample $X_n$ at each of the $500$ steps, while a single realization of $X_n$ is kept throughout the simulation study in \cite{Tib15a}.

For each realization of $X_n$ and $Y_n$, and for every $k=1,2,3$, we compute $\hat{\M}_n^{(k)}$, $\hat{j}_k$, and the corresponding target, as described above. Then, for the two confidence intervals
$\mathrm{CI}_{1-\alpha, \hat{\mathbb{M}}_n^{(k)}}^{(\hat{j}_k), \mathrm{lm}}$ and $\mathrm{CI}_{1-\alpha, \hat{\mathbb{M}}_n^{(k)}}^{(\hat{j}_k), \mathrm{TG}}$, we record the length and whether the target is covered or not.
We also record whether, simultaneously, the three targets $\beta_{\hat{\mathbb{M}}_n^{(3)},n}^{*(j)}$, $j=1,2,3$, corresponding to step~3 of the LAR algorithm, belong to their respective confidence intervals.

In Table \ref{table:homolm}, we report the coverage proportions, the median lengths and the $90 \%$ quantiles of the lengths for each of the six procedures (``POSI'' and ``TG'' for $k=1,2,3$), in different settings. We also report the proportions of times where the three targets are simultaneously contained by the three confidence intervals. We first observe that the results are approximately the same for the four types of error distributions, so that the Gaussian asymptotic approximation is accurate for these values of $n,p$. The ``POSI'' confidence intervals always have target-specific and simultaneous coverage above the nominal level. The coverage proportions are large, which is so because these confidence intervals offer strong guarantees: they are valid for any model selection procedure, and simultaneously over all the variables in the selected model. 

Turning to the ``TG'' confidence intervals, we observe that these intervals have coverage probabilities approximately equal to the nominal level when the three targets are considered separately. This is in agreement with the asymptotic guarantees obtained in \cite{Tib15a}. However, the simultaneous coverage is between $0.75$ and $0.84$ and thus always below the nominal level ($0.90$). In \cite{Tib15a}, no asymptotic results are given concerning simultaneous coverage. This can be a practical limitation. If one wished to use the ``TG'' confidence intervals simultaneously they would have to increase their lengths, for instance by a Bonferroni correction.  

We observe in Table \ref{table:homolm} that the confidence intervals we suggest in this paper have slightly larger median length than the ``TG'' intervals (by about $20 \%$) only in the ``independent'' design case and for the first step of the LAR procedure. In all the other cases, the ``POSI'' confidence intervals have smaller median lengths than the ``TG'' intervals. The difference of median lengths in these cases can be very significant. For instance, in the ``correlated'' design case, at the third step of LAR, the median length for the ``TG'' intervals is about $9$ times as large as for the ``POSI'' ones.  

Finally, the $90 \%$ quantiles are always larger for the ``TG'' intervals, for which they can be very large. In Table~\ref{table:homolm}, we sometimes report infinite $90 \%$ quantiles for the ``TG'' intervals. This is because, although the confidence intervals in \cite{Tib15a} (in the two-sided case as is considered here) always have finite length in theory, the numerical implementation in the \verb@R@ package \verb@selectiveInference@ can return lower or upper bounds equal to $\pm \infty$. [When, say, the lower bound is equal to $- \infty$ and the upper bound is larger than the target value, we consider the target to be covered.] In contrast, the confidence intervals suggested in this paper are more robust, in the sense that their $90 \%$ quantile lengths are always less than twice as large as their median lengths.

We believe that the numerical results of Table~\ref{table:homolm} favor the ``POSI'' confidence intervals suggested in this paper over the ``TG'' procedure. Indeed, we have seen that, even though the LAR model selector is used, the ``POSI'' confidence intervals have larger coverage proportions, remain valid when considered simultaneously, generally have smaller median lengths, and never exhibit very large quantile lengths. On top of this, the ``POSI'' confidence intervals are significantly more broadly applicable, as they have theoretical guarantees for any model selection procedure.

One needs to mention here that \cite{Tib15a} also discuss a bootstrap version of their ``TG'' intervals. These bootstrap confidence intervals have similar coverage properties as the ``TG'' intervals, but much smaller median width. Their width seems to be comparable to the width of our ``POSI'' confidence intervals (cf. Tables~1 and 2 in \cite{Tib15a}). All the advantages of the ``POSI'' method discussed in the preceding paragraph (besides the comments concerning their smaller width) also apply to the bootstrapped ``TG'' intervals. Furthermore, this suggests that our ``POSI'' intervals could potentially be improved by using suitable bootstrap methods as well. However, answering this question goes beyond the scope of the present article.

\begin{table}[h]
\begin{center}
\begin{tabular}{c|ccc|ccc|ccc|c}
$u$ & \multicolumn{3}{c}{Step 1} \vline & \multicolumn{3}{c}{Step 2} \vline & \multicolumn{3}{c}{Step 3}  & Simult. \\ 
& cov. & med. &  qua. & cov. & med. &  qua. & cov. & med. &  qua.  & cov. \\  \hline\hline
\multirow{2}{*}{N} &  1.00 & 6.58 & 7.44 & 0.99 & 6.07 & 6.98 & 0.98 & 6.05 & 7.20 & 0.95 \\ 
& 0.87 & 5.09 & 21.32 & 0.89 & 11.05 & 62.11 & 0.89 & 24.51 & Inf & 0.80 \\ \cline{2-11}
\multirow{2}{*}{L} & 1.00 & 6.52 & 7.79 & 0.99 & 5.99 & 7.39 & 0.98 & 5.99 & 7.59 & 0.95 \\
& 0.91 & 5.15 & 16.37 & 0.91 & 10.25 & 58.06 & 0.87 & 25.09 & Inf & 0.75 \\  \cline{2-11}
\multirow{2}{*}{U}  & 1.00 & 6.58 & 7.26 & 1.00 & 6.08 & 6.72 & 0.99 & 6.09 & 6.99 & 0.95 \\
& 0.90 & 5.08 & 20.46 & 0.90 & 11.93 & 56.45 & 0.90 & 25.02 & Inf & 0.77 \\  \cline{2-11}
\multirow{2}{*}{SN}  &  1.00& 6.56&  7.59 &0.99&  6.05 & 7.16 &0.97  & 6.07&  7.34 & 0.95\\
& 0.90 & 5.08 &16.33& 0.88 &11.55& 59.57& 0.89  &26.79&   Inf & 0.76 \\ \hline\hline
\multirow{2}{*}{N}  & 0.99 & 6.45 & 7.36 &1.00 & 8.43& 13.02 &1.00 & 11.27 &16.63 & 0.97 \\
& 0.91 & 7.71 & 36.29 & 0.91 & 56.00 & Inf & 0.92 & 97.82 & Inf & 0.78  \\ \cline{2-11}
\multirow{2}{*}{L}  & 1.00 & 6.26 & 7.65 & 0.99 & 8.45 & 13.24 & 1.00 & 10.65 & 16.62 & 0.97 \\
& 0.92 & 8.03 & 43.09 & 0.89 & 65.66 & Inf & 0.89 & 97.82 & Inf & 0.81 \\ \cline{2-11}
\multirow{2}{*}{U}  & 0.99 & 6.47 & 7.05 & 0.99 & 8.36 & 13.17 & 1.00 & 11.44 & 16.88 & 0.98\\
& 0.89 & 8.23 & 39.62 & 0.89 & 61.46 & Inf & 0.88 & 117.81 & Inf & 0.84\\ \cline{2-11}
\multirow{2}{*}{SN}  & 0.99 & 6.48 & 7.45 & 1.00 & 8.29 & 13.00 & 1.00 & 10.96 & 16.97 & 0.97 \\
& 0.91 & 8.36 & 33.18 & 0.92 & 55.24 & Inf & 0.90 & 97.82 & Inf & 0.81 \\ 
\end{tabular}%
\end{center}
\caption{Coverage proportion (cov.), median length (med.) and $90 \%$ quantile length (qua.) for the ``POSI'' and ``TG'' confidence intervals at nominal level $1-\alpha=0.9$. The design matrix is generated with independent (upper half of the table) or correlated (lower half) columns. The errors $u$ have normal (N), Laplace (L), uniform (U) or skewed normal (SN) distributions. For each setting, the ``POSI'' intervals correspond to the first row and the ``TG'' intervals correspond to the second row. The coverage proportions, median and quantile lengths are given for each of the three targets for the three first steps of the LAR algorithm. The last column provides the simultaneous coverage proportion of the three targets after step~3.}
\label{table:homolm}
\end{table}

\subsubsection{The case of a significance hunting procedure} \label{section:holm:significance:hunting}

We now run a similar simulation study as for Table \ref{table:homolm}, in the ``independent'' design case, and where we only address the ``POSI'' confidence intervals. As a model selector, we now consider the following procedure, which we call ``significance hunting''. For given $X_n$ and $Y_n$, we first sort all the possible candidate models $\mathbb{M}\in\mathsf M_n$ according to their penalized log-likelihood, with additive penalty term equal to $\lambda |M|$, where $\mathbb{M}$ has index set $M$. Then, for the $n_{best}$ models with largest penalized log-likelihood, we compute all the possible test statistics of the form 
\[
\frac{
\left| \hat{\beta}^{(j)}_{\mathbb{M},n} \right|
}{
  \sqrt{\hat{\sigma}^2_{\mathbb{M},n} \left[\left( X_n[M]'X_n[M]  \right)^{-1}\right]_j}
},
\]
with $j =1,...,|M|$. We then return the pair $\hat{\mathbb{M}},\hat{j}$ with largest test statistics. Hence, the ``significance hunting'' procedure consists first in selecting the $n_{best}$ best models according to the penalized log-likelihood criterion, and then in finding, among these models, the configuration which results in the most significant regression coefficient. We believe that similar procedures may be common, yet unreported, practice. We also observe that, when $n_{best} = 2^p - 1$ , the ``significance hunting'' procedure corresponds to the SPAR procedure in \cite{Berk13}, and yields an asymptotic coverage exactly equal to $1-\alpha$ in Theorem~\ref{thm:homlinmods}, if $\beta$ is the zero vector, so that all candidate models are first order correct. [This can be deduced from Proposition~\ref{prop:lmvarest} and Theorem~\ref{thm:fbl}, where Condition~\ref{cond:sum} is verified, e.g., in the proof of Theorem~\ref{thm:homlinmods}.]

In Table~\ref{table:homolm:significance:hunting}, we report the coverage proportions, median lengths and $90 \%$ quantiles of the ``POSI'' confidence intervals for the ``significance hunting'' model selector, over $500$ repetitions conducted similarly as for Table~\ref{table:homolm}. Here, the target is $\beta_{\hat{\M},n}^{*(\hat{j})}$ and we try to cover it with $\mathrm{CI}_{1-\alpha, \hat{\mathbb{M}}}^{(\hat{j}), \mathrm{lm}}$. We set $p=5$ (so that it becomes numerically easier to compute the penalized log-likelihood for all possible models) and consider two settings for $\beta$. In the ``zero'' setting, we set $\beta = (0,...,0)'$. In the ``non-zero'' setting, we set $\beta = (2,-1,0,0,1)'$. We consider the values $n_{best}=5$ and $n_{best} = 20$.
We set $n=100$, $1-\alpha=0.9$, $\lambda=2$ and $\sigma=1$. The errors are normally generated.

In Table~\ref{table:homolm:significance:hunting}, the coverage proportion is almost equal to, and slightly below, the nominal level when $\beta = (0,...,0)'$ and $n_{best}=20$. The fact that we sometimes obtain coverage slightly below the nominal level is mainly due to the estimation error of $\sigma^2$. Then, the coverage proportion increases when $\beta$ is non-zero or when $n_{best} = 5$. This is well-interpreted, because decreasing $n_{best}$ decreases the weight of the significance hunting in the model selector. Also, if $\beta=0$, the pair $(\hat{\mathbb{M}},\hat{j})$ with largest test statistic corresponds to the target $\beta_{\hat{\M},n}^{*(\hat{j})}=0$ which is ``furthest away'' from its confidence interval $\mathrm{CI}_{1-\alpha, \hat{\mathbb{M}}}^{(\hat{j}), \mathrm{lm}}$. 

In Table~\ref{table:homolm:significance:hunting}, the coverage proportions are significantly lower than in Table~\ref{table:homolm}, and closer to the nominal level. Hence, the confidence intervals suggested in this paper may have conservative coverage proportions for some model selection procedures (such as LAR) but this is somehow necessary, since there exist other model selection procedures (such as ``significance hunting'') for which the coverage proportions are close to the nominal level.

\begin{table}[tbp]
\begin{center}
\begin{tabular}{c|c|ccc}
$n_{best}$ & $\beta$ & cov. & med. & qua. \\ \hline
\multirow{2}{*}{20} & zero & 0.88 & 4.99 & 5.83  \\ \cline{2-5}
 & non-zero & 0.93 & 5.00 & 5.73  \\ \hline
\multirow{2}{*}{5} & zero &  0.92 & 4.87 & 5.35  \\ \cline{2-5}
 & non-zero & 0.94 & 4.94 & 5.41
\end{tabular}%
\end{center}
\caption{Coverage proportion (cov.), median lengths (med.) and $90 \%$ quantile lengths (qua.) of the ``POSI'' confidence intervals at level $1-\alpha=0.9$ for the ``significance hunting'' model selection procedure.}
\label{table:homolm:significance:hunting}
\end{table}

\subsection{Binary regression}

We study the setting of Section \ref{sec:binary}, where binary regression models are fit to binary data. Thus, the data are of the form $X_n, Y_n$, where $X_n = (X_{1,n}',...,X_{n,n}')'$ is $n \times p$ and will be randomly generated and $Y_n$ has independent components given $X_n$ and takes values in $\{0,1\}^n$. Here, we only consider the canonical link function, so $\mathcal H = \{h^{(c)}\}$ with $h^{(c)}(x) = \frac{e^x}{1+e^x}$. Thus, each candidate model $\M\in\mathsf M_n$ is identified with the corresponding set $M\in\mathcal I= 2^{\{1,\dots, p\}}\setminus\varnothing$ of selected regressors. Here, $d=d_2 = 2^p-1$.

For the distribution of $Y_n$ given $X_n$, we let $\P( Y_{i,n} = 1 ) = h^{(c)}(\gamma_i)$ and we consider two different cases for the construction of $\gamma\in\R^n$. In the ``well-specified'' case, we let $\gamma = X_n \beta$ for a fixed $\beta \in \mathbb{R}^p$. In the ``misspecified'' case, we consider an $n \times \bar{p}$ matrix $\bar{X}_n$, with $\bar{p} > p$ and where the first $p$ columns of $\bar{X}_n$ correspond to $X_n$. Then, we let $\gamma = \bar {X}_n \bar{\beta}$ where $\bar{\beta}$ is a fixed $\bar{p} \times 1$ vector. 
For a given model $\mathbb{M}$ and a given $X_n$, the target of inference $\beta_{\M,n}^*$ is given by Lemma~\ref{lemma:pseudo}, where $\mathbb{P}_n $ in this lemma is the (conditional) distribution of $Y_n$ under $\P$ (given $X_n$).

\subsubsection{Comparison with the confidence intervals of \cite{taylor17post} and with the naive procedure}

As model selection procedure $\hat{\M}_n$ we consider the lasso for logistic regression (cf. \cite{friedman2010regularization}) with a fixed regularization parameter $\lambda$. [We maximize the difference of the log-likelihood and of $\lambda$ times the $L^1$ norm of $\beta$.] In order to compute $\hat{\M}_n$, we use the function \verb@glmnet@ of the \verb@R@ package \verb@glmnet@. The corresponding coefficient index of interest is $j = 1$, that is, we are interested in the first coefficient $\beta_{\hat{\M}_n,n}^{*(1)}$ of the post-model-selection target. We consider the confidence interval $\mathrm{CI}_{1-\alpha, \hat{\mathbb{M}}_n}^{(1), \mathrm{bin}}$ of Theorem~\ref{thm:binary} (with the adjustment discussed in Remark~\ref{rem:canonical:response:function}) which we call ``POSI''. We also consider the confidence interval $\mathrm{CI}_{1-\alpha, \hat{\mathbb{M}}_n}^{(1), \mathrm{lasso}}$ suggested in \cite{taylor17post}, which we compute by using the function \verb@fixedLassoInf@ of the \verb@R@ package \verb@selectiveInference@. This confidence interval is developed for the lasso model selector and has some asymptotic guarantees which are discussed in \cite{taylor17post}. We call it ``LASSO''.  Finally, we consider the ``naive'' confidence interval defined by
\begin{equation} \label{eq:naive:CI:logistic}
\mathrm{CI}_{1-\alpha, \hat{\mathbb{M}}_n}^{(j), \mathrm{naive}}
=
\hat{\beta}^{(j)}_{\hat{\mathbb{M}}_n, n} \pm
q_{1-\alpha/2}
\sqrt{ (\bar{S}_{\hat{\M}_n,n})_{jj} },
\end{equation}
where $q_{1-\alpha/2}$ is the $1-\alpha/2$ quantile of the standard Gaussian distribution and where $\bar{S}_{\M,n}$ is as $\tilde{S}_{\M,n}$ in \eqref{eq:tildeSlogistic} but with $\hat{u}^2_{i,\M}$ replaced by
$(h^{(c)}(\hat{\gamma}_{\M,i})(1-h^{(c)}(\hat{\gamma}_{\M,i})))$
with $\hat{\gamma}_{\M,i} = X_{i,n} \hat{\beta}_{\M,n}$. The ``naive'' interval is constructed by a plug-in of the estimate $\hat{\beta}_{\M,n}$ in the expression of the asymptotic variance of $\hat{\beta}_{\M,n}$, under the assumptions that the model $\M$ is fixed and contains the true distribution of $Y_n$ (see also the proof of Corollary~\ref{corr:binaryCanoncial}). Hence, this interval ignores the model selection step, and the potential misspecification. 

We compare the three confidence intervals in a simulation study where we repeat $N=1000$ data generations, model selections and confidence interval computations, similarly as for Table \ref{table:homolm}.  We set $1-\alpha=0.9$, $p=10$ and $n=30$ or $n=100$. The rows of the design matrix $X_n$ are independently generated from a Gaussian distribution with mean vector zero, variances $1$ and off-diagonal covariances $\rho = 0.2$.
We consider the well-specified and misspecified settings as described above. In the well-specified setting, $\beta$ is equal to $(0,...,0)'$ (``zero''), equal to $(1,0,...,0)'$ (``sparse'') or equal to $(1/n^{1/2})( (-1,1),...,(-1,1) )'$ (``scaled''). In the misspecified case, we set $\bar{p} = 21$, we generate $\bar{X}_n$ in the same way as $X_n$ (up to the change of dimension) and we set $\bar{\beta} = ((-3/2,3/2,0),...,(-3/2,3/2,0))' \in \mathbb{R}^{\bar{p}}$. We set $\lambda$ equal to $0.012 n$ (``small'') or to $0.05 n$ (``large''). 

The results are reported in Table \ref{table:logistic}. We observe that all the confidence interval lengths decrease when $n$ increases, which is natural. The ``POSI'' confidence intervals always have coverage proportions above the nominal level. In fact, the coverage proportions are quite large, which is explained by a similar argument as for Table \ref{table:homolm}: since the ``POSI'' intervals are valid for any model selection procedure, and simultaneously over the selected coefficients, they become conservative when applied specifically to the lasso and only for the first coefficient. In Table~\ref{table:logistic}, the ``POSI'' intervals compare favorably with the ``LASSO'' ones. Indeed, the median lengths are generally comparable between the ``POSI'' and ``LASSO'' intervals (less than a factor $2$ between the two median lengths). Depending on the situation, any of these two intervals can have the smallest median length. On the other hand, the $90 \%$ quantile lengths are always larger for the ``LASSO'' intervals. In some cases, they can be up to $7$ times as large as the ``POSI'' ones (for instance in the ``well-specified'', ``scaled'' $\beta$, ``large'' $\lambda$ setting with $n=100$). Also, more importantly in our opinion, the ``LASSO'' intervals can have coverage proportions way below the nominal level (down to $0.25$ instead of $0.9$), even though the lasso model selector is used here. In fact, the ``LASSO'' intervals have sufficient coverage proportion only in the cases where $\beta$ has up to one non-zero coefficient, and the coverage proportion becomes too small otherwise. In particular, in the ``misspecified'' setting, the ``LASSO'' coverage becomes very small.

Turning to the ``naive'' intervals, we observe that they are smaller than the ``POSI'' ones, by approximately a factor $2$. For both the ``POSI'' and ``naive'' intervals, the $90 \%$ quantile lengths are moderately above the median lengths and are never very large. Importantly, the ``naive'' intervals can have coverage proportions significantly below the nominal level (down to $0.68$ instead of $0.9$). This is in agreement with the fact that these intervals have no asymptotic guarantee, in the post-model-selection context. Hence, these intervals are smaller than the ``POSI'' ones at the price of not offering reliable coverage properties. [\cite{bachoc14valid} and \cite{leeb13various} reach a similar conclusion in the linear regression context.]

\begin{table}[tbp]
\begin{center}
\begin{tabular}{c|c|c|ccc|ccc|ccc}
 $\beta$ ($\bar{\beta}$) & $\lambda$ & $n$  & \multicolumn{3}{c}{cov.} \vline &
\multicolumn{3}{c}{med.} \vline & \multicolumn{3}{c}{qua.}  \\ 
& & & P & L &  N & P & L &  N & P & L &  N  \\ 
\hline\hline
\multirow{2}{*}{zero} & \multirow{2}{*}{small} &
30 & 0.99 & 0.89 & 0.84 & 4.26 & 7.44 & 2.09 & 6.97 & 43.33 & 3.42 \\
& & 100 & 1.00 &0.90& 0.85 &1.73& 2.46& 0.78&  1.95& 13.04& 0.87 \\ \cline{2-12}
\multirow{2}{*}{sparse} & \multirow{2}{*}{small} &
30 &0.97 & 0.88& 0.93 &5.88& 5.99& 3.05& 11.40& 26.10& 8.35  \\
& & 100 & 1.00 &0.88 &0.85& 2.21& 1.59& 1.01 & 2.80&  6.51& 1.24 \\ \cline{2-12}
\multirow{2}{*}{scaled} & \multirow{2}{*}{small} &
30 & 0.99& 0.87& 0.85 &4.49& 6.64& 2.21&  7.11& 35.73& 3.86  \\
& & 100 & 1.00& 0.83 &0.83 &1.74 &2.13& 0.79 & 2.00& 10.91& 0.89 \\ \cline{2-12}
\multirow{2}{*}{scaled} & \multirow{2}{*}{large} &
30 & 0.99& 0.88& 0.79& 3.97& 5.90& 1.89&  6.30& 28.56& 3.05  \\
& & 100 & 1.00& 0.85& 0.68& 1.63& 2.31& 0.74 & 1.90 &13.52& 0.84 \\ \hline\hline 
\multirow{2}{*}{dense} & \multirow{2}{*}{small} &
30 & 0.99& 0.68 &0.98& 5.47& 6.55& 3.37&  9.75& 36.35& 7.08  \\
& & 100 &1.00 &0.25 &0.98& 2.22& 1.23& 1.01 & 2.83 & 3.50 &1.24 \\ \cline{2-12}
\multirow{2}{*}{dense} & \multirow{2}{*}{large} &
30 & 0.99 &0.74 &0.99& 5.09& 4.65 &2.67 & 9.41& 20.99& 6.64  \\
& & 100 & 1.00 &0.48 &0.98 &2.12& 1.42& 0.97 & 2.64&  5.58 &1.19 \\
\end{tabular}%
\end{center}
\caption{Binary regression with the lasso model selector. Coverage proportion (cov.), median lengths (med.) and $90 \%$ quantile lengths (qua.) of the ``POSI'' (P), ``LASSO'' (L) and ``naive'' (N) confidence intervals for the lasso model selection procedure at level $1-\alpha=0.9$. We consider the well-specified (upper part of the table) and misspecified (lower part) cases.}
\label{table:logistic}
\end{table}

\subsubsection{The case of a significance hunting procedure}

Similarly as in Section \ref{section:holm:significance:hunting}, we consider a significance hunting procedure for which we study the ``POSI'' and ``naive'' confidence intervals. For the significance hunting procedure, we consider the $n_{best}$ models with largest penalized log-likelihood criteria (we subtract $\lambda |M|$ from the maximum log-likelihood for a model $\M$ with index set $M$). Then, we consider test statistics of the form 
\[
\frac{
\left| \hat{\beta}^{(j)}_{\mathbb{M},n} \right|
}{
 \sqrt{ (\bar{S}_{\hat{\M},n})_{jj} }
},
\]
with the notation of \eqref{eq:naive:CI:logistic}
and otherwise proceed as in Section \ref{section:holm:significance:hunting}. 

In Table \ref{table:logistic:significance:hunting}, we report the results for the significance hunting procedure, where we have repeated $N=1000$ data generations, model selections and confidence interval computations as for Table \ref{table:homolm:significance:hunting}. We set $\lambda=2$, $p=5$ and $1-\alpha=0.9$. The design matrices $X_n$ are randomly generated as for Table~\ref{table:logistic} but with $\rho = 0.8$. We consider the ``well-specified'' case with $\beta$ equal to $(0,...,0)'$ (``zero'') or to $(-1,1,0,0,0)'$ (``non-zero''). We set $n_{best} = 20$ or $n_{best} = 5$ and $n=30$ or $n=100$. 

We observe that for both intervals, similarly as in Table~\ref{table:logistic}, the lengths decrease when $n$ goes from $30$ to $100$ and the quantile lengths are above the median lengths by a factor less than $3$. Also, the ``naive'' intervals are about half of the length of the ``POSI'' ones.
For the same reasons as for Table \ref{table:homolm:significance:hunting}, the coverage proportions decrease when $n_{best} = 20$ or when $\beta = (0,...,0)'$. The ``POSI'' intervals can have smaller coverage proportions than for the lasso model selector, down to $0.95$ (for a nominal level equal to $0.9$). 
Finally, the coverage proportions of the ``naive'' intervals are always much too small, with a minimum of $0.38$. Hence we have another illustration, more pronounced than in Table~\ref{table:logistic}, that these intervals do not offer reliable guarantees for post-model selection inference.

\begin{table}[tbp]
\begin{center}
\begin{tabular}{c|c|c|cc|cc|cc}
$n_{best}$ & $\beta$  & $n$ & \multicolumn{2}{c}{cov.} &
\multicolumn{2}{c}{med.} & 
\multicolumn{2}{c}{qua.} \\ 
& & & P & N & P & N & P & N \\ \hline
\multirow{4}{*}{20} & 
\multirow{2}{*}{zero} & 
30 & 0.95& 0.39& 4.40& 2.63& 6.22& 3.63  \\ 
 & 
 & 
100 & 0.95 &0.38 &2.12 &1.23& 2.49& 1.43  \\ \cline{2-9}
 & 
\multirow{2}{*}{non-zero} & 
30 & 0.98 &0.65& 4.78& 2.91& 7.13& 4.20  \\ 
 & 
 & 
100 & 0.99& 0.78 &2.41 &1.40 &2.95 &1.68  \\ \cline{1-9}
\multirow{4}{*}{5} & 
\multirow{2}{*}{zero} & 
30 & 0.96 &0.55 &2.57& 1.48 &5.63 &3.32  \\ 
 & 
 & 
100 &  0.96 &0.58& 1.23 &0.71& 2.20& 1.27 \\ \cline{2-9}
 & 
\multirow{2}{*}{non-zero} & 
30 & 0.98& 0.68& 4.08& 2.53 &6.89 &3.97 \\ 
 & 
 & 
100 & 0.99 &0.78 &2.32& 1.35& 2.86 &1.64  \\ 
\end{tabular}%
\end{center}
\caption{Binary regression with the significance hunting procedure. Coverage proportion (cov.), median length (med.) and $90\%$ quantile length (qua.) of the ``POSI'' (P) and ``naive'' (N) confidence intervals. }
\label{table:logistic:significance:hunting}
\end{table}

To conclude the simulation study in the binary case, we believe that the results in Tables \ref{table:logistic} and \ref{table:logistic:significance:hunting} provide a complimentary picture of the confidence intervals suggested in this paper. Indeed, these intervals can be much shorter than the ``LASSO'' ones and they are never more than twice as large as the ``LASSO'' or ``naive'' intervals. Furthermore, only they have sufficient coverage proportions in all the settings studied. The two other types of intervals can yield significant under-coverage. More precisely, the `LASSO'' intervals can exhibit strong under-coverage even though the lasso model selector is used. The ``naive'' intervals can yield small coverage proportions for the lasso model selector, and yield even smaller ones for the significance hunting procedure.
Finally, the ``POSI'' intervals are asymptotically valid for any model selection procedure, while the ``LASSO'' ones can only be used in conjunction with the lasso model selector, and the ``naive'' ones do not have asymptotic guarantees in the post-model-selection context at all.


\section{Auxiliary results}\label{app:aux}

In this section, for every $n \in \N$ and for some $k \in \N$, not depending on $n$, let $z_{1,n}, \hdots, z_{n,n}$ be independent $k$-variate random vectors defined on a probability space $(\Omega_n, \mathcal{A}_n, \mathbb{P}_n)$. Denote 
\begin{equation}
r_n \quad = \quad (r^{(1)}_n, \hdots, r^{(k)}_n)' \quad = \quad \sum_{i = 1}^n (z^{(1)}_{i,n}, \hdots, z^{(k)}_{i,n})',
\end{equation}
and let
\begin{equation}
S_n \quad = \quad \sum_{i = 1}^n z_{i,n} z_{i,n}'.
\end{equation}

In this section, the expectation operator and the variance-covariance operator w.r.t. $\mathbb{P}_n$ is denoted by $\mathbb{E}_n$ and $\mathbb{V}_n$, respectively.

\begin{condition}\label{cond:zmom}
For every $n \in \N$, every $j \in \{1, \hdots, k\}$ and every $i \in \{1, \hdots, n\}$ 
\begin{equation}\label{EQ:mom}
\E_n(z^{(j)}_{i,n}) = 0 \quad \text{ and } \quad \mathbb{V}_n(r^{(j)}_n) = 1.
\end{equation}
Furthermore, for every $j \in \{1, \hdots, k\}$ we have
\begin{equation}\label{EQ:conv}
\mathbb{P}_n \circ r^{(j)}_n \quad \Rightarrow \quad N(0,1),
\end{equation}
and for every $\varepsilon > 0$
\begin{equation}\label{EQ:neg}
\max_{1 \leq i \leq n} \mathbb{P}_n\left(|z^{(j)}_{i,n}|\geq \varepsilon \right) \to 0.
\end{equation}
\end{condition}

The first statement in the subsequent lemma is essentially Corollary 2 in \cite{pollak1972} combined with a tightness argument. The second statement is obtained via an application of Raikov's theorem (\cite{raikov}, cf. the statement given in \cite{gnedenko} on p. 143).

\begin{lemma}\label{lemma:zmom}
Suppose Condition~\ref{cond:zmom} holds. Then
\begin{equation}
d_w \big( \mathbb{P}_n \circ r_n, N(0, \mathbb{V}_n(r_n)) \big) \to 0.
\end{equation}
Furthermore, for every $\varepsilon > 0$ it holds that
\begin{equation}\label{eqn:empvar}
\mathbb{P}_n \left(\big \|S_n -\mathbb{V}_n(r_n) \big \| \geq \varepsilon   \right) \to 0,
\end{equation}
and hence that
\begin{equation}\label{eqn:empclt}
\mathbb{P}_n \left( d_w \big( \mathbb{P}_n \circ r_n, N(0, S_n) \big) \geq \varepsilon\right) \to 0.
\end{equation}

\end{lemma}

\begin{proof}
For the first claim, let $n'$ be an arbitrary subsequence. From Equation \eqref{EQ:mom} we see that $\mathbb{E}_n(r_n) = 0$ and that $\mathbb{V}_n(r_n)$ is norm-bounded and, hence, that $\mathbb{P}_n \circ r_n$ is tight. Therefore, there exists a subsequence $n''$ of $n'$ along which $\mathbb{V}_n(r_n)$ converges to $\Sigma$, say, and along which $\mathbb{P}_n \circ r_n$ converges weakly. We now need to show that $d_w(\mathbb{P}_{n''} \circ r_{n''}, N(0, \Sigma)) \rightarrow 0$, which then proves the statement in view of the triangle inequality, continuity of $\Gamma \mapsto N(0, \Gamma)$ w.r.t. $d_w$, and the fact that $n'$ was arbitrary. That the weak limit of $\mathbb{P}_{n''} \circ r_{n''}$ must be normal follows from Equations \eqref{EQ:conv} and \eqref{EQ:neg}, applying Corollary 2 in \cite{pollak1972}. The mean vector of the limiting distribution of $\mathbb{P}_{n''} \circ r_{n''}$ is $0$ from \eqref{EQ:conv}. It remains to verify that the covariance matrix of the limiting distribution is $\Sigma$. From Equations \eqref{EQ:mom} and \eqref{EQ:conv} and, e.g., Theorem 5.4 in \cite{bill2} it follows that $[r_{n}^{(j)}]^2$ is uniformly integrable for $j = 1, \hdots, k$. The inequality $ab \leq \frac{1}{2}(a^2 + b^2)$, together with the fact that the sum of two uniformly integrable sequences is uniformly integrable, then shows uniform integrability of $r_n^{(s)}r_n^{(t)}$, and hence (e.g., again Theorem 5.4 in \cite{bill2} together with weak convergence of $\mathbb{P}_{n''} \circ r_{n''}$ and the continuous mapping theorem) that the covariance matrix of the limiting distribution of $\mathbb{P}_{n''} \circ r_{n''}$ coincides with $\Sigma$.

To prove the second claim, we start with the observation that it suffices to verify that for every $\varepsilon > 0$ and every $\gamma \in \R^k$ it holds that
\begin{equation}
\mathbb{P}_n \left(|\gamma'\left( S_n - \mathbb{V}_n(r_n) \right) \gamma | \geq \varepsilon \right) \to 0.
\end{equation}
To see this, it suffices to first take $\gamma$ equal to the elements of the standard basis in $\R^k$ in order to show that the diagonal entries converge to zero. Then, taking $\gamma$ equal to $(1,1,0,\dots, 0)'$, $(0,1,1,0,\dots, 0)'$, etc., and using symmetry shows that also the entries above and below the main diagonal converge. Continuing this process with vectors containing exactly three, four, five, etc., consecutive ones, establishes the claim.
Next, to verify the statement in the previous display, let $\gamma \in \R^k$, $\eps > 0$, and let $n'$ be an arbitrary subsequence. Choose $n''$ a subsequence of $n'$ along which $\mathbb{P}_n \circ r_n \Rightarrow N(0, \Sigma)$ - such a subsequence exists because of the already established part of the lemma. We also already know from the uniform integrability argument above, that then $\mathbb{V}_{n''}(r_{n''}) \rightarrow \Sigma$, and hence that $\sigma_{n''}^2 := \gamma' \mathbb{V}_{n''}(r_{n''}) \gamma \rightarrow \gamma' \Sigma \gamma =: \sigma^2$. Now, if $\sigma^2 = 0$, then eventually $\sigma_{n''}^2 < \varepsilon/2$, and by Markov's inequality
\begin{equation}
 \mathbb{P}_{n''} \left(\big|\gamma'\left(S_{n''}-\mathbb{V}_{n''}(r_{n''}) \right)\gamma \big| \geq \varepsilon   \right) 
\leq 
\mathbb{P}_{n''} \left(\gamma' S_{n''} \gamma \geq \frac{\varepsilon}{2}   \right) 
\leq 
\frac{2 \sigma_{n''}^2}  { \varepsilon}  \to 0.
\end{equation}
Suppose next that $\sigma^2 > 0$. Then, we can assume without loss of generality that $0 <\delta_1 \le \sigma_{n''}^2 \leq \delta$ for some $\delta,\delta_1 \in \R$, and it remains to verify that
\begin{equation}\label{eq:remains}
\mathbb{P}_{n''} \left( \left|\frac{\gamma' S_{n''} \gamma}{\sigma_{n''}^2 } - 1 \right| \geq \frac{\varepsilon}{\delta}   \right) \to 0.
\end{equation}
To that end, define
\begin{equation}
\xi_{i,n''} \quad = \quad \frac{\gamma' z_{i, n''}}{\sigma_{n''}}
\end{equation}
where $\sigma_{n''}$ denotes the positive square root of $\sigma_{n''}^2$, and note that by Equation~\eqref{EQ:mom} we have $\mathbb{E}_{n''} (\xi_{i,n''} ) = 0$, that by construction $\sigma_{n''}^{-2} \gamma' S_{n''} \gamma = \sum_{i = 1}^{n''} \xi_{i,n''}^2$, and that $\mathbb{V}_{n''}(\sum_{i = 1}^{n''} \xi_{i,n''}) = 1$. Note also that it follows from Equation~\eqref{EQ:neg} that for every $\bar{\delta} > 0$ we have
\begin{equation}
\max_{1 \leq i \leq n''} \mathbb{P}_{n''} \left(|\xi_{i, n''}| \geq \bar{\delta} \right) \to 0 \text{ as } n'' \rightarrow \infty.
\end{equation}
Furthermore, since $\mathbb{P}_{n''} \circ r_{n''} \Rightarrow N(0, \Sigma)$, we see that $\sum_{i = 1}^{n''} \xi_{i,n''}$ is asymptotically normal with mean $0$ and variance $1$. But then Equation~\eqref{eq:remains} follows from Raikov's theorem (\cite{gnedenko}, p. 143, Theorem 4). Since $n'$ was arbitrary, this proves the second statement. 

The statement in Equation \eqref{eqn:empclt} is an immediate consequence of the triangle inequality and the first two statements.
\end{proof}

\begin{condition}\label{cond:zmomstd}
For every $n \in \N$, every $j \in \{1, \hdots, k\}$ and every $i \in \{1, \hdots, n\}$ we have
\begin{equation}\label{EQ:momsts}
\mathbb{E}_n(z^{(j)}_{i,n}) = 0 \quad \text{ and } \quad 0 < \mathbb{V}_n(r^{(j)}_n) < \infty.
\end{equation}
Furthermore, setting $r^{(j)}_{n,*} = \frac{r^{(j)}_{n}}{\sqrt{\mathbb{V}_n(r^{(j)}_n)}}$ and $z^{(j)}_{i, n, *} = \frac{z^{(j)}_{i, n}}{\sqrt{\mathbb{V}_n(r^{(j)}_n)}}$, for every $j \in \{1, \hdots, k\}$ we have
\begin{equation}\label{EQ:convstd}
\mathbb{P}_n \circ r^{(j)}_{n,*} \Rightarrow N(0,1),
\end{equation}
and for every $\varepsilon > 0$
\begin{equation}\label{EQ:negstd}
\max_{1 \leq i \leq n} \mathbb{P}_n\left(|z^{(j)}_{i,n, *}|\geq \varepsilon \right) \to 0 \text{ as } n \to \infty.
\end{equation}
\end{condition}

\begin{lemma}\label{lemma:nonstd}
Suppose Condition~\ref{cond:zmomstd} holds. Then, for every $\varepsilon > 0$, we have
\begin{equation}\label{EQ:diags}
\mathbb{P}_n \left( \|\diag(\mathbb{V}_n(r_n))^{-1} \diag(S_n)  - I_k\| \geq \varepsilon \right) \to 0
\end{equation}
and
\begin{equation}\label{EQ:corrs}
\mathbb{P}_n \left( \|\corr \left( S_n \right) - \corr \left(  \mathbb{V}_n(r_n) \right)\| \geq \varepsilon \right) \to 0.
\end{equation}
Furthermore, for every $\eps > 0$, $\hat{r}_{n,*} = \diag(S_n)^{\dagger/2} ~  r_n$ satisfies 
\begin{equation}
\mathbb{P}_n \left(d_w\left(\mathbb{P}_n \circ  \hat{r}_{n,*}  , N(0, \corr(S_n))\right) \geq \eps \right) \rightarrow 0.
\end{equation}

\end{lemma}

\begin{proof}
For the statements in Equations \eqref{EQ:diags} and \eqref{EQ:corrs} we first note that the triangular array $z_{i,n,*}^{(j)}$, and the corresponding quantities $r_{n, *}$ and $r_{n,*}^{(j)}$, satisfy Condition \ref{cond:zmom}. Hence, Lemma \ref{lemma:zmom} is applicable, and shows, in particular, for every $\varepsilon > 0$ and with the abbreviation $S_{n,*} = \sum_{i = 1}^n z_{i,n,*} z_{i,n,*}'$, that
\begin{equation}
\mathbb{P}_n \left(\| S_{n,*} -\mathbb{V}_n(r_{n, *}) \| \geq \varepsilon   \right) \to 0 \quad \text{ as } n \rightarrow \infty .
\end{equation}
Noting that the diagonal entries of $\mathbb{V}_n(r_{n, *})$ are all equal to $1$ (in fact $\mathbb{V}_n(r_{n, *})$ $=$ $\corr (\mathbb{V}_n(r_{n}))$) and that $\diag \left(S_{n,*}\right) = \diag(\mathbb{V}_n(r_n))^{-1} \diag(S_n)$ holds, establishes the claimed convergence in \eqref{EQ:diags}. But this together with the preceding display then establishes the convergence in \eqref{EQ:corrs}, because, using the abbreviation $A_n = \diag(S_n)^{\dagger/2}\diag(\mathbb{V}_n(r_n))^{1/2}$, we have $A_n \to I_k$ in $\mathbb{P}_n$-probability and
\begin{align*}
\corr(S_n) - \corr(\mathbb{V}_n(r_n))
&=
\diag(S_n)^{\dagger/2} S_n \diag(S_n)^{\dagger/2} - \mathbb{V}_n(r_{n,*})\\
&=
A_n (S_{n,*} - \mathbb{V}_n(r_{n,*}) ) A_n' \\
&\quad+ A_n \mathbb{V}_n(r_{n,*}) A_n' -  \mathbb{V}_n(r_{n,*}),
 \end{align*}
 which converges to zero in $\mathbb{P}_n$-probability.
The last part is an application of the statements already established, together with 
\begin{equation}
\mathbb{P}_n \left( d_w\left( \mathbb{P}_n \circ r_{n, *}, N(0, S_{n,*})  \right)  \geq \eps \right) \to 0 \text{ for every } \eps > 0,
\end{equation}
which we obtain (as above) from Lemma \ref{lemma:zmom}. 
\end{proof}

\begin{lemma}\label{lemma:POSIcont}
For every $\alpha \in (0, 1)$ the map $\Gamma \mapsto K_{1-\alpha}(\Gamma)$ is continuous on the subset of $k \times k$-dimensional covariance matrices of $\R^{k \times k}$. 
\end{lemma}

\begin{proof}
Let $\Gamma_n$ be a sequence of covariance matrices converging to $\Gamma$. By definition, $K_{1-\alpha}(\Gamma_n)$ is the $1-\alpha$-quantile of the distribution of $\|Z_n\|_{\infty}$, where $Z_n$ is a Gaussian random vector with mean $0$ and covariance matrix $\Gamma_n$. By the continuous mapping theorem, $\|Z_n\|_{\infty}$ converges weakly to $\|Z\|_{\infty}$, where $Z$ is a Gaussian random vector with mean $0$ and covariance matrix $\Gamma$. In case $\Gamma \neq 0$ it is easy to see that the distribution function of $\|Z\|_{\infty}$ is everywhere continuous and strictly increasing on $[0,\infty)$, and the result then follows, because weak convergence of distribution functions is equivalent to weak convergence of the corresponding quantile functions. Consider now the case where $\Gamma = 0$. Fix $0 < \varepsilon < 1$. Let $z$ be a random variable taking values in $[0, \infty)$, with continuous and strictly increasing (on $[0,\infty)$) distribution function and $1-\alpha$-quantile equal to $\varepsilon$. Clearly, $\|Z_n\|_{\infty} + z$ converges weakly to $z$. Hence $K_n$, say, the $1-\alpha$ quantile of $\|Z_n\|_{\infty} + z$ converges to $\varepsilon$. From $K_{1-\alpha}(\Gamma_n) \leq K_n$ it then follows that
\begin{equation*}
0 \quad \leq \quad \limsup_{n \rightarrow \infty} K_{1-\alpha}(\Gamma_n) \quad \leq \quad \varepsilon.
\end{equation*}
Therefore, $K_{1-\alpha}(\Gamma_n) \to 0 = K_{1-\alpha}(0)$.
\end{proof}

\begin{lemma}\label{lemma:auxi}
For $n \in \N$, for every $i = 1, \hdots, n$, and for $j = 1, 2$, let  $a_{i,n}(j)$ and $b_{i,n}(j)$ be random variables on a probability space $(\Omega_n, \mathcal{A}_n, \mathbb{P}_n)$. Furthermore, let $a_{i,n}=a_{i,n}(1)$ and $b_{i,n}=b_{i,n}(1)$.

\begin{enumerate}
\item If $\mathbb{P}_n(\sum_{i = 1}^n a^2_{i,n} = 0) \to 0$ holds, and if $\frac{\sum_{i = 1}^n (a_{i,n} - b_{i,n})^2}{\sum_{i = 1}^n a^2_{i,n} } = o_{\mathbb{P}_n}(1)$, then 
\begin{equation}
\mathbb{P}_n \left( \bigg| \frac{\sum_{i = 1}^n b^2_{i,n} }{\sum_{i = 1}^n  a^2_{i,n}} - 1 \bigg| \geq \eps \right) \to 0 \quad \text{ for every } \eps > 0.
\end{equation}
\item If $\mathbb{P}_n(\sum_{i = 1}^n a^2_{i,n}(j) = 0) \to 0$ and $\frac{\sum_{i = 1}^n (a_{i,n}(j) - b_{i,n}(j))^2}{\sum_{i = 1}^n a^2_{i,n}(j) } = o_{\mathbb{P}_n}(1)$ holds for $j = 1, 2$, then for every $\eps > 0$ 
\begin{align}
&\mathbb{P}_n \bigg( \bigg | \frac{\sum_{i = 1}^n a_{i,n}(1)  a_{i,n}(2)}{\sqrt{\sum_{i = 1}^n  a_{i,n}^2(1) } \sqrt{  \sum_{i = 1}^n  a^2_{i,n}(2)}} 
\\[7pt] & \hspace{3cm}  - \quad  \frac{\sum_{i = 1}^n b_{i,n}(1)  b_{i,n}(2)}{\sqrt{\sum_{i = 1}^n  b_{i,n}^2(1)}  \sqrt{\sum_{i = 1}^n  b^2_{i,n}(2)}} \bigg | \geq \eps \bigg) \to 0.
\end{align}
\item Suppose that $\mathbb{V}_n(\sum_{i = 1}^n a_{i,n}) > 0$ holds eventually. Assume that for every $n$ the random variables $a_{i,n}$ for $i = 1, \hdots, n$ have mean $0$ and are uncorrelated, and that it holds that  
\begin{equation}\label{eqn:consia}
\mathbb{P}_n\left(\bigg|\frac{\sum_{i = 1}^n a^2_{i,n}}{\mathbb{V}_n(\sum_{i = 1}^n a_{i,n})} - 1\bigg| > \eps\right) \to 0 \text{ for every } \eps > 0,
\end{equation}
that
\begin{equation}\label{eqn:maxto0}
\frac{\max_{i = 1, \hdots, n} \mathbb{V}_n(a_{i,n})}
{\mathbb{V}_n(\sum_{i = 1}^n a_{i,n})} \to 0,
\end{equation}
and that $\frac{\sum_{i = 1}^n (a_{i,n} - b_{i,n})^2}{\sum_{i = 1}^n a^2_{i,n} } = o_{\mathbb{P}_n}(1)$. Then for every array of real numbers $c_{i,n}$ we have
\begin{equation}
\mathbb{P}_n \left(  \frac{\sum_{i = 1}^n (b_{i,n} + c_{i,n})^2}{\mathbb{V}_n(\sum_{i = 1}^n a_{i,n})} \leq 1 - \eps \right) \to 0 \quad \text{ for every } \eps > 0.
\end{equation}
\end{enumerate}
\end{lemma}

\begin{proof}
For the first part note that the quotient under consideration is well defined with probability converging to one, that
\begin{equation}
\frac{\sum_{i = 1}^n b^2_{i,n} }{\sum_{i = 1}^n  a^2_{i,n}} - 1 = o_{\mathbb{P}_n}(1) + 2 \frac{\sum_{i = 1}^n  a_{i,n} [b_{i,n} - a_{i,n}]}{\sum_{i = 1}^n  a^2_{i,n}},
\end{equation}
and that by the Cauchy-Schwarz inequality 
\begin{equation}
\bigg| \frac{\sum_{i = 1}^n  a_{i,n} [b_{i,n} - a_{i,n}]}{\sum_{i = 1}^n  a^2_{i,n}} \bigg|
\leq 
\sqrt{\frac{\sum_{i = 1}^n (a_{i,n} - b_{i,n})^2}{\sum_{i = 1}^n  a^2_{i,n}}} = o_{\mathbb{P}_n}(1).
\end{equation}
For the second part note that the quotients are well defined with probability converging to 1 (by applying Part 1), and write
\begin{equation}
\frac{\sum_{i = 1}^n a_{i,n}(1)  a_{i,n}(2)}{\sqrt{\sum_{i = 1}^n  a_{i,n}^2(1) } \sqrt{  \sum_{i = 1}^n  a^2_{i,n}(2)}} - \frac{\sum_{i = 1}^n b_{i,n}(1)  b_{i,n}(2)}{\sqrt{\sum_{i = 1}^n  b_{i,n}^2(1)}  \sqrt{\sum_{i = 1}^n  b^2_{i,n}(2)}}
\end{equation}
as the sum of
\begin{align}
A_n &:= \frac{\sum_{i = 1}^n \big(a_{i,n}(1) - b_{i,n}(1)\big)\big(a_{i,n}(2) - b_{i,n}(2)\big)}{\sqrt{\sum_{i = 1}^n  a_{i,n}^2(1) } \sqrt{  \sum_{i = 1}^n  a^2_{i,n}(2)}}
\\[7pt]
B_n &:= \frac{\sum_{i = 1}^n b_{i,n}(2) \big(a_{i,n}(1) - b_{i,n}(1)\big)}{\sqrt{\sum_{i = 1}^n  a_{i,n}^2(1) } \sqrt{  \sum_{i = 1}^n  a^2_{i,n}(2)}}
\\[7pt]
C_n &:= \frac{\sum_{i = 1}^n b_{i,n}(1)\big(a_{i,n}(2) - b_{i,n}(2)\big)}{\sqrt{\sum_{i = 1}^n  a_{i,n}^2(1) } \sqrt{  \sum_{i = 1}^n  a^2_{i,n}(2)}}
\end{align}
and
\begin{align}
D_n :&= \frac{\sum_{i = 1}^n b_{i,n}(1)b_{i,n}(2)}{\sqrt{\sum_{i = 1}^n  a_{i,n}^2(1) } \sqrt{  \sum_{i = 1}^n  a^2_{i,n}(2)}} - \frac{\sum_{i = 1}^n b_{i,n}(1)  b_{i,n}(2)}{\sqrt{\sum_{i = 1}^n  b_{i,n}^2(1)}  \sqrt{\sum_{i = 1}^n  b^2_{i,n}(2)}}.
\end{align}
Using Cauchy-Schwarz inequality, the assumptions, and the first part of the lemma, we now see that $A_n$, $B_n$, $C_n$, and $D_n$ are $o_{\mathbb{P}_n}(1)$. 

For the third part we note that (eventually)
\begin{equation}
\frac{\sum_{i = 1}^n (b_{i,n} + c_{i,n})^2}{\mathbb{V}_n(\sum_{i = 1}^n a_{i,n})} - 1 \quad = \quad F_n + G_n + H_n,
\end{equation}
where
\begin{align}\label{eqn:Fn}
F_n \quad &:= \quad \frac{\sum_{i = 1}^n (b_{i,n} - a_{i,n})^2}{\mathbb{V}_n(\sum_{i = 1}^n a_{i,n})} = o_{\mathbb{P}_n}(1) \\[7pt]
\label{eqn:Gn}
G_n \quad &:= \quad \frac{\sum_{i = 1}^n (a_{i,n} + c_{i,n})^2}{ \mathbb{V}_n(\sum_{i = 1}^n a_{i,n})} - 1 = \frac{\sum_{i = 1}^n c_{i,n}(2a_{i,n} + c_{i,n})}{ \mathbb{V}_n(\sum_{i = 1}^n a_{i,n})}+ o_{\mathbb{P}_n}(1)  \\[7pt]
\label{eqn:Hn}
H_n \quad &:= \quad 2 \frac{\sum_{i = 1}^n(b_{i,n} - a_{i,n})(a_{i,n} + c_{i,n})}{ \mathbb{V}_n(\sum_{i = 1}^n a_{i,n})} 
\end{align}
where the second equality for $F_n$ follows from the last assumption appearing in Part 3 together with \eqref{eqn:consia}, and the second equality for $G_n$ follows from \eqref{eqn:consia}. By the Cauchy-Schwarz inequality and the last assumption appearing in Part 3
\begin{equation}\label{eqn:Hnbound}
|H_n| \leq o_{\mathbb{P}_n}(1) \sqrt{G_n + 1}.
\end{equation}
Now, define
\begin{equation}
\kappa_n \quad = \quad \frac{\sum_{i = 1}^n c_{i,n}(2a_{i,n} + c_{i,n})}{ \mathbb{V}_n(\sum_{i = 1}^n a_{i,n})}
\end{equation}
and note that, since by assumption $\mathbb{E}_n(a_{i,n}) = 0$, we have
\begin{equation}
\mathbb{E}_n(\kappa_n) = \frac{\sum_{i = 1}^n c_{i,n}^2}{\mathbb{V}_n(\sum_{i = 1}^n a_{i,n})} =: d_n \geq 0,
\end{equation}
and it holds, using uncorrelatedness of $a_{i,n}$ for $i = 1, \hdots, n$, that
\begin{equation}\label{eqn:kappavb}
\mathbb{V}_n(\kappa_n) = 4 \frac{\sum_{i = 1}^n c_{i,n}^2 \mathbb{V}_n(a_{i,n})}{\mathbb{V}^2_n(\sum_{i = 1}^n a_{i,n})}  \leq 4 \frac{\max_{i = 1, \hdots, n} \mathbb{V}_n(a_{i,n})}{\mathbb{V}_n(\sum_{i = 1}^n a_{i,n})} ~ d_n.
\end{equation}
We need to verify that for every $\eps > 0$ it holds that
\begin{equation}
\mathbb{P}_n \left( F_n + G_n + H_n \leq - \eps \right) \to 0.
\end{equation}
We argue by contradiction: Suppose there exists an $\eps > 0$ so that the convergence in the previous display does not hold. Then, by compactness of the Cartesian product of the extended real line with the unit interval, there exists a subsequence $n'$ along which $d_n$ converges to a $c \in [0, \infty]$ and along which the probability in the previous display converges to a $\gamma \in (0, 1] $. Suppose first that $0 \leq c < \infty$. Then, from Equation \eqref{eqn:kappavb} and Assumption \eqref{eqn:maxto0}, we see that $\kappa_{n'}$, and hence $G_{n'}$, converges to $c$ in $\mathbb{P}_{n'}$-probability, and, by Equation \eqref{eqn:Hnbound}, that $H_{n'}$ converges to $0$ in $\mathbb{P}_{n'}$-probability, showing that $F_{n'} + G_{n'} + H_{n'}$ converges to $c \geq 0$ in $\mathbb{P}_{n'}$-probability, and hence that the sequence in the previous display converges along $n'$ to $0<\gamma$, a contradiction. Assume next that $c = \infty$, and assume then, without loss of generality, that $d_{n'} > 0$ holds. We show that $(F_{n'} + G_{n'} + H_{n'})/d_{n'}$ converges to $1$ in $\mathbb{P}_{n'}$-probability, which then again contradicts $\gamma \in (0, 1]$. Note that it suffices to verify that $\kappa_{n'}/d_{n'}$ converges to $1$  in $\mathbb{P}_{n'}$-probability. But this follows, because the expectation of $\kappa_{n'}/d_{n'}$ is $1$, and because, by relation~\eqref{eqn:kappavb}, its variance is bounded from above by
\begin{equation}
4 d_{n'}^{-1} \frac{\max_{i = 1, \hdots, n'} \mathbb{V}_{n'}(a_{i,n'})}{ \mathbb{V}_{n'}(\sum_{i = 1}^{n'} a_{i,n'})} \to 0,
\end{equation}
where we used $d_{n'} \to \infty$ and Assumption~\eqref{eqn:maxto0} to obtain the limit.
\end{proof}


\section{Proofs for Section 2}
\label{app:posi}

\subsection{Proof of Lemma \ref{lemma:equivSn}}
\label{sec:equivSn}

We actually prove the following more detailed statement.

\begin{lemma}\label{lemma:equivSn:full}
Under Condition \ref{cond:sum}, for $\eps > 0$ we have
\begin{align*}
&\mathbb{P}_n \left( \|\diag(\mathbb{V}_n(r_n))^{-1} \diag(S_n)  - I_k\| \geq \varepsilon \right) \to 0, \\[7pt]
&\mathbb{P}_n \left( \|\corr \left( S_n \right) - \corr \left(  \mathbb{V}_n(r_n) \right)\| \geq \varepsilon \right) \to 0 \quad\text{and}\\[7pt]
&\mathbb{P}_n \left(d_w\left(\mathbb{P}_n \circ  \left[\diag(S_n)^{\dagger/2} \left(\hat{\theta}_n - \theta_n^* \right)\right] , N(0, \corr(S_n))\right) \geq \eps \right) \to 0.
\end{align*}
The last statement remains valid upon replacing $S_n$ by $\mathbb{V}_n(r_n)$ and then reduces to $d_w\left(\mathbb{P}_n \circ  \left[\diag(\mathbb{V}_n(r_n))^{\dagger/2} \left(\hat{\theta}_n - \theta_n^* \right)\right] , N(0, \corr(\mathbb{V}_n(r_n)))\right) \to 0$.
\end{lemma}

\begin{proof}
Lemma \ref{lemma:nonstd} applied to the array $z_{i,n} = g_{i,n} \circ \pi_{i,n}$ defined on the space $(\R^{n \times \ell}, \mathcal{B}(\R^{n \times \ell}), \mathbb{P}_n)$, where $\pi_{i,n}: \R^{n \times \ell} \to \R^{1 \times \ell}$ extracts the $i-th$ row of an $n \times \ell$ matrix (to verify Condition~\ref{cond:zmomstd} we use Condition \ref{cond:sum} and replace the Lindeberg condition as discussed in Remark \ref{rem:lindeequiv}), shows that for every $\eps > 0$ we have
\begin{align}\label{eqn:varconvintro}
&\mathbb{P}_n \left( \|\diag(\mathbb{V}_n(r_n))^{-1} \diag(S_n)  - I_k\| \geq \varepsilon \right) \to 0, \\
&\mathbb{P}_n \left( \|\corr \left( S_n \right) - \corr \left(  \mathbb{V}_n(r_n) \right)\| \geq \varepsilon \right) \to 0
\end{align}
and
\begin{equation}
\mathbb{P}_n \left( d_w\left(\mathbb{P}_n \circ  \hat{r}_{n,*}  , N(0, \corr(S_n))\right) \geq \eps \right) \to 0,
\end{equation}
where $\hat{r}_{n, *} = \diag(S_n)^{\dagger/2} r_n$. The last part of Condition~\ref{cond:sum} together with Equation \eqref{eqn:varconvintro} now shows that $ \diag(S_n)^{\dagger/2} \Delta_n \to 0$ w.r.t. $\mathbb{P}_n$, so that 
\begin{equation}
\hat{r}_{n, *} = \diag(S_n)^{\dagger/2} \left( \hat{\theta}_n - \theta^*_n \right) + o_{\mathbb{P}_n}(1),
\end{equation}
which then proves the claim.
\end{proof}


\subsection{Proof of Theorem \ref{thm:fbl}}

For any (measurable) model selection procedure $\hat{\mathbb{M}}_n$ we have
\begin{align}
& \mathbb{P}_n\left( \theta^{*(j)}_{\hat{\mathbb{M}}_n, n} \in \mathrm{CI}_{1- \alpha, \hat{\mathbb{M}}_n}^{(j), \mathrm{est}} \text{ for all } j = 1, \hdots, m(\hat{\mathbb{M}}_n) \right) \\
\geq \quad &
\mathbb{P}_n \left(\theta^{*(j)}_{\mathbb{M}, n}   \in  \mathrm{CI}_{1-\alpha, \mathbb{M}}^{(j), \mathrm{est}} \text{ for all } \mathbb{M} \in \mathsf{M}_n \text{ and } j \in \{1, \hdots, m(\mathbb{M})\}  \right).
\end{align}
It hence suffices to verify that the lower bound converges to $1-\alpha$. And for that (cf. Equation \eqref{eqn:hvarconv} below, and Condition \ref{cond:sum}) it suffices to verify that the following quantity converges to $1-\alpha$:
\begin{equation}
\mathbb{P}_n \left( \| \diag(\hat{S}_n)^{\dagger/2} (\hat{\theta}_n - \theta^*_n) \|_{\infty} \leq K_{1-\alpha}(\corr(\hat{S}_n)) \right).
\end{equation}
Lemma \ref{lemma:equivSn:full} shows that for every $\eps > 0$ we have
\begin{align}\label{eqn:varconv}
&\mathbb{P}_n \left( \|\diag(\mathbb{V}_n(r_n))^{-1} \diag(S_n)  - I_k\| \geq \varepsilon \right) \to 0 \\
\label{eqn:corrconv}
&\mathbb{P}_n \left( \|\corr \left( S_n \right) - \corr \left(  \mathbb{V}_n(r_n) \right)\| \geq \varepsilon \right) \to 0 \\
\label{eqn:wconv}
&\mathbb{P}_n \left( d_w\left(\mathbb{P}_n \circ  \left[ \diag(S_n)^{\dagger/2} (\hat{\theta}_n - \theta^*_n) \right]  , N(0, \corr(S_n))\right) \geq \varepsilon\right) \to 0.
\end{align}
This also shows that the two conditions imposed  on $\hat{S}_n$ in the statement of the theorem are indeed equivalent. Furthermore, we immediately see that from any of these two assumptions, together with the previous display, it follows that for every $\eps > 0$ we have
\begin{align}\label{eqn:hvarconv}
&\mathbb{P}_n \left( \|\diag(\mathbb{V}_n(r_n))^{-1} \diag(\hat{S}_n)  - I_k\| \geq \varepsilon \right) \to 0 \\
\label{eqn:hcorrconv}
&\mathbb{P}_n \left( \|\corr \left( \hat{S}_n \right) - \corr \left(  \mathbb{V}_n(r_n) \right)\| \geq \varepsilon \right) \to 0 \\
\label{eqn:hwconv}
&\mathbb{P}_n \left( d_w\left(\mathbb{P}_n \circ  \bar{r}_{n, *}  , N(0, \corr(\hat{S}_n))\right) \geq \varepsilon\right) \to 0,
\end{align}
where $\bar{r}_{n, *} = \diag(\hat{S}_n)^{\dagger/2} (\hat{\theta}_n - \theta^*_n)$. Next, let $n'$ be an arbitrary subsequence of $n$, and let $n''$ be a subsequence of $n'$ along which the norm-bounded sequence $\corr(\mathbb{V}_n(r_n))$ converges to $\bar{\Sigma}$, say. By \eqref{eqn:hcorrconv} it holds that $\corr(\hat{S}_{n''})$ converges to $\bar{\Sigma}$ in $\mathbb{P}_{n''}$-probability, and from the previous display it follows that $\mathbb{P}_{n''} \circ  \bar{r}_{n'',*}  \Rightarrow N(0, \bar{\Sigma})$. Combining these two statements, it then follows that 
\begin{equation}
\mathbb{P}_{n''} \circ (\bar{r}_{{n''}, *}, \corr(\hat{S}_{n''})) \Rightarrow Q_{\bar{\Sigma}} \otimes \delta_{\bar{\Sigma}},
\end{equation}
where $Q_{\bar{\Sigma}} := N(0, \bar{\Sigma})$, and where $\delta_{\bar{\Sigma}}$ denotes point mass at $\bar{\Sigma} \in \R^{k \times k}$. Now, define the map $F: \R^k \times \R^{k \times k}_{s, \geq 0}\to\R$ via $(z, \Sigma) \mapsto \|z\|_{\infty} - K_{1-\alpha}(\Sigma)$, where $\R^{k \times k}_{s, \geq 0}$ denotes the set of real, symmetric and nonnegative definite $k \times k$ dimensional matrices, and note that the map $F$ is continuous everywhere (cf. Lemma~\ref{lemma:POSIcont}). It follows from the continuous mapping theorem together with the previous display that
\begin{equation}
\mathbb{Q}_{n''} := \mathbb{P}_{n''} \circ F(\bar{r}_{{n''}, *}, \corr(\hat{S}_{n''})) \Rightarrow Q_{\bar{\Sigma}} \circ \left(\|.\|_{\infty} - K_{1-\alpha}(\bar{\Sigma})\right).
\end{equation}
Since the diagonal elements of $\bar{\Sigma}$ are ones (by its definition together with Condition \ref{cond:sum}), one can easily show that the $Q_{\bar{\Sigma}}$-probability of $\|.\|_{\infty} - K_{1-\alpha}(\bar{\Sigma})$ being equal to $0$ is $0$. It hence follows from the Portmanteau theorem, together with the definition of $ K_{1-\alpha}(\bar{\Sigma})$ and the previous display, that
\begin{align}
&\mathbb{P}_{n''} \left( \| \diag(\hat{S}_{n''})^{\dagger/2} (\hat{\theta}_{n''} - \theta^*_{n''}) \|_{\infty} \leq K_{1-\alpha}(\corr(\hat{S}_{n''})) \right) \\ 
= \quad &\mathbb{Q}_{n''} \left((-\infty, 0]  \right)
\to Q_{\bar{\Sigma}} \left(y \in \R^k: \|y\|_{\infty} \leq K_{1-\alpha}(\bar{\Sigma})\right) = 1-\alpha.
\end{align}
This finishes the proof. \hfill\qed

\subsection{Proof of Proposition \ref{prop:obtest}}

As in the proof of Lemma \ref{lemma:equivSn:full}, Lemma \ref{lemma:nonstd} applied to the array $z_{i,n} := g_{i,n} \circ \pi_{i,n}$ defined on $(\R^{n \times \ell}, \mathcal{B}(\R^{n \times \ell}), \mathbb{P}_n)$ shows (in particular) that for every $\eps > 0$ 
\begin{align}
&\mathbb{P}_n \left( \|\diag(\mathbb{V}_n(r_n))^{-1} \diag(S_n)  - I_k\| \geq \varepsilon \right) \to 0.
\end{align}
This shows that the two conditions given in the statement of the proposition are indeed equivalent, and, together with Condition \ref{cond:sum}, it also shows that for every $j = 1, \hdots, k$ we have $\mathbb{P}_n(\sum_{i = 1}^n [z^{(j)}_{i,n}]^2 = 0 ) \to 0$. Now, for $j = 1, \hdots, k$, we apply the first part of Lemma~\ref{lemma:auxi} (with $a_{i,n} = g^{(j)}_{i,n} \circ \pi_{i,n}$ and $b_{i,n} = \hat{g}^{(j)}_{i,n}$) to obtain for every $\eps > 0$ that
\begin{equation}
\mathbb{P}_n \left( \|\diag(S_n)^{\dagger} \diag(\hat{S}_n)  - I_k\| \geq \varepsilon \right) \to 0.
\end{equation}
Next, we can, in a similar way, apply the second part of Lemma \ref{lemma:auxi} to obtain
\begin{equation}
\mathbb{P}_n \left( \|\corr \left( \hat{S}_n \right) - \corr \left(  S_n \right)\| \geq \varepsilon \right) \to 0.
\end{equation}
This finishes the proof. \hfill\qed

\subsection{Proof of Theorem \ref{thm:cons}}

Similarly as in the proof of Theorem \ref{thm:fbl} we now need to verify that
\begin{equation}\label{eqn:statem}
\liminf_{n \to \infty} \mathbb{P}_n \left( \| \diag(\hat{\nu}^2_n)^{\dagger/2} \left[r_n + \Delta_n \right]\|_{\infty} \leq \hat{K}_n \right) \geq 1- \alpha.
\end{equation}
We make the following preparatory observation: Denote the event on which $\kappa_n$ is well defined by $A_n$ (recall that $\mathbb{P}_n(A_n) \to 1$), and let $\eps > 0$. Observe that the limit inferior in Equation \eqref{eqn:statem} is not smaller than
\begin{equation}
\liminf_{n \to \infty}
 \mathbb{P}_n \left( \kappa_n \| \diag(\mathbb{V}_n(r_n))^{-1/2} \left[r_n + \Delta_n \right] \|_{\infty} \leq  K_{1-\alpha}(\corr(\mathbb{V}_n(r_n))), A_n \right),
\end{equation}
which, in turn, is bounded from below (using that $\kappa_n$ is positive on $A_n$, and Equation \eqref{eqn:asptup}) by
\begin{equation}\label{eq:liminfbd}
\liminf_{n \to \infty} \mathbb{P}_n \left( \| \diag(\mathbb{V}_n(r_n))^{-1/2} \left[r_n + \Delta_n \right] \|_{\infty} \leq \frac{K_{1-\alpha}(\corr(\mathbb{V}_n(r_n)))}{1 + \eps} \right).
\end{equation}
Now, we argue by contradiction, and suppose that \eqref{eqn:statem} is false: Then there exists a $\delta > 0$, that can be chosen independently of $\eps$, so that the limit inferior in \eqref{eq:liminfbd} is an element of $[0, 1-\alpha - \delta)$. Next, let $n'(\eps)$ denote a subsequence along which \eqref{eq:liminfbd} is attained. Arguing as in the proof of Theorem \ref{thm:fbl} (borrowing some of its notation) we can obtain a subsequence $n''(\eps)$ of $n'(\eps)$ along which the sequence of probabilities in the preceding display converges to
\begin{align}
Q_{\bar{\Sigma}(\eps)}\left(y \in \R^k: \|y\|_{\infty} \leq \frac{K_{1-\alpha}(\bar{\Sigma}(\eps))}{1+\eps} \right) \in [0, 1-\alpha - \delta).
\end{align}
Note that $\eps > 0$ was arbitrary, and let $\eps_m > 0$ converge to $0$. Assume (otherwise pass to a subsequence) that the sequence of correlation matrices $\bar{\Sigma}(\eps_m)$ (with diagonal entries equal to $1$) converges to $\bar{\Sigma}$, say. It is then not difficult to obtain (by a weak convergence argument involving Portmanteau theorem) the contradiction 
\begin{align}
& Q_{\bar{\Sigma}(\eps_m)}\left(y \in \R^k: \|y\|_{\infty} \leq \frac{K_{1-\alpha}(\bar{\Sigma}(\eps_m))}{1+\eps_m} \right) \\ 
&\xrightarrow[m\to\infty]{} \quad  Q_{\bar{\Sigma}}\left(y \in \R^k: \|y\|_{\infty} \leq K_{1-\alpha}(\bar{\Sigma}) \right) = 1-\alpha.
\end{align}
The remaining part follows immediately from what we have already established. \hfill\qed

\subsection{Proof of Proposition \ref{prop:overest}}

Fix $j$ and note that with the same notation and argumentation as in the beginning of the proof of Proposition~\ref{prop:obtest}, and with the convention of Remark~\ref{rem:gin}, for every $\eps > 0$, it holds that
\begin{equation}\label{eqn:overest1}
\mathbb{P}_n \left(\left|\frac{\sum_{i = 1}^n [g_{i,n}^{(j)}]^2}{\sum_{i = 1}^n \mathbb{V}_n(g_{i,n}^{(j)})} - 1 \right| \geq \eps \right) \to 0.
\end{equation}
Obviously, 
\begin{equation}
\tilde{g}_{i,n}^{(j)} = \left( \hat{g}_{i,n}^{(j)} - g_{i,n}^{(j)}\right) + g_{i,n}^{(j)} + a_{i,n}^{(j)}. 
\end{equation}
Equation~\eqref{eq:negldiff}, or equivalently (equivalence being due to  Equation~\eqref{eqn:overest1} above) Equation~\eqref{eq:negldiff2}, together with Part~3 of Lemma~\ref{lemma:auxi} (applied with: $a_{i,n} = g_{i,n}^{(j)}$, $b_{i,n} = \left( \hat{g}_{i,n}^{(j)} - g_{i,n}^{(j)}\right) + g_{i,n}^{(j)}$ and $c_{i,n} = a_{i,n}^{(j)}$) now shows that
\begin{equation}
\mathbb{P}_n \left( \frac{\sum_{i = 1}^n [\tilde{g}_{i,n}^{(j)}]^2 }{\sum_{i = 1}^n \mathbb{V}_n(g_{i,n}^{(j)})} \leq 1-\eps \right) \to 0 \text{ for every } \eps > 0,
\end{equation}
implying the claimed statement. Note that Equation~\eqref{eqn:maxto0} in Lemma~\ref{lemma:auxi} is satisfied here because Condition~\ref{cond:sum} (in particular the Lindeberg condition in Equation \eqref{eq:linde}) implies the corresponding Feller condition
\begin{equation}\label{eqn:feller}
\frac{\max_{i = 1, \hdots, {n}} \mathbb{V}_{n}(g_{i,{n}}^{(j)})}{\sum_{i = 1}^{n} \mathbb{V}_{n}(g_{i,{n}}^{(j)})} \to 0.
\end{equation}
All remaining assumptions in Part 3 of Lemma \ref{lemma:auxi} can be easily checked using Condition~\ref{cond:sum}, \eqref{eq:negldiff} and \eqref{eqn:overest1}. \hfill\qed

\subsection{Proof of Lemma \ref{lem:upper}} 

Let $\omega := \rank(\Gamma)$ and let $Z \sim N(0, \Gamma)$. Since $\Gamma$ is a correlation matrix of rank $\omega$, by the spectral decomposition, we can find a $k \times \omega$-dimensional matrix $V$ so that $VV' = \Gamma$. In particular if $\eps \sim N(0, I_{\omega})$ it holds that $V \eps \sim N(0, \Gamma)$, and hence the $1-\alpha$-quantiles of the distributions of $\|Z\|_{\infty}$ and of $\| V\eps\|_{\infty} = \max_{i = 1, \hdots, k } |v_i \eps |$ coincide, $v_i$ denoting the $i$-th row of $V$. Since $\Gamma$ is a correlation matrix it furthermore holds that each row $v_i$ of $V$ has Euclidean norm less than or equal to $1$. From the discussion after the definition of $B_{\alpha}$ it then follows that $K_{1-\alpha}(\Gamma)$, the $1-\alpha$-quantile of the distributions of $\|Z\|_{\infty}$, is not greater than $B_{\alpha}(\omega, k)$. \hfill\qed


\section{Proofs for Section 3}
\label{app:appl}

\subsection{Proof of Proposition~\ref{prop:lmvarest}}
\label{sec:lmvarest}

Fix $\delta>0$, $\tau\ge1$, $n\in\N$, $\P_n\in\mathbf{P}_n^{(\mathrm{lm})}(\delta,\tau)$ and $\M\in\mathsf{M}_n$ with corresponding index set $M\in\mathcal I$. Let $\mu_n$ and $\sigma_n^2$ be the mean vector and the (component-wise) variance of $\P_n$, as defined in Section~\ref{sec:homlinmods}. Abbreviate $m = |M|$, $u_n = u_n(y) = y - \mu_n$ and $H_M = I_n-P_{X_n[M]}$. The mean of $\hat{\sigma}_{\M,n}^2(y) = (n-m)^{-1} y'H_My$ is easily seen to be
$\E_n[\hat{\sigma}_{\M,n}^2] = \mu_n'H_M\mu_n/(n-m) + \sigma_n^2$, eventually.
Now consider
\begin{align}
&\left|\frac{\hat{\sigma}_{\M,n}^2}{\E_n[\hat{\sigma}_{\M,n}^2]} - 1\right|
\quad=\quad
\left|\frac{2\mu_nH_M u_n/(n-m) + u_n'H_M u_n/(n-m) - \sigma_n^2}{\mu_n' H_M \mu_n/(n-m) + \sigma_n^2}\right|\notag\\
&\quad=\quad
\left|\frac{2(\mu_n/\sigma_n)H_M (u_n/\sigma_n)/(n-m) + (u_n/\sigma_n)'H_M (u_n/\sigma_n)/(n-m) - 1}{(\mu_n/\sigma_n)' H_M (\mu_n/\sigma_n)/(n-m) + 1}\right|\notag\\
&\quad\le\quad
2\left|\frac{(\mu_n/\sigma_n)H_M (u_n/\sigma_n)}{(\mu_n/\sigma_n)'H_M (\mu_n/\sigma_n) + (n-m)}\right| +
\left| \frac{(u_n/\sigma_n)'H_M (u_n/\sigma_n)}{n-m} - 1\right|. \label{eq:prop:ConservativeLinMod}
\end{align}
The first fraction on the last line of the previous display converges to zero in $\P_n$-probability, because its mean is $0$ and its variance is upper bounded by $\|H_M \mu_n/\sigma_n\|^2/( \|H_M \mu_n/\sigma_n\|^4 + (n-m)^2 ) $, which converges to $0$, as is seen by maximizing it with respect to $\|H_M \mu_n/\sigma_n\|^2$. To show that the second fraction converges to one, abbreviate the random $n$-vector $v_n = u_n/\sigma_n$, and note that $v_n$ has independent standardized components under $\P_n$. Now decompose the quadratic form as
$$
\frac{v_n'H_Mv_n}{n-m} \quad=\quad \frac{1}{n-m}\sum_{i=1}^n (H_M)_{ii} v_{i,n}^2 + \frac{1}{n-m}\sum_{i\ne j} (H_M)_{ij} v_{i,n}v_{j,n},
$$
and note that $(n-m)^{-1}\sum_{i\ne j} (H_M)_{ij} v_{i,n}v_{j,n}$ has mean zero and variance equal to $(n-m)^{-2}\sum_{i\ne j} [(H_M)_{ij}]^2\le (n-m)^{-2}\trace(H_M^2) = (n-m)^{-1}\to0$. To show that $(n-m)^{-1}\sum_{i=1}^n (H_M)_{ii} v_{i,n}^2$ converges to one, we use a standard truncation argument. For $K>0$ define $\tilde{v}_{i,n} = v_{i,n}\{|v_{i,n}|\le K\}$, $S_n = \frac{1}{n-m}\sum_{i=1}^n(H_M)_{ii} v_{i,n}^2$, $\tilde{S}_n = \frac{1}{n-m}\sum_{i=1}^n(H_M)_{ii} \tilde{v}_{i,n}^2$ and
\begin{align*}
D_n &:= S_n - \tilde{S}_n 
=
\frac{1}{n-m}\sum_{i=1}^n(H_M)_{ii} v_{i,n}^2\{|v_{i,n}|>K\} \;\ge\; 0.
\end{align*}
Using first H\"older's inequality, and then Markov's inequality (recall that $\mathbb{E}(v^2_{i,n}) = 1$), $\max_{i = 1, \hdots, n} \left(\E_n[|v_{i,n}|^{2+\delta}]\right)^{2/(2+\delta)} \leq \tau$ (recall that $\P_n\in\mathbf{P}_n^{(\mathrm{lm})}(\delta,\tau)$), and $\trace(H_M) = n-m$, the mean of $D_n$ can be bounded by
\begin{align*}
\E_n[D_n] &\le \frac{1}{n-m}\sum_{i=1}^n(H_M)_{ii} \left(\E_n[|v_{i,n}|^{2+\delta}]\right)^{2/(2+\delta)}\P_n(|v_{i,n}|^2>K^2)^{\delta/(2+\delta)}\\
&\le 
\tau K^{-\frac{2\delta}{2+\delta}}.
\end{align*}
Now for $\eps>0$,
\begin{align*}
&\P_n(|S_n-1|>\eps) \\
&\quad\le
\P_n(|S_n-\tilde{S}_n|>\eps/2)  +  \P_n(|\tilde{S}_n- \E_n[\tilde{S}_n] + \E_n[\tilde{S}_n]-1|>\eps/2)\\
&\quad\le
2\tau K^{-\frac{2\delta}{2+\delta}}/\eps + 16\V_n[\tilde{S}_n]/\eps^2 + \P_n(\tau K^{-\frac{2\delta}{2+\delta}}>\eps/4).
\end{align*}
Since the variance of $\tilde{S}_n$ clearly converges to zero as $n\to\infty$, for every $K>0$, the limit superior of $\P_n(|S_n-1|>\eps)$ is bounded by a quantity that approaches zero as $K\to\infty$. Thus, we have established the convergence $\hat{\sigma}_{\M,n}^2/\E_n[\hat{\sigma}_{M,n}^2] \to 1$, in $\P_n$-probability.
Therefore, we can write
\begin{align*}
\frac{\hat{\sigma}_{\M,n}^2}{\sigma_n^2} 
=
\frac{\hat{\sigma}_{\M,n}^2}{\E_n[\hat{\sigma}_{\M,n}^2]} \frac{\E_n[\hat{\sigma}_{\M,n}^2]}{\sigma_n^2}
=
(1+o_{\P_n}(1))  
\left(
\frac{(\mu_n/\sigma_n)'H_M(\mu_n/\sigma_n)}{n-m} + 1
\right),
\end{align*}
which finishes the proof.\hfill\qed

\subsection{Proof of Theorem \ref{thm:homlinmods}}

For every $n \in \N$, let $\mathbb{P}_n \in \mathbf{P}_n^{(\mathrm{lm})}(\delta, \tau)$. Abbreviate $u_{i,n}(y_i) := y_i - \mu_{i,n}$, and $u_n(y) = (u_{1,n}(y_1), \hdots, u_{n,n}(y_n))'$, where $y=(y_1,\dots, y_n)'\in\R^n$ and $\mu_n = (\mu_{1,n},\dots, \mu_{n,n})'$. We note that by assumption
\begin{equation}\label{eqn:astau}
\frac{\max_{i = 1, \hdots, n} \mathbb{E}_{i,n} (|u_{i,n}|^{2+\delta})^{\frac{2}{2+\delta}}}{\sigma_n^2} \leq \tau.
\end{equation}
We now verify Condition \ref{cond:sum} (with $\Delta_n \equiv  0$): Since this condition is formulated in a component-wise fashion, it suffices to verify it for an arbitrary component $j\in\{1,\dots, m(\M)\}$ of the estimation error $r_{n,\M}$ in an arbitrary model $\M\in\mathsf{M}_n$ with corresponding index set $M\in\mathcal I$. Recall that 
\begin{equation}
\beta_{\mathbb{M}, n}^{*} = \left(X_n[M]' X_n[M] \right)^{-1} X_n[M]'\mu_n = U_n^{-1} X_n[M]'\mu_n,
\end{equation}
for $U_n := X_n[M]' X_n[M] $. The $j$-th coordinate of the estimation error
\begin{equation}
r_{n,\M}(y) \;:=\;  \hat{\beta}_{\M,n}(y) - \beta_{\mathbb{M}, n}^{*}  = U_n^{-1}X_n[M]'u_n(y)
\end{equation}
can be written as
\begin{equation}
\sum_{i = 1}^n e_{|M|}'(j)U_n^{-1} X_{i,n}[M]' u_{i,n}(y_i) =: \sum_{i=1}^n g^{(j)}_{i,n,\M}(y_i),
\end{equation}
where $e_m(j)$ denotes the $j$-th element of the canonical basis of $\R^m$. 
By definition $\mathbb{E}_{i,n} \left( g_{i,n,\M}^{(j)} \right) = 0$  and $0 < \mathbb{V}_{n} \left( r_{n,\M}^{(j)} \right) < \infty$ holds (eventually), the latter following from
\begin{equation}\label{eqn:varhomlin}
\mathbb{V}_{n} \left( r_{n,\M} \right) = \sigma_n^2 U_n^{-1}
\end{equation}
together with Condition~\ref{cond:X}. We now verify that for every $\eps > 0$ it holds that
\begin{equation}\label{eqn:lindehomreg}
\mathbb{V}_n^{-1}\left(r_{n,\M}^{(j)}\right) \sum_{i = 1}^n \int_\R \left[ g_{i,n,\M}^{(j)}\right]^2 \left\{ |g_{i,n,\M}^{(j)}| \geq \eps \mathbb{V}_n^{\frac{1}{2}}\left(r_{n,\M}^{(j)}\right) \right\} d \mathbb{P}_{i,n} \to 0.
\end{equation}
An application of H\"older's inequality (with $p = \frac{2+\delta}{2}$ and $q = \frac{2+\delta}{\delta}$) shows that the quantity to the left in the previous display is bounded from above by
\begin{equation}
\frac{\max_{i = 1, \hdots, n} \mathbb{E}_{i,n} (|u_{i,n}|^{2+\delta})^{\frac{2}{2+\delta}}}{\sigma_n^2}\left( \max_{i = 1, \hdots, n} \mathbb{P}_{i,n} \left( | g_{i,n,\M}^{(j)}  | \geq \eps \mathbb{V}_n^{\frac{1}{2}}\left(r_{n,\M}^{(j)}\right) \right) \right)^{\frac{\delta}{2 + \delta}} ,
\end{equation}
which, using the bound \eqref{eqn:astau} and Markov's inequality, does not exceed
\begin{equation}\label{eq:thm:homlinmods:w2}
\tau ~ \left( 
\frac{\max_{i = 1, \hdots, n}  \left( U_n^{-1}X_{i,n}[M]'X_{i,n}[M]U_n^{-1}\right)_{j} }{ \left( U_n^{-1}\right)_{j}}
\right)^{\frac{\delta}{2 + \delta}} ~ \eps^{-\frac{2\delta}{2 + \delta}} .
\end{equation}
Finally, since the term within brackets coincides with
\begin{equation}
\max_{i = 1, \hdots, n} \bigg( X_{i,n}[M] U_n^{-\frac{1}{2}} \frac{W_{j,n}}{\|W_{j,n}\|} 
U_n^{-\frac{1}{2}} X_{i,n}[M]' \bigg),
\end{equation}
for 
\begin{align}
W_{j,n} = 
U_n^{-\frac{1}{2}} e_{|M|}(j) e'_{|M|}(j) U_n^{-\frac{1}{2}},
\end{align}
and since this quantity is not greater than
\begin{equation}
\max_{i = 1, \hdots, n} X_{i,n}[M]U_n^{-1} X_{i,n}[M]' \to 0,
\end{equation}
where convergence holds by Condition \ref{cond:X}, the statement in \eqref{eqn:lindehomreg} follows. Since $\M$ and $j$ were arbitrary, we have verified Condition~\ref{cond:sum}. 

Now, for $\M$, $M$ and $j$ as before, set $\hat{\nu}_{j,n,\M}^2 = \hat{\sigma}_{\M,n}^2 (U_n^{-1})_{j}$. Thus, using \eqref{eqn:varhomlin} and Proposition~\ref{prop:lmvarest}, we see that

\begin{align*}
\mathbb{P}_n \left( \sqrt{ \frac{[\mathbb{V}_n(r_{n,\M})]_j
}{\hat{\nu}_{j,n,\M}^2 }} \geq 1+\eps   \right) 
=
\mathbb{P}_n \left( \sqrt{ \frac{\sigma_n^2
}{\hat{\sigma}_{\M,n}^2 } } \geq 1+\eps   \right) 
\to 0.
\end{align*}
Finally, note that for the stacked vector $r_n = (r_{n,\M_1}',\dots, r_{n,\M_d}')'$, we have
\begin{equation}
\mathbb{V}_n(r_n) = \sigma_n^2 \Gamma_n, 
\end{equation}
and that $K_{1-\alpha}(\corr(\Gamma_n)) = K_{1-\alpha}(\corr(\mathbb{V}_n(r_n))) > 0$. It now follows from Theorem \ref{thm:cons} (the special case with $\hat{K}_n = K_{1-\alpha}(\corr(\Gamma_n))$, that for any (measurable) model selection procedure $\hat{\mathbb{M}}_n$ it holds that

\begin{equation}
\liminf_{n \to \infty} \mathbb{P}_n \left(\beta_{\hat{\mathbb{M}}_n, n}^{*, (j)} \in \mathrm{CI}_{1-\alpha, \hat{\mathbb{M}}_n}^{(j), \mathrm{lm}} \text{ for all } j = 1, \hdots, m(\hat{\mathbb{M}}_n)  \right) \geq 1-\alpha.
\end{equation}

The theorem now follows because the selection $\mathbb{P}_n$ we started with was arbitrary. \hfill\qed

\subsection{Proof of Theorem \ref{thm:homlinmodsindi}}
\label{subsection:proof:thm:homlinmodsindi}

The proof is analogous to that of Theorem~\ref{thm:homlinmods}, with the only modification that throughout $j=1$ and that the stacked vector $r_n$ is now given by $r_n = (r_{n,\M_1}^{(1)},\dots, r_{n,\M_d}^{(1)})'$, so that
$\mathbb{V}_n(r_n) = \sigma_n^2 \Xi_n$ and thus $K_{1-\alpha}(\mathrm{corr}(\Xi_n)) = K_{1-\alpha}(\mathrm{corr}(\mathbb{V}_n(r_n))) > 0$.\hfill\qed


\subsection{Proof of Proposition~\ref{prop:no:consistent:sigma}} \label{subsection:proof:no:consistent:sigma}

We argue by contradiction, and assume existence of a sequence of measurable functions $(\hat{\sigma}^2_{n})_{n \in \mathbb{N}}$  with $\hat{\sigma}^2_{n}: \mathbb{R}^n \to [0, \infty)$ and so that for every $\varepsilon > 0$ Equation \eqref{eqn:conscounter} holds. First, we define for every $x \in \mathbb{R}^n$ and every $\rho \geq 0$  the product measure
\begin{equation}
Q_n(x, \rho) = \bigotimes_{i = 1}^n \bigg( Q(x_i) \ast N(0, \rho)\bigg),
\end{equation}
where $Q(x_i)$ puts mass $1/2$ to $x_i + 1$ and to $x_i - 1$, respectively, and $Q(x_i) \ast N(0, \rho)$ denotes the convolution of $Q(x_i)$ and $N(0, \rho)$, where we interpret $N(0, 0)$ as point mass at $0$, i.e., $Q(x_i) \ast N(0, 0) = Q(x_i)$. For simplicity, denote the $q$-th absolute central moment of a distribution $F$ on the Borel sets of $\R$ by $m_q(F)$. We note that $Q_n(x, 0) \in \mathbf{P}_n^{(\mathrm{lm})}(\delta, \tau)$, because $m_{2+h}( Q(x_i))^{\frac{2}{2+h}} = 1$ for every $h \geq 0$, and since $\tau > 1$ holds by assumption. It is easy to verify that for every $h \geq 0$ the quantity $m_{2+h}( Q(x_i)  \ast N(0, \rho) )^{\frac{2}{2+h}}$ does not depend on $x_i$, and that we have $m_{2+h}( Q(x_i)  \ast N(0, \rho) )^{\frac{2}{2+h}} \to 1$ as $\rho \to 0$. Hence there exists a $\rho^* > 0$, so that for every $\rho \in[0,\rho^*]$ and every $x \in \mathbb{R}^n$ we have $Q_n(x, \rho) \in \mathbf{P}_n^{(\mathrm{lm})}(\delta, \tau)$. Therefore, by our assumption, it holds for every $\varepsilon > 0$ that
\begin{equation}\label{eqn:contra}
\sup_{x \in \mathbb{R}^n}  \sup_{0 \leq \rho \le \rho^*} ~ Q_n(x, \rho) \left( \bigg | \frac{\hat{\sigma}^2_n}{1 + \rho} - 1 \bigg | > \varepsilon \right) \to 0.
\end{equation}
Next, let $(F, \mathcal{F}, \mathbb{Q})$ be a probability space on which, for every $n$, there are defined two independent random $n$-vectors $X_n^{(1)}$ and $X_n^{(2)}$, so that $\mathbb{Q}
\circ X_n^{(1)} \sim Q_n(0, 0)$ and $\mathbb{Q}
\circ X_n^{(2)} \sim N(0, \rho^*I_n)$, and hence the distribution of $Y_n := X_n^{(1)} + X_n^{(2)}$ is $Q_n(0, \rho^*)$. Let $\varepsilon > 0$ be fixed. From the previous display it follows that $\mathbb{Q}(|\hat{\sigma}^2_n(Y_n) - (1 + \rho^*)  | > \varepsilon )$ converges to $0$. Since the conditional distribution of $Y_n$ given $X_n^{(2)}$ is $Q_n(X_n^{(2)}, 0)$, it furthermore holds that
\begin{equation}
\mathbb{Q}(|\hat{\sigma}^2_n(Y_n) - (1 + \rho^*)  | > \varepsilon ) = \mathbb{E}_{\mathbb{Q}} \left( Q_n(X_n^{(2)}, 0)(|\hat{\sigma}^2_n - (1 + \rho^*)  | > \varepsilon )\right) \to 0,
\end{equation}
from which it now follows that $I_n(\omega) = Q_n(X_n^{(2)}(\omega), 0)(|\hat{\sigma}^2_n - (1 + \rho^*)  | > \varepsilon ) \to 0$ in $\mathbb{Q}$-probability as $n \to \infty$. Thus, there exists a subsequence $n'$ so that $I_{n'} \to 0$, $\mathbb{Q}$-almost surely. As a consequence, there exists $\bar{\omega} \in \Omega$ for which $I_{n'}(\bar{\omega}) \to 0$ as $n' \to \infty$. But this now means that for $x_{n'} = X_{n'}^{(2)}(\bar{\omega})$ it holds that $Q_{n'}(x_{n'}, 0)(|\hat{\sigma}^2_{n'} - (1 + \rho^*)  | > \varepsilon ) \to 0$. Since $\rho^* > 0$, this contradicts Equation \eqref{eqn:contra} which implies $Q_{n'}(x_{n'}, 0)(|\hat{\sigma}^2_{n'} - 1  | > \varepsilon ) \to 0$.
 \hfill\qed


\subsection{Proof of Theorem \ref{thm:hetlinmods}}

We proceed as in the proof of Theorem~\ref{thm:homlinmods}, noting that now we allow for heteroskedasticity, so that $\sigma_{i,n}^2 = \mathbb{V}_{i,n}(u_{i,n})$ depends on $i$. The bound \eqref{eqn:astau} in the proof of Theorem \ref{thm:homlinmods} is now replaced by
\begin{equation}\label{eqn:astauhet}
\frac{\max_{i = 1, \hdots, n} \mathbb{E}_{i,n} (|u_{i,n}|^{2+\delta})^{\frac{2}{2+\delta}}}{\min_{i = 1, \hdots, n} \sigma_{i,n}^2} \leq \tau.
\end{equation}
To verify Condition \ref{cond:sum} we replace \eqref{eqn:varhomlin} in the proof of Theorem \ref{thm:homlinmods} by
\begin{align}\label{eqn:varhetlin}
\mathbb{V}_{n} \left( r_{n,\M} \right) =  
U_n^{-1} \left(\sum_{i = 1}^n \sigma_{i,n}^2 X_{i,n}[M]'X_{i,n}[M] \right)U_n^{-1}
\end{align}
which, replacing each $\sigma_{i,n}^2$ by $\min_{i = 1, \hdots, n} \sigma_{i,n}^2 > 0$, is seen to be eventually positive by Condition~\ref{cond:X}. For the verification of the Lindeberg condition \eqref{eqn:lindehomreg} we use essentially the same argument as in the proof of Theorem \ref{thm:homlinmods}, now using \eqref{eqn:astauhet} above. Hence, Condition~\ref{cond:sum} holds.
Next, we verify \eqref{eqn:special2} in Theorem~\ref{thm:cons} by means of Proposition~\ref{prop:overest}. Let
\begin{align*}
&\tilde{g}_{i,n,\M}(y) = U_n^{-1}X_{i,n}[M]'\hat{u}_{i,\M}(y) \\
&\quad=
U_n^{-1}X_{i,n}[M]'\left( u_{i,n}(y_i)  + X_{i,n}(\beta_{\M,n}^* - \hat{\beta}_{\M,n}(y)) + \mu_{i,n}-X_{i,n}[M]\beta_{\M,n}^*\right)\\
&\quad= \hat{g}_{i,n,\M}(y) + a_{i,n,\M},
\end{align*}
where $\hat{g}_{i,n,\M} = U_n^{-1}X_{i,n}[M]'u_{i,n} - U_n^{-1}X_{i,n}[M]'X_{i,n}[M]r_{n,\M}$ and $a_{i,n,\M}= U_n^{-1}X_{i,n}[M]'(\mu_{i,n}-X_{i,n}[M]\beta_{\M,n}^*)\in\R^{|M|}$, and recall that $g_{i,n,\M} = U_n^{-1}X_{i,n}[M]'u_{i,n}$.
Therefore, 
$$
\hat{\nu}_{j,\M,n}^2 := \hat{\sigma}_{j,\M,n}^2 = \sum_{i=1}^n \left[\tilde{g}_{i,n,\M}^{(j)}\right]^2,
$$
and, from \eqref{eqn:varhetlin}, 
\begin{align*}
&\E_n\left(g_{i,n,\M}^{(j)} - \hat{g}_{i,n,\M}^{(j)}\right)^2
= 
\left[U_n^{-1}X_{i,n}[M]'X_{i,n}[M]\V_n(r_{n,\M})X_{i,n}[M]'X_{i,n}[M]U_n^{-1}\right]_j \\
&\quad\le
(e_{|M|}(j)'U_n^{-1}X_{i,n}[M]')^2 
\left(\max_{i=1,\dots, n}\sigma_{i,n}^2\right) 
\max_{i=1,\dots, n}X_{i,n}[M]U_n^{-1}X_{i,n}[M]'.
\end{align*}
Moreover, $\sum_{i=1}^n \V_n(g_{i,n,\M}^{(j)})=\V_n(r_{n,\M}^{(j)}) \ge [U_n^{-1}]_j\min_{i=1,\dots, n}\sigma_{i,n}^2 $.
Thus, \eqref{eq:negldiff2} follows, because
$$
\sum_{i=1}^n(e_{|M|}(j)'U_n^{-1}X_{i,n}[M]')^2 = [U_n^{-1}]_j
\quad\text{and}\quad\frac{\max_{i=1,\dots, n}\sigma_{i,n}^2}{\min_{i=1,\dots, n}\sigma_{i,n}^2}\le \tau,
$$
in view of \eqref{eqn:astauhet}, and, finally, because of Condition~\ref{cond:X}. Consequently, Proposition~\ref{prop:overest} implies \eqref{eqn:special2} for the proposed variance estimators.
If $r_n$ is now, again, the stacked vector of the $r_{n,\M_1},\dots, r_{n,\M_d}$, then, by Lemma~\ref{lem:upper}, we have that $B_{\alpha}(\min(k, p),k) \geq K_{1-\alpha}(\corr(\mathbb{V}_n(r_n))) > 0$ (it is easy to see that the rank of the $k \times k$-dimensional matrix $\mathbb{V}_n(r_n)$ can not exceed $p$), so we can apply the special case discussed in Theorem~\ref{thm:cons} to conclude that 
\begin{equation}
\liminf_{n \to \infty}  \mathbb{P}_n \left(\beta_{\hat{\mathbb{M}}_n, n}^{*, (j)} \in \mathrm{CI}_{1-\alpha, \hat{\mathbb{M}}_n}^{(j), \mathrm{hlm}} \text{ for all } j = 1, \hdots, m(\hat{\mathbb{M}}_n)  \right) \geq 1-\alpha,
\end{equation}
which proves the claim as the sequence $\mathbb{P}_n \in \mathbf{P}_n^{(\mathrm{het})}(\delta, \tau)$ was arbitrary.\hfill\qed


\subsection{Proof of Lemma~\ref{lemma:pseudo}}
\label{sec:pseudo:proof}

Fix $\M\in\mathsf{M}_n$ and $\P_n\in\bigcup_{\delta>0}\mathbf{P}_n^{(\mathrm{bin})}(\delta)$ and recall that $\M\triangleq (h,M)\in \mathcal H\times \mathcal I$. For $i=1,\dots, n$, we abbreviate $p_i = \P_{i,n}(\{1\})$ and for $\gamma\in\R$, $\phi_1(\gamma) = \log(h(\gamma))$, $\phi_2(\gamma) = \log(1-h(\gamma))$, and we note that $p_i\in(0,1)$. Thus, the expected log-likelihood function can be expressed as
\begin{equation}\label{eq:Eloglik}
\beta \mapsto \int_{\R^n}\ell_{\M,n}(y,\beta)\,d\P_n(y) = \sum_{i=1}^n \left[p_i\phi_1(X_{i,n}[M]\beta) + (1-p_i)\phi_2(X_{i,n}[M]\beta)\right].
\end{equation}
The function in the previous display is continuous on its domain $\R^{|M|}$, by Condition~\ref{cond:H}\eqref{cH:cdf}. To see that it also has a maximizer on $\R^{|M|}$, consider an arbitrary sequence $\beta_k\in\R^{|M|}$ such that $\|\beta_k\|\to\infty$ as $k\to\infty$. Then $\|X_n[M]\beta_k\|^2 \ge \|\beta_k\|^2 \lmin(X_n[M]'X_n[M]) \to \infty$ as $k\to \infty$, by Condition~\ref{cond:X2}\eqref{cX2:rank}, so that at least for a sequence $i_k\in\{1,\dots, n\}$, we must have $|X_{i_k,n}[M]\beta_k|\to \infty$ as $k\to\infty$. Therefore, using the fact that $p_i\in(0,1)$ for $i=1,\dots, n$, it is easy to see that the sequence of summands in \eqref{eq:Eloglik} corresponding to the indices $i_k$, $k\in\N$, with $\beta$ replaced by $\beta_k$, converges to $-\infty$ as $k\to \infty$. Since for each $k$ the remaining summands in \eqref{eq:Eloglik} are non-positive, we see that the expected log-likelihood diverges to $-\infty$ along $(\beta_k)$, and thus, by continuity, attains its maximum at some $\beta_{\M,n}^*=\beta_{\M,n}^*(\P_n,X_n[M])\in\R^{|M|}$. 

For uniqueness, we show that the function in \eqref{eq:Eloglik} is strictly concave.
Take $\beta_1,\beta_2\in\R^{|M|}$, $\beta_1\ne\beta_2$ and $\alpha\in(0,1)$, and note that because $X_n[M]$ is of full rank $|M|$, we must have $X_n[M]\beta_1\ne X_n[M]\beta_2$. Thus, there is at least one $i_0\in\{1,\dots, n\}$ such that $X_{i_0,n}[M]'\beta_1 \ne X_{i_0,n}[M]'\beta_2$, and, by strict concavity (Condition~\ref{cond:H}\eqref{cH:concave}), $\phi_j(\alpha X_{i_0,n}[M]'\beta_1 + (1-\alpha)X_{i_0,n}[M]'\beta_2)> \alpha\phi_j( X_{i_0,n}[M]'\beta_1) + (1-\alpha) \phi_j( X_{i_0,n}[M]'\beta_2)$, for $j=1,2$. For the remaining indices $i\ne i_0$, the same inequalities hold, but are possibly not strict. Therefore, the expected log-likelihood in \eqref{eq:Eloglik} is strictly concave and the maximizer $\beta_{\M,n}^*\in\R^{|M|}$ is unique. \hfill{\qed}


\subsection{Auxiliary results for Section~\ref{sec:binary}}

\begin{lemma}\label{lemma:pseudoBound}
Suppose that Conditions~\ref{cond:X2}(\ref{cX2:rank},\ref{cX2:Huber}) and \ref{cond:H}(\ref{cH:cdf},\ref{cH:concave}) hold and fix $\tau\in(0,1/4)$. There exists a finite positive constant $K^*(\tau, C)$, depending only on $\tau$ and the constant $C$ from Condition~\ref{cond:X2}\eqref{cX2:Huber}, such that eventually 
$$
\sup_{\substack{\M\in\mathsf{M}_n\\\P_n\in\mathbf{P}_n^{(\mathrm{bin})}(\tau)}} \max_{i=1,\dots, n}|X_{i,n}[M]\beta_{\M,n}^*(\P_n)| \;\le\; K^*(\tau, C).
$$ 
Here, $\beta_{\M,n}^*(\P_n)$ is the pseudo parameter from Lemma~\ref{lemma:pseudo}.
\end{lemma}

\begin{proof}
We begin by establishing the following preliminary result. For every pair $(h,M)\in\mathcal H\times\mathcal I$, there exists a bounded set $B_{h,M}(\tau,C)\subseteq \R^{|M|}$, such that eventually $U_{M,n}^{1/2}\beta_{\M,n}^*(\P_n)\in B_{h,M}(\tau,C)$, for all $\P_n\in\mathbf{P}_n^{(\mathrm{bin})}(\tau)$, where $U_{M,n} = X_n[M]'X_n[M]/n$. Here, $C>0$ is the constant from Condition~\ref{cond:X2}. Fix $(h,M)\in\mathcal H\times\mathcal I$ and $t>0$, and define $g_h(t) = \sup_{|\gamma|>t} \min\{\phi_1(\gamma), \phi_2(\gamma)\}$ and the set $B_{h,M}(\tau,C)$ by
$$
B_{h,M}(\tau,C) = \{ v\in\R^{|M|}: g_h(\|v\|/\sqrt{2}) \ge 2C[\phi_1(0) + \phi_2(0)]/\tau\},
$$
where $\phi_1$ and $\phi_2$ are as in Condition~\ref{cond:H}\eqref{cH:concave}. Note that, indeed, $B_{h,M}(\tau,C)$ is bounded, because $g_h(t) \to -\infty$ as $t\to\infty$, in view of Condition~\ref{cond:H}\eqref{cH:cdf}. Next, fix $n$ large enough, such that Conditions~\ref{cond:X2}(\ref{cX2:rank},\ref{cX2:Huber}) hold. Then the pseudo parameter $\beta_{\M,n}^*$ of Lemma~\ref{lemma:pseudo} uniquely exists. For $\beta\in\R^{|M|}$ and $\xi>0$, define $R_{\beta,n}(\xi) = \{ i\le n : |X_{i,n}[M]\beta|\ge \xi\|U_{M,n}^{1/2}\beta\|\}$. Now fix $\xi>0$ and $\beta\in\R^{|M|}$ such that $\|U_{M,n}^{1/2}\beta\|=1$, and observe that
\begin{align*}
1 \;&=\; \beta'X_n[M]'X_n[M]\beta/n \;=\; \frac{1}{n} \sum_{i=1}^n (X_{i,n}[M]U_{M,n}^{-1/2}U_{M,n}^{1/2}\beta)^2 \\
&\le\; \xi^2 + \frac{1}{n} \sum_{i\in R_{\beta,n}(\xi)} \|U_{M,n}^{-1/2}X_{i,n}[M]'\|^2
\;\le\;\frac{1}{n} |R_{\beta,n}(\xi)| C \;+\; \xi^2,
\end{align*}
which implies that  $\inf_{\beta\in\R^{|M|}} |R_{\beta,n}(\xi)| = \inf_{\beta:\|U_{M,n}^{1/2}\beta\|=1} |R_{\beta,n}(\xi)| \ge n(1-\xi^2)/C$. Since $\phi_1$ and $\phi_2$ are negative, we get for every $\P_n\in\mathbf{P}_n^{(\mathrm{bin})}(\tau)$, every $\beta\in\R^{|M|}$ and for $\xi=1/\sqrt{2}$, that
\begin{align*}
\E_n[\ell_{\M,n}(\cdot,\beta)] \;&\le\; \sum_{i\in R_{\beta,n}(\xi)}  \P_{i,n}(\{1\})\P_{i,n}(\{0\})  \Big( \phi_1(X_{i,n}[M]\beta) + \phi_2(X_{i,n}[M]\beta)\Big)\\
&\le \; \tau \,|R_{\beta,n}(\xi)| \,g_h(\xi\|U_{M,n}^{1/2}\beta\|) \;\le\; n \frac{\tau}{2C} g_h(\|U_{M,n}^{1/2}\beta\|/\sqrt{2}),
\end{align*}
where $\M \triangleq (h,M)$.
Therefore, we have $n[\phi_1(0) + \phi_2(0)] \le \E_n[\ell_{\M,n}(\cdot, 0)] \le \E_n[\ell_{\M,n}(\cdot, \beta_{\M,n}^*(\P_n)] \le n \tau g_h(\|U_{M,n}^{1/2}\beta_{\M,n}^*\|/\sqrt{2}) / (2C)$, which yields 
$$
g_h(\|U_{M,n}^{1/2}\beta_{\M,n}^*\|/\sqrt{2}) \;\ge \; \frac{[\phi_1(0) + \phi_2(0)] 2C}{\tau},
$$
i.e., $U_{M,n}^{1/2}\beta_{\M,n}^*(\P_n)\in B_{h,M}(\tau, C)$. So we have established the preliminary result. 
Since the bounded set $B_{h,M}(\tau, C)$ depends only on the indicated quantities, there exists a finite positive constant $K^*_{h,M}(\tau, C)$, depending on the same quantities, such that for all large $n$, for all $\P_n\in\mathbf{P}_n^{(\mathrm{bin})}(\tau)$ and for all $\M\in\mathsf{M}_n$,
\begin{align*}
&\max_{i}|X_{i,n}[M]\beta_{\M,n}^*(\P_n)| \le \max_i \|U_{M,n}^{-1/2}X_{i,n}[M]'\| \|U_{M,n}^{1/2}\beta_{\M,n}^*(\P_n)\|\\
&\quad\le
\sqrt{C}K_{h,M}^*(\tau, C) \le \max_{(h,M)\in\mathcal H\times \mathcal I}\sqrt{C}K_{h,M}^*(\tau, C) =:K^*(\tau, C).
\end{align*}
This finishes the proof.
\end{proof}

\begin{lemma}\label{lemma:Hcontinuity}
Suppose that Conditions~\ref{cond:X2}(\ref{cX2:rank},\ref{cX2:Huber}) and \ref{cond:H}(\ref{cH:cdf},\ref{cH:concave},\ref{cH:C2}) hold and fix $\tau\in(0,1/4)$. 
\begin{list}{\thercnt}{
	 \usecounter{rcnt}
        \setlength\itemindent{0pt}
        \setlength\leftmargin{0pt}
        \setlength\partopsep{0pt} 
}
	\renewcommand{\theenumi}{(\roman{enumi})}
	\renewcommand{\labelenumi}{{\theenumi}} 
\item \label{l:Hcont:A}
There exist positive constants $\overline{K}(\tau, C)$ and $\underline{K}(\tau, C)$, depending only on $\tau$ and the constant $C$ from Condition~\ref{cond:X2}\eqref{cX2:Huber}, such that for all sufficiently large $n\in\N$, for all $\M\in\mathsf{M}_n$, all $\P_n\in\mathbf{P}_n^{(\mathrm{bin})}(\tau)$ and all $y\in\{0,1\}^n$, 
\begin{align*}
\overline{K}(\tau, C) &\ge
\lmax\left((X_n[M]'X_n[M])^{-1/2}H_{\M,n}^*(y)(X_n[M]'X_n[M])^{-1/2}\right)\\
&\ge\lmin\left((X_n[M]'X_n[M])^{-1/2}H_{\M,n}^*(y)(X_n[M]'X_n[M])^{-1/2}\right) \ge \underline{K}(\tau, C),
\end{align*}
where $H_{\M,n}^*(y) := H_{\M,n}(y,\beta_{\M,n}^*(\P_n))$, $H_{\M,n}(y,\beta) = -\frac{\partial^2 \ell_{\M,n}(y,\beta)}{\partial\beta\partial\beta'}$ with $\beta\in\R^{m(\M)}$ and $\beta_{\M,n}^*(\P_n)$ is the pseudo parameter of Lemma~\ref{lemma:pseudo}.

\item \label{l:Hcont:B}
For $\delta>0$ and $n$ sufficiently large, such that the pseudo parameter $\beta_{\M,n}^*$ of Lemma~\ref{lemma:pseudo} exists, define 
$$
N_{\M,\P_n,n}(\delta) = \left\{ \beta\in\R^{m(\M)} : \left\| (X_n[M]'X_n[M])^{1/2}(\beta - \beta_{\M,n}^*(\P_n)) \right\| \le \delta\right\}.
$$
Then, for every $\delta>0$,
$$
\sup_{\substack{\M\in\mathsf{M}_n\\\P_n\in\mathbf{P}_n^{(\mathrm{bin})}(\tau)}}
\sup_{y\in\{0,1\}^n}\sup_{\beta\in N_{\M,\P_n,n}(\delta)} \left\| H_{\M,n}^*(y)^{-1/2}H_{\M,n}(y,\beta)H_{\M,n}^*(y)^{-1/2} - I_{m(\M)} \right\|
$$
converges to zero as $n\to\infty$.

\item \label{l:Hcont:C}
Suppose that, in addition, also Condition~\ref{cond:X2}\eqref{cX2:eigen} holds. Then there exists a positive finite constant $K(\tau, C)$, depending only on $\tau$ and the constant $C$ from Condition~\ref{cond:X2}, such that eventually
$$
\sup_{\M\in\mathsf{M}_n}\sup_{\P_n\in\mathbf{P}_n^{\mathrm{(bin)}}(\tau)}\frac{\lmax(\E_{\P_n}[H_{\M,n}^*])}{\lmin(\E_{\P_n}[H_{\M,n}^*])} \;\le\; K(\tau, C).
$$

\end{list}

\end{lemma}

\begin{proof}
First, fix $n$ large enough, such that the bound of Condition~\ref{cond:X2}\eqref{cX2:Huber} holds, the pseudo parameter of Lemma~\ref{lemma:pseudo} exists and the bound of Lemma~\ref{lemma:pseudoBound} applies. Fix $\M\in\mathsf{M}_n$ and $\P_n\in\mathbf{P}_n^{(\mathrm{bin})}(\tau)$. Since $H_{\M,n}(y,\beta) =$ \linebreak $X_n[M]'D_{\M,n}(y,\beta)X_n[M]$, for a diagonal matrix $D_{\M,n}(y,\beta)$ whose $i$-th diagonal entry is given by 
$$
-y_i\ddot{\phi}_1(X_{i,n}[M]\beta) - (1-y_i)\ddot{\phi}_2(X_{i,n}[M]\beta)\;>\;0,
$$
in view of Conditions~\ref{cond:H}(\ref{cH:cdf},\ref{cH:concave},\ref{cH:C2}), we see that $H_{\M,n}^*(y)$ is positive definite. Moreover, from continuity and positivity of $-\ddot{\phi}_j$, $j=1,2$, Lemma~\ref{lemma:pseudoBound} and finiteness of $\mathcal H$, we conclude that the diagonal entries of $D_{\M,n}(y,\beta_{\M,n}^*(\P_n))$ are lower and upper bounded by positive constants that depend only on $\tau$ and $C$. This finishes the claim in \ref{l:Hcont:A}. Consider now
\begin{align}
&\sup_{\beta\in N_{\M,\P_n,n}(\delta)} \|H_{\M,n}^*(y)^{-1/2} H_{\M,n}(y,\beta) H_{\M,n}^*(y)^{-1/2} - I_{m(\M)}\| \notag\\
&\quad\le\; 
\|X_n[M]H_{\M,n}^*(y)^{-1/2}\|^2 \sup_{\beta\in N_{\M,\P_n,n}(\delta)} \|D_{\M,n}(y,\beta)  - D_{\M,n}(y,\beta_{\M,n}^*(\P_n)\| \notag\\
&\quad\le\;
\frac{\|X_n[M](X_n[M]'X_n[M])^{-1}X_n[M]'\|}{\min_{i=1,\dots, n}(-y_i\ddot{\phi}_1(X_{i,n}[M]\beta_{\M,n}^*) - (1-y_i)\ddot{\phi}_2(X_{i,n}[M]\beta_{\M,n}^*))} \times\notag\\
&\quad\quad\quad
\sup_{\beta\in N_{\M,\P_n,n}(\delta)}\max_{\substack{i=1,\dots, n\\j=1,2}} |\ddot \phi_j(X_{i,n}[M]\beta) - \ddot \phi_j(X_{i,n}[M]\beta_{\M,n}^*)|. \label{eq:Hcontinuity}
\end{align}
We have just seen that the minimum on the far right side of the previous display is lower bounded by a positive constant that depends only on $\tau$ and $C$. To finish the proof, note that for $\beta\in N_{\M,\P_n,n}(\delta)$, we have 
\begin{align*}
&|X_{i,n}[M]\beta - X_{i,n}[M]\beta_{\M,n}^*| \\ &\le \|X_{i,n}[M](X_n[M]'X_n[M])^{-1/2}\| \|(X_n[M]'X_n[M])^{1/2}(\beta - \beta_{\M,n}^*)\|\le \delta \sqrt{C/n}
\end{align*} 
and that $|X_{i,n}[M]\beta_{\M,n}^*|\le K^*(\tau,C)$, by Lemma~\ref{lemma:pseudoBound}. Therefore, by uniform continuity of $\ddot{\phi}_j$ on the compact interval $[-K^*(\tau,C) - \delta \sqrt{C}, K^*(\tau,C) + \delta \sqrt{C}]$, for every $\eta>0$, there exists $n_0 = n_0(\eta, \tau, \delta, C,h)$, such that the supremum in \eqref{eq:Hcontinuity} is bounded by $\eta$, for all $n\ge n_0$. Since $\mathcal H$ is finite, the proof of \ref{l:Hcont:B} is finished. 
For part~\ref{l:Hcont:C}, simply combine part~\ref{l:Hcont:A} and Condition~\ref{cond:X2}\eqref{cX2:eigen}.
\end{proof}


\subsection{Proof of Lemma~\ref{lemma:unifCons}}
\label{sec:unifCons}

The proof is a variation of the consistency part of the proof of Theorem~4 in \citet{Fahrmeir90}. 
Fix $n$ large enough, such that the bound of Condition~\ref{cond:X2}\eqref{cX2:Huber} holds, the pseudo parameter of Lemma~\ref{lemma:pseudo} exists and the bound of Lemma~\ref{lemma:pseudoBound} applies. Fix $\M\in\mathsf{M}_n$ and $\P_n\in\mathbf{P}_n^{(\mathrm{bin})}(\tau)$, write $\beta_{\M,n}^* = \beta_{\M,n}^*(\P_n)$ and note that by Condition~\ref{cond:H}\eqref{cH:C2}, the function $\beta\mapsto \ell_{\M,n}(y,\beta)$ is twice continuously differentiable on $\R^{m(\M)}$ and thus, for every $y\in\{0,1\}^n$, admits the expansion
\begin{align*}
\ell_{\M,n}(y,\beta) \quad&=\quad \ell_{\M,n}(y,\beta_{\M,n}^*) \;+\; (\beta-\beta_{\M,n}^*)'U_{M,n}^{1/2}U_{M,n}^{-1/2}s_{\M,n}^*(y) \\
&\quad-\; \frac{1}{2}(\beta-\beta_{\M,n}^*)'U_{M,n}^{1/2}U_{M,n}^{-1/2}H_{\M,n}(y,\tilde{\beta}_n)U_{M,n}^{-1/2}U_{M,n}^{1/2}(\beta-\beta_{\M,n}^*), 
\end{align*}
for some $\tilde{\beta}_n\in\{a\beta + (1-a)\beta_{\M,n}^* : a\in[0,1]\}$, and where $s_{\M,n}^*(y) = \frac{\partial \ell_{\M,n}(y,\beta)}{\partial \beta}\Big|_{\beta=\beta_{\M,n}^*}$, $H_{\M,n}(y,\beta) = -\frac{\partial \ell_{\M,n}(y,\beta)}{\partial\beta\partial\beta'}$, and $U_{M,n} = X_n[M]'X_n[M]/n$. 
For $\delta>0$, define $\lambda_n = \sqrt{n}U_{M,n}^{1/2}(\beta-\beta_{\M.n}^*)/\delta$ to rewrite the previous equation as
\begin{align*}
\ell_{\M,n}(y,\beta) \;-\; \ell_{\M,n}(y,\beta_{\M,n}^*) \;&=\; \delta \lambda_n'U_{M,n}^{-1/2}s_{\M,n}^*(y)/\sqrt{n} \\
&\quad-\; \frac{1}{2}\delta^2 \lambda_n'U_{M,n}^{-1/2}(H_{\M,n}(y,\tilde{\beta}_n)/n)U_{M,n}^{-1/2}\lambda_n,
\end{align*}
for all $\beta\in\R^{m(\M)}$ and all $y\in\{0,1\}^n$. For $N_n(\delta) := N_{\M,\P_n,n}(\delta)$ as in Lemma~\ref{lemma:Hcontinuity}\ref{l:Hcont:B}, define 
$
L_n(y,\delta) := \inf_{\beta\in N_n(\delta)} \lmin\left(U_{M,n}^{-1/2}(H_{\M,n}(y,\beta)/n)U_{M,n}^{-1/2}\right), 
$
take $\beta\in\partial N_n(\delta) = \{\beta\in\R^{m(\M)} : \|\sqrt{n}U_{M,n}^{1/2}(\beta-\beta_{\M,n}^*)\|= \delta\}$ and observe that now $\|\lambda_n\|=1$, $\delta\lambda_n'U_{M,n}^{-1/2}s_{\M,n}^*(y)/\sqrt{n} \le \delta \|U_{M,n}^{-1/2}s_{\M,n}^*(y)/\sqrt{n}\|$ and 
\begin{align*}
\frac{1}{2}\delta^2 L_n(y,\delta) \;&\le\; \frac{1}{2}\delta^2 \lmin\left(U_{M,n}^{-1/2}(H_{\M,n}(y,\tilde{\beta})/n)U_{M,n}^{-1/2}\right) \\
\;&\le\; \frac{1}{2}\delta^2 \lambda_n'U_{M,n}^{-1/2}(H_{\M,n}(y,\tilde{\beta})/n)U_{M,n}^{-1/2}\lambda_n,
\end{align*}
for all $y\in\{0,1\}^n$.
Therefore, we have the inclusion 
\begin{align*}
E_{\M,\P_n,n}(\delta) &:= \{y\in\{0,1\}^n:\delta\|U_{M,n}^{-1/2}s_{\M,n}^*(y)/\sqrt{n}\| < \delta^2 L_n(y,\delta)/2\} \\
&\subseteq \{y\in\{0,1\}^n : \forall \beta\in\partial N_n(\delta) : \ell_{\M,n}(y,\beta) < \ell_{\M,n}(y,\beta_{\M,n}^*)\} =: F_n(\delta).
\end{align*}
As a consequence, for every $y \in E_{\M,\P_n,n}(\delta)$, the function $\beta\mapsto \ell_{\M,n}(y,\beta)$ has a local maximum $\hat{\beta}_{\M,n}(y)$ on the interior of $N_n(\delta)$. By strict concavity (Conditions~\ref{cond:H}\eqref{cH:concave} and \ref{cond:X2}\eqref{cX2:rank}), this is a unique global maximum. 
Moreover, we have $F_n(\delta) \subseteq \{y: \hat{\beta}_{\M,n}(y) \in N_n(\delta)\} = \{y: \|\sqrt{n}U_{M,n}^{1/2}(\hat{\beta}_{\M,n}(y) - \beta_{\M,n}^*)\| \le \delta\}$. Hence, 
$$
\P_n(\|\sqrt{n}U_{M,n}^{1/2}(\hat{\beta}_{\M,n}(y) - \beta_{\M,n}^*)\| > \delta)
\;\le\;
\P_n(E_{\M,\P_n,n}(\delta)^c).
$$
It remains to verify that $\P_n(E_{\M,\P_n,n}(\delta)^c)$ is small for large $n$, uniformly in $\M$ and $\P_n$. Take $\eps>0$ and note that
\begin{align}
\P_n(E_{\M,\P_n,n}(\delta)^c) \;&\le\; \notag
\P_n(4\|U_{M,n}^{-1/2}s_{\M,n}^*(\cdot)/\sqrt{n}\|^2\ge \delta^2 L_n^2(\cdot,\delta), L_n^2(\cdot,\delta) \ge \eps^2 ) \notag\\
&\;\quad\;+\; \P_n(L_n^2(\cdot,\delta) < \eps^2 ) \label{eq:lUnifCons:trace}\\
&\le \;
\P_n(4\|U_{M,n}^{-1/2}s_{\M,n}^*(\cdot)/\sqrt{n}\|^2\ge \delta^2 \eps^2) \;+\; \P_n(L_n^2(\cdot,\delta) < \eps^2 )\notag\\
&\le \;
4\frac{\trace(U_{M,n}^{-1/2}\VC_n(s_{\M,n}^*/\sqrt{n})U_{M,n}^{-1/2})}{\delta^2\eps^2} \;+\; \P_n(L_n(\cdot,\delta) < \eps ), \notag
\end{align}
in view of Markov's inequality and since $\E_n[s_{\M,n}^*] = 0$. Note that $\VC_n(s_{\M,n}^*) = X_n[M]'V_{\M,\P_n,n}^*X_n[M]$, for a diagonal matrix $V_{\M,\P_n,n}^*$ whose diagonal entries satisfy 
\begin{align*}
[V_{\M,\P_n,n}^*]_{ii} \;&=\; \P_{i,n}(\{1\})\P_{i,n}(\{0\})(\dot{\phi}_1(X_{i,n}\beta_{\M,n}^*) - \dot{\phi}_2(X_{i,n}\beta_{\M,n}^*))^2 \\
\;&\le\; \sup_{|\gamma|\le K^*(\tau,C)}(\dot{\phi}_1(\gamma) - \dot{\phi}_2(\gamma))^2,
\end{align*}
for the constant $K^*(\tau,C)$ of Lemma~\ref{lemma:pseudoBound}.
Thus, the trace on the last line of display~\eqref{eq:lUnifCons:trace} is bounded by $p\max_{h\in\mathcal H}\sup_{|\gamma|\le K^*(\tau,C)}(\dot{\phi}_1(\gamma) - \dot{\phi}_2(\gamma))^2$, which does not depend on $n$, $\M$ or $\P_n$. Finally, to bound the remaining probability, note that $L_n(y,\delta)$ is lower bounded by the product of $\lmin(U_{M,n}^{-1/2}H_{\M,n}^*(y)U_{M,n}^{-1/2})/n$ and  $\inf_{\beta\in N_n(\delta)}  \lmin(H_{\M,n}^*(y)^{-1/2}H_{\M,n}(y,\beta)H_{\M,n}^*(y)^{-1/2})$. The first factor is itself lower bounded by the positive constant $\underline{K}(\tau, C)$ from Lemma~\ref{lemma:Hcontinuity}\ref{l:Hcont:A}. Thus, $\P_n(L_n(\cdot,\delta) < \eps )$ is upper bounded by
$$
\P_n\left( 
\sup_{\beta\in N_{n}(\delta)} \left\| H_{\M,n}^*(y)^{-1/2}H_{\M,n}(y,\beta)H_{\M,n}^*(y)^{-1/2} - I_{m(\M)} \right\| > 1- \frac{\eps}{\underline{K}(\tau, C)}
\right).
$$
Choosing $\eps = \underline{K}(\tau, C)/2$ and using Lemma~\ref{lemma:Hcontinuity}\ref{l:Hcont:B}, we conclude that for every $\delta>0$,
\begin{align}
&\sup_{\M\in\mathsf{M}_n}\sup_{\P_n\in\mathbf{P}_n^{(\mathrm{bin})}(\tau)} 
\P_n\left( 
	\left\| (X_n[M]'X_n[M])^{1/2}(\hat{\beta}_{\M,n} - \beta^*_{\M,n}(\P_n))\right\| > \delta\right)\label{eq:lUnifCons:sup}\\
&\quad\le	
\sup_{\M\in\mathsf{M}_n}\sup_{\P_n\in\mathbf{P}_n^{(\mathrm{bin})}(\tau)} 
\P_n(E_{\M,\P_n,n}(\delta)^c)\\	
&\quad\le 
16\frac{p\max_{h\in\mathcal H}\sup_{|\gamma|\le K^*(\tau,C)}(\dot{\phi}_1(\gamma) - \dot{\phi}_2(\gamma))^2}{\delta^2\underline{K}(\tau,C)^2}	
+
o(1),
\end{align}
where the $o(1)$ term refers to convergence as $n\to\infty$.
Now, to establish the asymptotic existence of the MLE, we simply take $E_{\M,\P_n,n} :=E_{\M,\P_n,n}(\delta_n)$, for $\delta_n\to\infty$ sufficiently slowly as $n\to\infty$. For the uniform consistency part, note that the limit superior as $n\to\infty$ of the expression in \eqref{eq:lUnifCons:sup} is bounded by a quantity that converges to zero as $\delta\to\infty$. \hfill{\qed}


\subsection{Proof of Theorem~\ref{thm:binary}}

Fix $n\in\N$, a candidate model $\M\in\mathsf{M}_n$ and $\P_n \in \mathbf{P}_n^{(\mathrm{bin})}(\tau)$, and let $\E_n$ and $\V_n$ denote the expectation and variance-covariance operators with respect to $\P_n$ on $\{0,1\}^n$. Define $s_{\M,n}(y,\beta) := \partial \ell_{\M,n}(y,\beta)/\partial\beta$ and note that by assumption $\beta\mapsto s_{\M,n}(y,\beta)$ is continuously differentiable on $\R^{m(\M)}$, for all $y\in\{0,1\}^n$.
Therefore, we can expand $s_{\M,n}$ around $\beta_0\in\R^{m(\M)}$ as follows,
\begin{align*}
s_{\M,n}(y,\beta) - s_{\M,n}(y,\beta_0) =\; \int\limits_0^1 H_{\M,n}(y,t\beta + (1-t) \beta_0) dt\cdot (\beta_0-\beta).
\end{align*}
For $n$ sufficiently large, such that $\beta_{\M,n}^* = \beta_{\M,n}^*(\P_n)$ of Lemma~\ref{lemma:pseudo} exists, define $\tilde{H}_{\M,n}(y) := \int_0^1 H_{\M,n}(y,t\beta_{\M,n}^* + (1-t)\hat{\beta}_{\M,n}(y))\,dt$ and note that with this we have 
\begin{equation}\label{eq:thm:binary:sstar}
s_{\M,n}^*(y) \;:=\; s_{\M,n}(y,\beta_{\M,n}^*) \;=\; \tilde{H}_{\M,n}(y)\cdot(\hat{\beta}_{\M,n}(y) - \beta_{\M,n}^*),
\end{equation}
for $y\in E_{\M,n} := E_{\M,\P_n,n}$, the set defined in Lemma~\ref{lemma:unifCons}. Moreover, since $H_{\M,n}(y,\beta)$ is positive definite under Conditions~\ref{cond:H}(\ref{cH:C2}) and \ref{cond:X2}\eqref{cX2:rank}, so are $\tilde{H}_{\M,n}(y)$ and $H_{\M,n}^*(y) := H_{\M,n}(y,\beta_{\M,n}^*)$. Thus, if we set 
$$
\psi_{i,n,\M}^*(y_i) \;:=\; y_i\dot{\phi}_1(X_{i,n}[M]\beta_{\M,n}^*) + (1-y_i)\dot{\phi}_2(X_{i,n}[M]\beta_{\M,n}^*),
$$
$$
g_{i,n,\M}(y_i) \;:=\; \E_n[H_{\M,n}^*]^{-1} X_{i,n}[M]' \left( \psi_{i,n,\M}^*(y_i) - \E_n[\psi_{i,n,\M}^*] \right),
$$
and
$$
\Delta_{n,\M}(y) \;:=\; \hat{\beta}_{\M,n}(y) - \beta_{\M,n}^* - \sum_{i=1}^n g_{i,n,\M}(y_i),
$$
we see that \eqref{eq:repres} is satisfied, that $\E_n[g_{i,n,\M}] = 0$, $r_{n,\M}(y) := \sum_{i=1}^n g_{i,n,\M}(y_i) = \E_n[H_{\M,n}^*]^{-1} s_{\M,n}^*(y)$, because $\E_n[s_{\M,n}^*] = 0$, and that $\V_n[r_{n,\M}^{(j)}]$ is given by
\begin{align*}
\sum_{i=1}^n \left(e_{m(\M)}(j)'\E_n[H_{\M,n}^*]^{-1}X_{i,n}[M]'\right)^2 \V_n[\psi_{i,n,\M}^*],
\end{align*}
where $e_m(j)$ is the $j$-th element of the canonical basis in $\R^m$ and $j\in\{1,\dots, m(\M)\}$. Note that by Lemma~\ref{lemma:pseudoBound}, Conditions~\ref{cond:H}(\ref{cH:cdf},\ref{cH:C2},\ref{cH:Hdot}) and the finiteness of $\mathcal H$, there exists a positive constant $\underline{K'}(\tau, C)$, depending only on $\tau>0$ and $C$ from Condition~\ref{cond:X2}, such that for all large $n$,
\begin{align*}
\infty > \V_n[\psi_{i,n,\M}^*]
&\ge \tau \left(\frac{\dot{h}(X_{i,n}[M]\beta_{\M,n}^*)}{h(X_{i,n}[M]\beta_{\M,n}^*)} + \frac{\dot{h}(X_{i,n}[M]\beta_{\M,n}^*)}{1-h(X_{i,n}[M]\beta_{\M,n}^*)}\right)^2\\
&\ge \underline{K'}(\tau, C) > 0.
\end{align*}
In particular, for such $n$, we have $0 < \V_n[r_{n,\M}^{(j)}] < \infty$. Furthermore, by a similar argument, we obtain the upper bound $|\psi_{i,n,\M}^*(y_i)|^2\le \overline{K'}(\tau, C)$ and, in turn,
\begin{align}\label{eq:thm:binary:bound}
\frac{|g_{i,n,\M}^{(j)}(y_i)|^2}{\V_n[r_{n,\M}^{(j)}]} \le 
2\frac{\overline{K'}(\tau, C)}{\underline{K'}(\tau, C)} 
	\frac{ \left( e_{m(\M)}(j)'\E_n[H_{\M,n}^*]^{-1}X_{i,n}[M]'\right)^2}{\sum_{i=1}^n \left(e_{m(\M)}(j)'\E_n[H_{\M,n}^*]^{-1}X_{i,n}[M]'\right)^2}.
\end{align}
But the numerator of the second fraction on the right of the previous display can be bounded by 
\begin{align*}
&\|(X_n[M]'X_n[M])^{1/2}\E_n[H_{\M,n}^*]^{-1}e_{m(\M)}(j)\|^2  \|(X_n[M]'X_n[M])^{-1/2}X_{i,n}[M]'\|^2\\
&\quad\le
\|X_n[M]\E_n[H_{\M,n}^*]^{-1}e_{m(\M)}(j)\|^2 \cdot C/n,
\end{align*}
in view of Condition~\ref{cond:X2}\eqref{cX2:Huber}, whereas the denominator of that same fraction coincides with $\|X_n[M]\E_n[H_{\M,n}^*]^{-1}e_{m(\M)}(j)\|^2$. Thus, we conclude that also \eqref{eq:linde} is satisfied. Finally, for asymptotic negligibility of $\V_n[r_{n,\M}^{(j)}]^{-1/2}\Delta_{n,\M}^{(j)}(y)$, first note that using \eqref{eq:thm:binary:sstar}, for $y\in E_{\M,n}$, we have
\begin{align*}
&\Delta_{n,\M}(y) = \tilde{H}_{\M,n}(y)^{-1}\tilde{H}_{\M,n}(y)(\hat{\beta}_{\M,n}(y) - \beta_{\M,n}^*) - \E_n[H_{\M,n}^*]^{-1}s_{\M,n}^*(y) \\
&=
\left(\tilde{H}_{\M,n}(y)^{-1}\E_n[H_{\M,n}^*] - I_{m(\M)} \right)\E_n[H_{\M,n}^*]^{-1}s_{\M,n}^*(y)\\
&= \E_n[H_{\M,n}^*]^{-1/2}\left(\E_n[H_{\M,n}^*]^{1/2}\tilde{H}_{\M,n}(y)^{-1}\E_n[H_{\M,n}^*]^{1/2} - I_{m(\M)} \right) \E_n[H_{\M,n}^*]^{1/2}\\
&\quad\times R_{n,\M}^{1/2} R_{n,\M}^{-1/2}r_{n,\M}(y),
\end{align*}
where $R_{n,\M} = \diag(\V_n[r_{n,\M}^{(1)}], \dots, \V_n[r_{n,\M}^{(m(\M))}])$. Therefore, 
\begin{align*}
&\|R_{n,\M}^{-1/2}\Delta_{n,\M}(y)\| \\
&\quad\le 
\|\E_n[H_{\M,n}^*]^{1/2}\tilde{H}_{\M,n}(y)^{-1}\E_n[H_{\M,n}^*]^{1/2} - I_{m(\M)}\| \frac{\lmax(\E_n[H_{\M,n}^*]^{1/2})}{\lmin(\E_n[H_{\M,n}^*]^{1/2})}\\
&\quad\quad\times\frac{\max_j \V_n^{1/2}[r_{n,\M}^{(j)}]}{\min_j \V_n^{1/2}[r_{n,\M}^{(j)}]} \|R_{n,\M}^{-1/2}r_{n,\M}(y)\|.
\end{align*}
Here, $R_{n,\M}^{-1/2}r_{n,\M}(y)$ has mean zero and covariance matrix with ones on the main diagonal, and consequently its norm is bounded in probability. The ratio of the largest and smallest variance component of $r_{n,\M}$ is bounded by the condition number of the matrix $\E_n[H_{\M,n}^*]^{-1}X_n[M]'X_n[M]\E_n[H_{\M,n}^*]^{-1}$ times a constant that depends only on $\tau$ and $C$, because of the previously derived upper and lower bounds on $|\psi_{i,n,\M}^*(y_i)|^2$ and $\V_n[\psi_{i,n,\M}^*]$, respectively. But this condition number is eventually bounded by a finite constant that depends only on $\tau$ and $C$ from Condition~\ref{cond:X2}, in view of Lemma~\ref{lemma:Hcontinuity}\ref{l:Hcont:C}. In particular, this lemma shows that the condition number of $\E_n[H_{\M,n}^*]$ is bounded, eventually. 
Therefore, since $\P_n(E_{\M,n})\to 1$, as $n\to\infty$, it remains to show that $\E_n[H_{\M,n}^*]^{1/2}\tilde{H}_{\M,n}(y)^{-1}\E_n[H_{\M,n}^*]^{1/2}  \to I_{m(\M)}$, in $\P_n$-probability. The result follows if we can show that the eigenvalues of $H_{\M,n}^*(y)^{-1/2}\tilde{H}_{\M,n}(y)H_{\M,n}^*(y)^{-1/2}$ and of $\E_n[H_{\M,n}^*]^{-1/2}H_{\M,n}^*(y)\E_n[H_{\M,n}^*]^{-1/2}$ converge to $1$, in $\P_n$-probability, because for $A=\E_n[H_{\M,n}^*]$, $B=\tilde{H}_{\M,n}(y)$ and $C=H_{\M,n}^*(y)$, we have
\begin{align*}
&\| A^{1/2}B^{-1}A^{1/2}  - I_{m(\M)}\| \\
&=
\| A^{1/2}C^{-1/2}C^{1/2}B^{-1}C^{1/2}C^{-1/2}A^{1/2}  - I_{m(\M)}\| \\
&\le
\| A^{1/2}C^{-1}A^{1/2}\| \|C^{1/2}B^{-1}C^{1/2} - I_{m(\M)}\| + \|A^{1/2}C^{-1}A^{1/2}  - I_{m(\M)}\|.
\end{align*}
For the first of these two, let $N_{\M,n}(\delta) := N_{\M,\P_n,n}(\delta)$ be as in Lemma~\ref{lemma:Hcontinuity}\ref{l:Hcont:B}, which is a convex set, and note that for every $\eps>0$ and every $\delta>0$, 
\begin{align*}
&\P_n\left( \|H_{\M,n}^*(y)^{-1/2}\tilde{H}_{\M,n}(y)H_{\M,n}^*(y)^{-1/2} - I_{m(\M)}\|>\eps\right)\\
&\quad\le
\P_n\Big( \sup_{\beta\in N_{\M,n}(\delta)}\|H_{\M,n}^*(y)^{-1/2} H_{\M,n}(y,\beta)H_{\M,n}^*(y)^{-1/2} - I_{m(\M)}\|>\eps,\\
&\hspace{2cm}  \hat{\beta}_{\M,n}\in N_{\M,n}(\delta)\Big)
+
\P_n\left( \hat{\beta}_{\M,n}\notin N_{\M,n}(\delta)\right)\\
&\quad=
o(1) + \P_n\left( \sqrt{n}\|U_{M,n}^{1/2}(\hat{\beta}_{\M,n} - \beta_{\M,n}^*)\|>\delta\right),
\end{align*}
where $U_{M,n} = X_n[M]'X_n[M]/n$. Since $\delta>0$ was arbitrary Lemma~\ref{lemma:unifCons} shows that the probability on the far left-hand-side of the previous display converges to zero as $n\to\infty$. Finally, for $v_1, v_2\in\R^{m(\M)}$ with $\|v_1\|=\|v_2\|=1$, write
\begin{align*}
&v_1'(\E_n[H_{\M,n}^*]^{-1/2}H_{\M,n}^*(y)\E_n[H_{\M,n}^*]^{-1/2} - I_{m(\M)})v_2\\
&= \sum_{i=1}^n v_1'\E_n[H_{\M,n}^*]^{-1/2}X_{i,n}[M]'X_{i,n}[M] \E_n[H_{\M,n}^*]^{-1/2}v_2 \\
&\hspace{2cm} \times \left(D_{i,\M,n}^*(y_i) - \E_n[D_{i,\M,n}^*] \right),
\end{align*}
where $D_{i,\M,n}^*(y_i) = -y_i\ddot{\phi}_1(X_{i,n}[M]\beta_{\M,n}^*) - (1-y_i)\ddot{\phi}_2(X_{i,n}[M]\beta_{\M,n}^*)$ and \linebreak $\V_n[D_{i,\M,n}^*]$ is bounded by a constant that depends only on $\tau$ and $C$. The mean of the expression in the previous display is clearly equal to zero, while its variance is bounded by $\sum_{i=1}^n (X_{i,n}[M] \E_n[H_{\M,n}^*]^{-1}X_{i,n}[M]')^2$ times a constant that depends only on $\tau$ and $C$. In view of Lemma~\ref{lemma:Hcontinuity}\ref{l:Hcont:A} and Condition~\ref{cond:X2}\eqref{cX2:Huber}, the latter sum is itself bounded by 
$$\sum_{i=1}^n (X_{i,n}[M] (X_n[M]'X_n[M])^{-1}X_{i,n}[M]')^2 \le n C^2 /n^2 \to 0,$$ where we have omitted another constant that depends only on $\tau$ and $C$. We have thus verified Condition~\ref{cond:sum}.

To show that the proposed estimators $\hat{\sigma}_{j,\M,n}^2$ of \eqref{eq:Sbinary} consistently overestimate the asymptotic variances of the MLE, we verify the assumptions of Proposition~\ref{prop:overest} with 
$$
\tilde{g}_{i,n,\M}(y) = \hat{H}_{\M,n}(y)^{-1}X_{i,n}[M]'\hat{\psi}_{i,n,\M}(y),
$$ 
and $a_{i,n,\M} = \E_n[H_{\M,n}^*]^{-1}X_{i,n}[M]'\E_n[\psi_{i,n,\M}^*]$,
where $\hat{H}_{\M,n}(y) = H_{\M,n}(y,\hat{\beta}_{\M,n}(y))$ and $\hat{\psi}_{i,n,\M}(y) = y_i\dot{\phi}_1(X_{i,n}[M]\hat{\beta}_{\M,n}(y)) + (1-y_i)\dot{\phi}_2(X_{i,n}[M]\hat{\beta}_{\M,n}(y))$. In particular, $\hat{g}_{i,n,\M}(y) = \tilde{g}_{i,n,\M}(y) - a_{i,n,\M}$. First note that for any $\gamma\in\R$,
$$
y_i \dot{\phi}_1(\gamma) + (1-y_i)\dot{\phi}_2(\gamma) = 
\frac{\dot{h}(\gamma)}{h(\gamma)(1-h(\gamma))}(y_i-h(\gamma)),
$$ 
so that $\hat{\psi}_{i,n,\M}(y) = \hat{u}_{i,\M}(y)$ and the diagonal entries of $\tilde{S}_{\M,n}$ can, indeed, be represented as
$$
\hat{\sigma}_{j,\M,n}^2(y) = \sum_{i=1}^n\left[ \tilde{g}_{i,n,\M}^{(j)}(y)\right]^2,
$$
for $j=1, \dots, m(\M)$, as required for the application of Proposition~\ref{prop:overest}. Next, consider 
\begin{align*}
|g_{i,n,\M}^{(j)}(y_i) - &\hat{g}^{(j)}_{i,n,\M}(y)|^2
=
|e_{m(\M)}(j)'\Big(\E_n[H_{\M,n}^*]^{-1} X_{i,n}[M]' \psi_{i,n,\M}^*(y_i) \\
&\quad- \hat{H}_{\M,n}(y)^{-1}X_{i,n}[M]'\hat{\psi}_{i,n,\M}(y)\Big)|^2\\
&\le
2|e_{m(\M)}(j)'(\E_n[H_{\M,n}^*]^{-1} - \hat{H}_{\M,n}(y)^{-1})X_{i,n}[M]' \psi_{i,n,\M}^*(y_i)|^2\\
&\quad+
2|e_{m(\M)}(j)'\hat{H}_{\M,n}(y)^{-1}X_{i,n}[M]' (\psi_{i,n,\M}^*(y_i) - \hat{\psi}_{i,n,\M}(y))|^2\\
&\le
2|e_{m(\M)}(j)'(\E_n[H_{\M,n}^*]^{-1} - \hat{H}_{\M,n}(y)^{-1})X_{i,n}[M]'|^2 \overline{K'}(\tau, C)\\
&\quad+
2|e_{m(\M)}(j)'\hat{H}_{\M,n}(y)^{-1}X_{i,n}[M]' (\psi_{i,n,\M}^*(y_i) - \hat{\psi}_{i,n,\M}(y))|^2.
\end{align*}
We want to show that
$$
\frac{\sum_{i=1}^n \left[ g_{i,n,\M}^{(j)}(y_i) - \hat{g}_{i,n,\M}^{(j)}(y) \right]^2}{\V_n[r_{n,\M}^{(j)}]}
\le
\frac{\sum_{i=1}^n \left[ g_{i,n,\M}^{(j)}(y_i) - \hat{g}_{i,n,\M}^{(j)}(y) \right]^2}{\underline{K'}(\tau,C) \|X_n[M]\E_n[H_{\M,n}^*]^{-1}e_{m(\M)}(j) \|^2},
$$
converges to zero in $\P_n$-probability, as in \eqref{eq:negldiff2}, where the inequality follows by the same argument as in \eqref{eq:thm:binary:bound}. Hence, it suffices to show that
\begin{equation}\label{eq:thm:binary:conv1}
\frac{\|X_n[M](\E_n[H_{\M,n}^*]^{-1} - \hat{H}_{\M,n}(y)^{-1})e_{m(\M)}(j)\|^2 }{\|X_n[M]\E_n[H_{\M,n}^*]^{-1}e_{m(\M)}(j) \|^2}
\end{equation}
and
\begin{equation}\label{eq:thm:binary:conv2}
\frac{\|X_n[M]\hat{H}_{\M,n}(y)^{-1}e_{m(\M)}(j) \|^2}{\|X_n[M]\E_n[H_{\M,n}^*]^{-1}e_{m(\M)}(j) \|^2}\max_{i=1,\dots, n} (\psi_{i,n,\M}^*(y_i) - \hat{\psi}_{i,n,\M}(y))^2,
\end{equation}
both converge to zero in $\P_n$-probability. For \eqref{eq:thm:binary:conv1}, simply note that this expression is bounded by
$$
\|I_{m(\M)} - \E_n[H_{\M,n}^*]\hat{H}_{\M,n}(y)^{-1}\|^2
$$
times the  condition number of the matrix $\E_n[H_{\M,n}^*]^{-1}X_n[M]'X_n[M]\E_n[H_{\M,n}^*]^{-1}$. We have already seen above that the latter is bounded by a constant that depends only on $\tau$ and $C$. To see that $\E_n[H_{\M,n}^*]\hat{H}_{\M,n}(y)^{-1}$ converges to $I_{m(\M)}$ in $\P_n$-probability, note that 
\begin{align*}
&\|I_{m(\M)} - \E_n[H_{\M,n}^*]\hat{H}_{\M,n}(y)^{-1}\|^2\\
&\quad\le
\|I_{m(\M)} - \E_n[H_{\M,n}^*]^{1/2}\hat{H}_{\M,n}(y)^{-1}\E_n[H_{\M,n}^*]^{1/2}\|^2 \frac{\lmax(\E_n[H_{\M,n}^*])}{\lmin(\E_n[H_{\M,n}^*])}\\
&\quad\le
\|I_{m(\M)} - \E_n[H_{\M,n}^*]^{1/2}\hat{H}_{\M,n}(y)^{-1}\E_n[H_{\M,n}^*]^{1/2}\|^2 K(\tau, C),
\end{align*}
by Lemma~\ref{lemma:Hcontinuity}\ref{l:Hcont:C}. Furthermore, for every $\eps>0$ and $\delta>0$,
\begin{align*}
&\P_n\left( \|\E_n[H_{\M,n}^*]^{-1/2}\hat{H}_{\M,n}(y)\E_n[H_{\M,n}^*]^{-1/2} - I_{m(\M)}\|>\eps\right)\\
&\quad\le
\P_n\Big( \sup_{\beta\in N_{\M,n}(\delta)}\|\E_n[H_{\M,n}^*]^{-1/2}H_{\M,n}(y,\beta)\E_n[H_{\M,n}^*]^{-1/2} - I_{m(\M)}\|>\eps\Big)\\
&\hspace{1cm} +\P_n\Big(\hat{\beta}_{\M,n} \notin N_{\M,n}(\delta)\Big),
\end{align*}
and we have already seen before that this entails convergence to zero of the probability on the left-hand-side of the previous display. We conclude that \eqref{eq:thm:binary:conv1} does converge to zero in $\P_n$-probability. To establish the same convergence also for \eqref{eq:thm:binary:conv2}, first note that it follows from the previous arguments that the fraction in that display is bounded in $\P_n$-probability. Finally, we have to establish the desired convergence for the maximum in that display. But this follows from the continuity of $\dot{\phi}_1$ and $\dot{\phi}_2$ on $\R$, the bound on $|X_{i,n}[M]\beta_{\M,n}^*|$ from Lemma~\ref{lemma:pseudoBound} and the consistency of Lemma~\ref{lemma:unifCons}. 
Therefore, Proposition~\ref{prop:overest} shows that \eqref{eqn:special2} is satisfied. Note that $\VC_n(r_n)$ has rank no larger than $\min(k,n)$, where $r_n = (r_{n,\M})_{\M\in\mathsf{M}_n}$. Hence, Theorem~\ref{thm:cons}, together with Lemma~\ref{lem:upper}, finishes the proof.\hfill\qed

\subsection{Canonical link function}
\label{sec:binaryCanonical}

\begin{corollary}
\label{corr:binaryCanoncial}
In the setting of Theorem~\ref{thm:binary}, if $\mathcal H$ contains only the canonical link function $h^{(c)}(\gamma) = e^\gamma/(1+e^\gamma)$, then the confidence intervals
\begin{equation}
\mathrm{CI}_{1-\alpha, \mathbb{M}}^{(j), \mathrm{binC}} = \hat{\beta}^{(j)}_{\mathbb{M}, n} \pm  \sqrt{\hat{\sigma}^2_{j,\mathbb{M}, n}} B_{\alpha}(\min(k, p),k),
\end{equation}
satisfy
\begin{equation}
\liminf_{n \to \infty} \inf_{\mathbb{P}_n \in \mathbf{P}_n^{(\mathrm{bin})} 
\left(\tau \right) } \mathbb{P}_n \left(\beta_{\hat{\mathbb{M}}_n, n}^{*, (j)} \in \mathrm{CI}_{1-\alpha, \hat{\mathbb{M}}_n}^{(j), \mathrm{binC}} \;\forall j = 1, \hdots, m(\hat{\mathbb{M}}_n)  \right) \geq 1-\alpha.
\end{equation}
\end{corollary}

\begin{proof}
From the first few lines of the proof of Theorem~\ref{thm:binary}, we see that the $m(\M)$-dimensional sub-vector $r_{n,\M}(y)$ of $r_n(y)$ that corresponds to the model $\M\triangleq (h,M)\in\{h^{(c)}\}\times \mathcal I$, is given by
$$
r_{n,\M}(y) = \E_n[H_{\M,n}^*]^{-1} s_{\M,n}^*(y),
$$
where $H_{\M,n}^*(y) := H_{\M,n}(y,\beta_{\M,n}^*)$, $s_{\M,n}^*(y) := s_{\M,n}(y,\beta_{\M,n}^*)$ and $s_{\M,n}(y,\beta) := \partial \ell_{\M,n}(y,\beta)/\partial \beta = X_n[M]'C_{\M,n}(y,\beta)$, where $C_{\M,n}(y,\beta)$ is an $n\times 1$ vector with $i$-th entry given by $y_i\dot{\phi}_1(X_{i,n}[M]\beta) + (1-y_i)\dot{\phi}_2(X_{i,n}[M]\beta)$. But it is easy to see that for $h=h^{(c)}$, $\phi_1(\gamma) - \phi_2(\gamma) = \gamma$ and thus $\dot{\phi}_1(\gamma) - \dot{\phi}_2(\gamma) = 1$, so that in this case the matrix $\VC_n(r_n)$ reduces to
$$
\VC_n(r_n) = \left(\E_n[H^*_{\M_s,n}]^{-1} X_n[M_s]'\VC(Y_n)X_n[M_t]\E_n[H^*_{\M_t,n}]^{-1}\right)_{s,t=1}^d.
$$  
The rank of this matrix is not larger than $\min(k,p)$, so by Lemma~\ref{lem:upper} we obtain the smaller bound $K_{1-\alpha}(\corr(\VC_n(r_n)))\le B_\alpha(\min(k,p),k)$.
\end{proof}

\bibliographystyle{imsart-nameyear}
\bibliography{Biblio}

\end{document}